\documentclass[reqno]{amsart}
\usepackage{amsmath, amsfonts, amsthm, amssymb, setspace, textcomp, bbm, multirow, hyperref, mathtools}
\usepackage{geometry}
\newtheorem{theorem}{Theorem}[section]

\newtheorem{corollary}[theorem]{Corollary}
\newtheorem{proposition}[theorem]{Proposition}

\theoremstyle{definition}

\theoremstyle{remark}

\numberwithin{equation}{section}

\newcommand{\Parans}[1]{\left(#1\right)}

\newcommand{\PieceTwo}[4]
{
	\left\{
   	\begin{array}{ll}
      	#1 & #3 \\
       	#2 & #4
     	\end{array}
	\right.
}
\newcommand{\aqprod}[3]{\Parans{#1;#2}_{#3}}

\newcommand{\Jac}[2]{\left(\frac{#1}{#2}\right)}

\newcommand{\TwoTwoMatrix}[4]
{
	\left(\begin{array}{cc}
		#1 & #2\\#3 & #4
	\end{array}\right)
}
\newcommand{\SmallMatrix}[4]{\begin{psmallmatrix}#1 &#2\\#3 &#4\end{psmallmatrix}}

\newcommand{\Floor}[1]{\left\lfloor #1 \right\rfloor}
\newcommand{\Fractional}[1]{\left\{#1\right\} }

\newcommand{\z}[2]{\zeta_{#1}^{#2}}
\newcommand{\tmu}{\widetilde{\mu}}
\newcommand{\SLTwo}{\mbox{SL}_2(\mathbb{Z})}

\author{CHRIS JENNINGS-SHAFFER}
\address{Department of Mathematics, Oregon State University\\
Corvallis, Oregon 97331, USA
\endgraf jennichr@math.oregonstate.edu}



\title{The generating function of the $M_2$-rank of partitions without repeated odd parts as a mock modular form}

\allowdisplaybreaks
\begin{document}

\allowdisplaybreaks

\begin{abstract}
By work of Bringmann, Ono, and Rhoades it is known that the generating function of the
$M_2$-rank of partitions without repeated odd parts 
is the so-called holomorphic part of a certain harmonic Maass form.
Here we improve the standing of this function as
a harmonic Maass form and show more can be done with this function.
In particular we show the related harmonic Maass form
transforms like the generating function for partitions without repeated odd
parts (which is a modular form).
We then use these improvements to determine formulas for the
rank differences modulo $7$. Additionally we give identities and formulas that
allow one to determine formulas for the rank differences modulo $c$, for any
$c>2$.
\end{abstract}

\maketitle

\section{Introduction}
\allowdisplaybreaks

To begin we recall that a partition of an integer is a non-increasing 
sequence of positive
integers that sum to $n$. For example the partitions of $5$ are
$5$, $4+1$, $3+2$, $3+1+1$, $2+2+1$, $2+1+1+1$, and $1+1+1+1+1$. It is standard
to let $p(n)$ denote the number of partitions of $n$. From our example we see
that $p(5)=7$. Partitions have a rich history in number theory
and combinatorics, going as far back as Euler. Today much emphasis is put on
Ramanujan's work with partitions and his related work with $q$-series. 

One point of interest are the congruences of Ramanujan,
\begin{align*}
	p(5n+4)&\equiv 0\pmod{5},\\
	p(7n+5)&\equiv 0\pmod{7},\\
	p(11n+6)&\equiv 0\pmod{11}.
\end{align*}
Dyson in \cite{Dyson} proposed the following combinatorial explanation of the
first two congruences. Given a partition of $n$, we define the rank of the
partition to be the largest part minus the number of parts. If we group the 
partitions of $5n+4$ according to the modulo $5$ value of their rank, then
it turns out we have five sets of equal size. If we group the 
partitions of $7n+5$ according to the modulo $7$ value of their rank, then
it turns out we have seven sets of equal size. These two statements were later proved by 
Atkin and Swinnerton-Dyer in \cite{AS}.

Another point of interest are the mock theta functions, which Ramanujan introduced
in his famous last letter to Hardy. Mock theta functions are the topic of 
Watson's final address as president of the London Mathematical Society
\cite{Watson}. We save the more technical discussion for the 
next section, but we give two of Ramanujan's examples:
\begin{align*}
	f(q)
	&=
		1+\frac{q}{(1+q)^2}+\frac{q^4}{(1+q)^2(1+q^2)^2}+\dots
	,
	\\
	\phi(q)
	&=
		1+\frac{q}{(1+q^2)}+\frac{q^4}{(1+q^2)(1+q^4)}+\dots
	.
\end{align*}
One of Ramanujan's mock theta conjectures (which was proved by Watson in 
\cite{Watson}) is that
$2\phi(-q)-f(q) = \vartheta_4(0,q)\prod_{n=1}^\infty (1+q^n)^{-1}$,
where $\vartheta_4(0,q)$ is a Jacobi theta function.

These two areas are actually strongly connected. In particular, we let
$N(m,n)$ denote the number of partitions of $n$ with rank $m$, 
and write the generating function as
\begin{align*}
	R(\zeta;q)
	&=
	\sum_{n=0}^\infty\sum_{m=-\infty}^\infty N(m,n) \zeta^mq^n
	.
\end{align*}
Of the seven functions commonly referred to as the third order mock theta 
functions, all seven can be expressed in terms of $R(\zeta;q)$ by 
replacing $q$ by a power of $q$, letting $\zeta$ be a power of $q$ times a
root of unity, and possibly adjusting by a constant and a power of $q$.
It is worth pointing out that it was Watson \cite{Watson} who observed the third order 
functions could be written in terms of the rank, and he did so nearly
a decade before the rank function was formally defined.
The fifth and seventh order mock theta functions can also be obtained in such
a manner, if one also allows adding on certain infinite products. 
Explicit versions of these statements can be found in Section 5 
of \cite{HickersonMortenson}.
For this reason $R(\zeta;q)$ is called a universal mock theta 
function. With this we see a strong understanding of $R(\zeta;q)$ leads to a better
understanding of the mock theta functions. 

Another universal mock theta function is $R2(\zeta;q)$, the generating function of 
the $M_2$-rank of partitions
without repeated odd parts, which appears among the tenth order mock
theta functions.
This also has an elegant combinatorial statement.
The $M_2$-rank of a partition without repeated odd parts is given by
taking the ceiling of the largest part divided by 2, and then subtracting the 
number of parts. 
This rank was introduced by Berkovich and Garvan in \cite{BerkovichGarvan}
and further studied by Lovejoy and Osburn in \cite{LovejoyOsburn2}.
We let $N_2(m,n)$ denote the number of partitions of $n$ without repeated odd parts
and $M_2$-rank $m$. We denote the generating function by
\begin{align*}
	R2(\zeta;q)
	&=
	\sum_{n=0}^\infty\sum_{m=-\infty}^\infty N_2(m,n) \zeta^m q^n.
\end{align*}

In this article we improve the understanding of the $M_2$-rank as the
holomorphic part of a harmonic Maass form. We determine a 
more precise formula for the transformation of the associated harmonic Maass 
form and do so on a 
larger group than previously known. 
To aid in determining dissection
formulas $R2(e^{\frac{2\pi i}{\ell}};q)=\sum_{r=0}^{\ell-1}q^rA_r(q^\ell)$, we 
find additional harmonic Maass forms with holomorphic parts corresponding
to the mock modular parts of the $A_r(q^\ell)$. Additionally we give lower
bounds on the orders of the cusps for the holomorphic parts of the harmonic
Maass forms.  

Partition ranks have been 
studied in terms of harmonic Maass forms in a number of recent works, 
perhaps the most influential being \cite{BringmannOno} and other
important works are \cite{AhlgrenTreneer, BringmannOnoRhoades, Dewar, Mao}
of which the second and third deal with the $M_2$-rank.
Recently Garvan \cite{Garvan3} has made remarkable improvements to the 
transformations and levels from \cite{BringmannOno}
for the harmonic Maass form related to $R(\zeta;q)$.
In Garvan's work, one sees that on a rather large subgroup, the harmonic
Maass form related to $R(\zeta;q)$ has the same multiplier as $\eta(\tau)^{-1}$,
which is the generating function for partitions. The same phenomenon
will occur here. That is to say the harmonic Maass form related to $R2(\zeta;q)$
will transform like $\frac{\eta(2\tau)}{\eta(4\tau)\eta(\tau)}$, which is
the generating function for partitions without repeated odd parts.
This also occurs with the Dyson rank of
overpartitions, $\overline{R}(\zeta;q)$, which the author studied in \cite{JenningsShaffer2}.
Although not stated explicitly in that article, one can easily verify that on a certain 
subgroup
the transformation formula for the Dyson rank of overpartitions agrees with
$\frac{\eta(2\tau)}{\eta(\tau)^2}$, the generating function for overpartitions.
The general explanation for this is that usually one can write the generating 
function of a rank function as the product of the generating function for the 
underlying partition function and a higher level Appell function. By the work of 
Zwegers in \cite{Zwegers2}, Appell functions have modular completions and these modular
completions have trivial multipliers on large subgroups of $\SLTwo$.
With an improved understanding of the $M_2$-rank, we give new identities for $R2(\zeta;q)$ when
$\zeta$ is set equal to certain roots of unity.
In the next section we formally state our definitions and results, and then
give an outline of the rest of the article. 

For the reader that wishes to compare the various appearances of the the three
rank functions mentioned above, we note the following. In the notation of
Gordon and McIntosh \cite{GordonMcintosh}, we have 
$R(\zeta;q)=h_3(\zeta,q)=(1-\zeta)(1+\zeta g_3(\zeta,q))$,
$\overline{R}(\zeta;q)=(1-\zeta)(1-\zeta^{-1})g_2(\zeta,q)$, and
$R2(\zeta;q)=h_2(\zeta,-q)$.
In the notation of Hickerson and Mortenson \cite{HickersonMortenson}, we have
$R(\zeta;q)=(1-\zeta)(1+\zeta g(\zeta,q))$,
$\overline{R}(\zeta;q)=(1-\zeta)(1-\zeta^{-1})h(\zeta,q)$, and
$R2(\zeta;q)=(1-\zeta)\zeta^{\frac12}k(\zeta^{\frac12},-q)$.

\section{Statement of Results}

As in the introduction we let $N_2(m,n)$ denote the number of partitions of $n$ without repeated odd parts 
with $M_2$-rank $m$. Furthermore we let $N_2(k,c,n)$ denote the number of 
partitions of $n$ without repeated odd parts with $M_2$-rank congruent to $k$ modulo $c$. 
To state out definitions and results
we use the standard product notation,
\begin{align*}
	\aqprod{\zeta}{q}{n} 
		&= \prod_{j=0}^{n-1} (1-\zeta q^j)
	,
	&\aqprod{\zeta}{q}{\infty} 
		&= \prod_{j=0}^\infty (1-\zeta q^j)
	,\\
	\aqprod{\zeta_1,\dots,\zeta_k}{q}{n} 
		&= \aqprod{\zeta_1}{q}{n}\dots\aqprod{\zeta_k}{q}{n}
	,
	&\aqprod{\zeta_1,\dots,\zeta_k}{q}{\infty} 
		&= \aqprod{\zeta_1}{q}{\infty}\dots\aqprod{\zeta_k}{q}{\infty}
	.	
\end{align*}
Also we let $q=\exp(2\pi i\tau)$ for $\tau\in \mathcal{H}$, that is $\mbox{Im}(\tau)>0$.
For $c$ a positive integer and $a$ an integer, we let
$\z{c}{a}=\exp(2\pi ia/c)$.

As noted in \cite{LovejoyOsburn2}, one may use Theorem 1.2 of \cite{Lovejoy1}
to find that the generating function for
$N_2(m,n)$ is given by
\begin{align*}
	R2(\zeta;q)
	&=
	\sum_{n=0}^\infty \sum_{m=-\infty}^\infty N_2(m,n)\zeta^mq^n
	=
	\sum_{n=0}^\infty
	\frac{q^{n^2}\aqprod{-q}{q^2}{n}}{\aqprod{\zeta q^2,\zeta^{-1}q^2}{q^2}{n}}
	\\
	&=
	\frac{\aqprod{-q}{q^2}{\infty}}{\aqprod{q^2}{q^2}{\infty}}
	\left(1+
		\sum_{n=1}^\infty \frac{ (1-\zeta)(1-\zeta^{-1})(-1)^nq^{2n^2+n}(1+q^{2n}) }
			{(1-\zeta q^{2n})(1-\zeta^{-1}q^{2n})}
	\right)	
	.
\end{align*}
Next for $a$ and $c$ integers, $c>0$, and $c\nmid 2a$, we define
\begin{align*}
	\mathcal{R}2(a,c;\tau)  
	&=
		q^{-\frac{1}{8}}R2(\z{c}{a};q)
	,\\
	\widetilde{\mathcal{R}2}(a,c;\tau)
		&=
		\mathcal{R}2(a,c;\tau)	
		-
		i(1-\z{c}{a})\z{2c}{-a}
		\left(
			e^{\frac{\pi i}{4}}	
			\int_{-\overline{\tau}}^{i\infty}			
			\frac{g_{\frac{3}{4}, \frac{1}{2}-\frac{2a}{c}} (4w) dw }{\sqrt{-i(w+\tau)}}					
			+
			e^{-\frac{\pi i}{4}}	
			\int_{-\overline{\tau}}^{i\infty}			
			\frac{g_{\frac{1}{4}, \frac{1}{2}-\frac{2a}{c}} (4w) dw }{\sqrt{-i(w+\tau)}}		
		\right)
,
\end{align*}
where $g_{a,b}(\tau)$ is a theta function of Zwegers \cite{Zwegers}, the definition
of which we give in the next section. The factor of $q^{-\frac{1}{8}}$ is necessary
for certain transformations to work correctly. We notice that it is immediately 
clear that
$\mathcal{R}2(a+c,c;\tau)=\mathcal{R}2(a,c;\tau)$;
it is also true that
$\widetilde{\mathcal{R}2}(a+c,c;\tau)=\widetilde{\mathcal{R}2}(a,c;\tau)$
and this follows immediately from properties of $g_{a,b}(\tau)$.

We will find $R2$, $\mathcal{R}2$, and $\widetilde{\mathcal{R}2}$ are related
to the following functions.
For $k$ and $c$ integers, $0\le k<c$, we define
\begin{align*}
S(k,c;\tau)
&=
	\frac{ q^{-c^2+4kc-2k^2+c-k} }
		{\aqprod{ (-1)^{c-1}q^{c^2},(-1)^{c-1}q^{3c^2},q^{4c^2}}{q^{4c^2}}{\infty}}	
	\sum_{n=-\infty}^\infty
		\frac{(-1)^{nc} q^{2c^2n^2+3c^2n} }
			{1-q^{4c^2n + (1+4k-c)c}}		
,		
\\
\mathcal{S}(k,c;\tau)
&=
	i^{c}q^{-\frac{1}{8}}S(k,c;\tau)	
,\\
\widetilde{\mathcal{S}}(k,c;\tau)
&=
		\mathcal{S}(k,c;\tau)
		-
		(-1)^k i^{1+c}e^{-\frac{\pi i}{4}}c
		\int_{-\overline{\tau}}^{i\infty}
		\frac{  g_{\frac{1+4k}{4c},\frac{c}{2}} (4c^2w)  }{\sqrt{-i(w+\tau)}}
		dw
.
\end{align*}
We note that $S(k,c;\tau)$ has a pole when $1+4k-c\equiv0\pmod{4c}$,
however this case does not occur in our usage of $S(k,c;\tau)$.
We will see in Section 3 that we can define the functions $\mathcal{S}(k,c;\tau)$
and $\widetilde{\mathcal{S}}(k,c;\tau)$ for all integer values of $k$, however
some care must be taken in doing so as the above definitions 
would not be the proper choice as it is not true that
$\mathcal{S}(k+c,c;\tau)=\mathcal{S}(k,c;\tau)$.

Our main theorem is the following. This theorem relates the $M_2$-rank generating
function to harmonic Maass forms and modular forms.

\begin{theorem}\label{TheoremModularity}
Suppose $a$ and $c$ are integers, $c>0$, and $c\nmid 2a$. Let $t=\frac{c}{\gcd(4,c)}$.
\begin{enumerate}
	\item $\widetilde{\mathcal{R}2}(a,c;8\tau)$ is a harmonic Maass form of weight
	$1/2$ on $\Gamma_0(256t^2)\cap\Gamma_1(8t)$ and $\mathcal{R}2(a,c;8\tau)$ is a mock modular 
	form.

	\item The function
	\begin{align*}
		\widetilde{\mathcal{R}2}(a,c;\tau) 
		- 
		i^{-c}(1-\z{c}{a})\z{c}{-a}\sum_{k=0}^{c-1} (-1)^k (\z{c}{-2ak} -\z{c}{2ak+a} )
		\widetilde{\mathcal{S}}(k,c;\tau)
	\end{align*}
	is holomorphic in $\tau$ and has at worst poles at the cusps. Furthermore,
	\begin{align*}		
		&
		\widetilde{\mathcal{R}2}(a,c;\tau) 
		- 
		i^{-c}(1-\z{c}{a})\z{c}{-a}\sum_{k=0}^{c-1} (-1)^k (\z{c}{-2ak} -\z{c}{2ak+a} )
		\widetilde{\mathcal{S}}(k,c;\tau)
		\\
		&=
		\mathcal{R}2(a,c;\tau) 
		- 
		i^{-c}(1-\z{c}{a})\z{c}{-a}\sum_{k=0}^{c-1} (-1)^k (\z{c}{-2ak} -\z{c}{2ak+a} )
		\mathcal{S}(k,c;\tau)
	.	
	\end{align*}

	\item The function
	\begin{align*}
		\frac{\eta(4\tau)\eta(\tau)}{\eta(2\tau)}\left(
			\mathcal{R}2(a,c;\tau) 
			- 
			i^{-c}(1-\z{c}{a})\z{c}{-a}\sum_{k=0}^{c-1} (-1)^k (\z{c}{-2ak} -\z{c}{2ak+a} )
			\mathcal{S}(k,c;\tau)
		\right)
	\end{align*}
	is a weakly holomorphic modular form of weight $1$ on $\Gamma_0(4c^2\cdot\gcd(c,2))\cap\Gamma_1(4c)$.

	\item If $c$ is odd, then the function
	\begin{align*}
		\frac{\eta(4c^2\tau)\eta(c^2\tau)}{\eta(2c^2\tau)}\left(
			\mathcal{R}2(a,c;\tau) 
			- 
			i^{-c}(1-\z{c}{a})\z{c}{-a}\sum_{k=0}^{c-1} (-1)^k (\z{c}{-2ak} -\z{c}{2ak+a} )
			\mathcal{S}(k,c;\tau)
		\right)
	\end{align*}
	is a weakly holomorphic modular form of weight $1$ on $\Gamma_0(4c^2)\cap\Gamma_1(4c)$.
\end{enumerate}
In the sums above, when $1+4k-c\equiv 0\pmod{4c}$ the summands
$(-1)^k(\z{c}{-2ak} -\z{c}{2ak+a})\mathcal{S}(k,c;\tau)$
and $(-1)^k(\z{c}{-2ak} -\z{c}{2ak+a})\widetilde{\mathcal{S}}(k,c;\tau)$
are taken to be zero.
\end{theorem}
Part (1) of Theorem \ref{TheoremModularity} is not entirely new. In \cite{BringmannOnoRhoades}
Bringmann, Ono, and Rhoades studied various general harmonic Maass forms, one of
which has essentially $R2(\z{c}{a}, 8t^2 \tau)$ as the holomorphic part,
and the subgroup is $\Gamma_1(2^{10}t^4)$. However,
the purpose of that article was to establish, for the first time, that various families 
of functions are mock modular forms, rather than obtain a precise formula for the 
multipliers on large subgroups. 
Our construction of the 
harmonic Maass forms greatly differs from that of \cite{BringmannOnoRhoades}.
Additionally, Hickerson and Mortenson in \cite{HickersonMortenson2} considered 
the function
\begin{align*}
	D_2(a;M) &= \sum_{n=0}^\infty \left(N_2(a,M,n) - \tfrac{p_o(n)}{M}\right)q^n
,
\end{align*}
where $p_o(n)$ is the number of partitions without repeated odd parts. There they 
showed that $D_2(a;M)$ can be expressed as the sum of two Appell-Lerch sums, of a
form similar to $S(k,c;\tau)$, and a theta function.

In Section $5$ we give various formulas so that we may
determine the orders of these modular forms at cusps. With this we can then 
verify various identities for the $M_2$-rank function by passing to modular
forms and using the valence formula.
In particular one can use this to give an exact description of
the $c$-dissection of $R2(\z{c}{a};q)$ in terms of generalized Lambert 
series and modular forms. We give these identities for
$c=7$. Similar formulas were determined by Lovejoy
and Osburn for $c=3$ and $c=5$ in \cite{LovejoyOsburn2} and for
$c=6$ and $c=10$ by Mao in \cite{Mao2}. Rather than using harmonic Maass forms,
those formulas used the $q$-series techniques developed by Atkin and 
Swinnerton-Dyer \cite{AS} to determine rank difference formulas for the rank of 
partitions. Furthermore in \cite{HickersonMortenson2} Hickerson and Mortenson 
demonstrated how their identities for $D_2(a;M)$ can be used to prove the formulas
for $c=3$ and $c=5$. With this in mind, we see that there are at least three
different proof techniques for determining the $c$-dissection of
$R2(\z{c}{a};q)$. One might ask for the pros and cons of the approach
we use in this article. Our approach with harmonic Maass forms shares a
difficulty with the other two methods, which is that we must correctly guess the
identity before we can prove it. This approach with harmonic Maass forms helps
us to guess the identity in that we know the weight and level of the modular 
forms.

\begin{theorem}\label{TheoremDissection}
Let $\z{7}{}$ be a primitive seventh root of unity,
$J_a = \aqprod{q^a,q^{28-a}}{q^{28}}{\infty}$
for $1\le a\le 14$,
$J_0 = \aqprod{q^{28}}{q^{28}}{\infty}$, and
$A(m,n,r) = m\z{7}{}+n\z{7}{2}+r\z{7}{3}+r\z{7}{4}+n\z{7}{5}+m\z{7}{6}$.
Then
\begin{align*}
	R2(\z{7}{};q)
	&= 
		R2_0(q^7)+qR2_1(q^7)+q^2R2_2(q^7)+q^3R2_3(q^7)+q^4R2_4(q^7)+q^5R2_5(q^7)+q^6R2_6(q^7)	
,
\end{align*}
where each $R2_i(q)$ can be written as a sum of terms with $S(k,7;\frac{\tau}{7})$
and quotients of $J_a$. Due to the complexity of the expressions, we only state
the definition of $R2_0(q)$ here and give the definitions of the other $R2_i(q)$
in Section 6. We have
\begin{align*}
&R2_0(q)
=
	A(3,2,2)S(0,7,\tau/7)
	+A(3,2,2)S(3,7,\tau/7)
	+\tfrac{A(-51,-25,-35)q^{-1}J_{0}J_{6}^{2}J_{7}J_{8}^{2}}
		{J_{1}^{2}J_{2}^{2}J_{3}^{2}J_{4}^{2}J_{5}^{2}J_{10}^{2}J_{11}^{2}J_{12}J_{13}^{2}J_{14}^{2}}
	+\tfrac{A(69,34,44)q^{-1}J_{0}J_{6}^{2}J_{7}^{2}J_{8}}
		{J_{1}^{2}J_{2}^{2}J_{3}^{2}J_{4}^{2}J_{5}^{2}J_{9}^{2}J_{11}J_{12}^{2}J_{13}^{2}J_{14}^{2}}
	\\&\quad
	+\tfrac{A(27,19,20)q^{-1}J_{0}J_{6}^{2}J_{7}^{2}J_{8}}
		{J_{1}^{2}J_{2}^{2}J_{3}^{2}J_{4}^{2}J_{5}^{2}J_{9}J_{10}J_{11}^{2}J_{12}^{2}J_{13}^{2}J_{14}}
	+\tfrac{A(28,14,22)q^{-1}J_{0}J_{6}J_{7}J_{8}J_{10}}
		{J_{1}^{2}J_{2}^{2}J_{3}^{2}J_{4}^{2}J_{5}J_{9}J_{11}^{2}J_{12}^{2}J_{13}^{2}J_{14}^{2}}
	+\tfrac{A(-45,-29,-27)q^{-1}J_{0}J_{6}J_{7}^{2}J_{8}}
		{J_{1}^{2}J_{2}^{2}J_{3}^{2}J_{4}^{2}J_{5}J_{9}^{2}J_{10}J_{11}^{2}J_{13}^{2}J_{14}^{2}}
	\\&\quad
	+\tfrac{A(-39,-22,-24)q^{-1}J_{0}J_{6}^{2}J_{8}J_{9}}
		{J_{1}^{2}J_{2}^{2}J_{3}^{2}J_{4}^{2}J_{5}J_{7}J_{11}^{2}J_{12}^{2}J_{13}^{2}J_{14}^{2}}
	+\tfrac{A(48,21,33)q^{-1}J_{0}J_{6}^{2}J_{8}}
		{J_{1}^{2}J_{2}^{2}J_{3}^{2}J_{4}^{2}J_{5}J_{9}J_{11}^{2}J_{12}^{2}J_{13}J_{14}^{2}}
	+\tfrac{A(15,8,6)q^{-1}J_{0}J_{6}^{2}J_{8}}
		{J_{1}^{2}J_{2}^{2}J_{3}^{2}J_{4}^{2}J_{5}J_{10}^{2}J_{11}^{2}J_{13}^{2}J_{14}^{2}}
	+\tfrac{A(-25,-10,-20)q^{-1}J_{0}J_{6}^{2}J_{8}^{2}}
		{J_{1}^{2}J_{2}^{2}J_{3}^{2}J_{4}^{2}J_{5}J_{9}^{2}J_{10}^{2}J_{12}J_{13}^{2}J_{14}^{2}}
	\\&\quad
	+\tfrac{A(-90,-43,-61)q^{-1}J_{0}J_{6}^{2}J_{7}}
		{J_{1}^{2}J_{2}^{2}J_{3}^{2}J_{4}^{2}J_{5}J_{9}^{2}J_{11}J_{12}J_{13}^{2}J_{14}^{2}}
	+\tfrac{A(14,1,9)q^{-1}J_{0}J_{6}^{2}J_{7}}
		{J_{1}^{2}J_{2}^{2}J_{3}^{2}J_{4}^{2}J_{5}J_{9}J_{10}J_{11}^{2}J_{12}J_{13}^{2}J_{14}}
	+\tfrac{A(39,23,29)q^{-1}J_{0}J_{6}^{2}J_{7}J_{8}}
		{J_{1}^{2}J_{2}^{2}J_{3}^{2}J_{4}^{2}J_{5}J_{9}^{2}J_{10}^{2}J_{11}^{2}J_{13}J_{14}^{2}}
	\\&\quad
	+\tfrac{A(-52,-32,-33)q^{-1}J_{0}J_{6}^{2}J_{7}J_{8}}
		{J_{1}^{2}J_{2}^{2}J_{3}^{2}J_{4}^{2}J_{5}J_{9}^{2}J_{10}^{2}J_{11}J_{12}^{2}J_{13}^{2}}
	+\tfrac{A(3,4,0)J_{0}J_{6}^{2}J_{7}J_{8}^{2}}
		{J_{1}^{2}J_{2}^{2}J_{3}^{2}J_{4}^{2}J_{5}J_{9}J_{10}^{2}J_{11}^{2}J_{12}J_{13}^{2}J_{14}^{2}}
	+\tfrac{A(3,3,1)q^{-1}J_{0}J_{6}^{2}J_{7}^{2}}
		{J_{1}^{2}J_{2}^{2}J_{3}^{2}J_{4}^{2}J_{5}J_{8}J_{9}^{2}J_{11}^{2}J_{12}^{2}J_{13}^{2}}
	\\&\quad
	+\tfrac{A(-1,-3,1)J_{0}J_{6}^{2}J_{7}^{2}J_{8}}
		{J_{1}^{2}J_{2}^{2}J_{3}^{2}J_{4}^{2}J_{5}J_{9}^{2}J_{10}J_{11}^{2}J_{12}^{2}J_{13}^{2}J_{14}}
	+\tfrac{A(3,2,2)q^{-1}J_{0}J_{7}J_{8}^{2}}
		{J_{1}^{2}J_{2}^{2}J_{3}^{2}J_{4}^{2}J_{9}^{2}J_{10}J_{11}^{2}J_{12}J_{13}^{2}J_{14}}
	+\tfrac{A(27,14,20)q^{-1}J_{0}J_{6}J_{8}^{2}}
		{J_{1}^{2}J_{2}^{2}J_{3}^{2}J_{4}^{2}J_{7}J_{9}J_{10}J_{11}J_{12}J_{13}^{2}J_{14}^{2}}
	\\&\quad
	+\tfrac{A(-31,-13,-21)q^{-1}J_{0}J_{6}J_{8}^{2}}
		{J_{1}^{2}J_{2}^{2}J_{3}^{2}J_{4}^{2}J_{7}J_{10}^{2}J_{11}^{2}J_{12}J_{13}^{2}J_{14}}
	+\tfrac{A(36,24,20)q^{-1}J_{0}J_{6}J_{10}}
		{J_{1}^{2}J_{2}^{2}J_{3}^{2}J_{4}^{2}J_{9}J_{11}^{2}J_{12}J_{13}^{2}J_{14}^{2}}
	+\tfrac{A(-8,-3,-7)q^{-1}J_{0}J_{6}J_{8}}
		{J_{1}^{2}J_{2}^{2}J_{3}^{2}J_{4}^{2}J_{9}^{2}J_{11}J_{12}^{2}J_{13}^{2}J_{14}}
	+\tfrac{A(26,17,20)q^{-1}J_{0}J_{6}J_{8}}
		{J_{1}^{2}J_{2}^{2}J_{3}^{2}J_{4}^{2}J_{9}J_{10}J_{11}^{2}J_{12}^{2}J_{13}^{2}}
	\\&\quad
	+\tfrac{A(39,22,22)q^{-1}J_{0}J_{6}J_{8}^{2}}
		{J_{1}^{2}J_{2}^{2}J_{3}^{2}J_{4}^{2}J_{9}^{2}J_{10}^{2}J_{11}^{2}J_{12}J_{13}J_{14}}
	+\tfrac{A(0,-1,1)q^{-1}J_{0}J_{6}J_{7}J_{10}^{2}}
		{J_{1}^{2}J_{2}^{2}J_{3}^{2}J_{4}^{2}J_{8}J_{9}^{2}J_{11}^{2}J_{12}^{2}J_{13}^{2}J_{14}}
	+\tfrac{A(-3,0,-4)q^{-1}J_{0}J_{6}J_{7}J_{12}}
		{J_{1}^{2}J_{2}^{2}J_{3}^{2}J_{4}^{2}J_{9}^{2}J_{10}J_{11}^{2}J_{13}^{2}J_{14}^{2}}
	+\tfrac{A(0,-1,1)q^{-1}J_{0}J_{6}^{2}J_{9}}
		{J_{1}^{2}J_{2}^{2}J_{3}^{2}J_{4}^{2}J_{7}^{2}J_{11}^{2}J_{12}J_{13}^{2}J_{14}^{2}}
	\\&\quad
	+\tfrac{A(-24,-16,-16)q^{-1}J_{0}J_{6}^{2}J_{8}}
		{J_{1}^{2}J_{2}^{2}J_{3}^{2}J_{4}^{2}J_{7}^{2}J_{9}J_{12}^{2}J_{13}^{2}J_{14}^{2}}
	+\tfrac{A(24,16,16)q^{-1}J_{0}J_{6}^{2}J_{8}^{2}}
		{J_{1}^{2}J_{2}^{2}J_{3}^{2}J_{4}^{2}J_{7}^{2}J_{9}J_{10}^{2}J_{11}J_{12}J_{13}J_{14}^{2}}
	+\tfrac{A(0,1,-1)q^{-1}J_{0}J_{6}^{2}J_{10}^{2}}
		{J_{1}^{2}J_{2}^{2}J_{3}^{2}J_{4}^{2}J_{7}J_{8}J_{9}J_{11}J_{12}^{2}J_{13}^{2}J_{14}^{2}}
	\\&\quad
	+\tfrac{A(-1,1,0)q^{-1}J_{0}J_{6}^{2}J_{10}}
		{J_{1}^{2}J_{2}^{2}J_{3}^{2}J_{4}^{2}J_{7}J_{8}J_{11}^{2}J_{12}^{2}J_{13}^{2}J_{14}}
	+\tfrac{A(-3,-5,0)q^{-1}J_{0}J_{6}^{2}J_{8}}
		{J_{1}^{2}J_{2}^{2}J_{3}^{2}J_{4}^{2}J_{7}J_{9}^{2}J_{10}^{2}J_{13}^{2}J_{14}^{2}}
	+\tfrac{A(-3,-3,-1)q^{-1}J_{0}J_{6}^{2}J_{10}}
		{J_{1}^{2}J_{2}^{2}J_{3}^{2}J_{4}^{2}J_{8}J_{9}^{2}J_{11}^{2}J_{12}^{2}J_{13}J_{14}}
	+\tfrac{A(-3,-3,-2)qJ_{0}J_{6}^{2}J_{7}J_{8}^{2}}
		{J_{1}^{2}J_{2}^{2}J_{3}^{2}J_{4}^{2}J_{9}^{2}J_{10}^{2}J_{11}^{2}J_{12}J_{13}^{2}J_{14}^{2}}
	\\&\quad
	+\tfrac{A(-24,-15,-17)q^{-1}J_{0}J_{5}J_{8}^{2}}
		{J_{1}^{2}J_{2}^{2}J_{3}^{2}J_{4}^{2}J_{7}^{2}J_{11}^{2}J_{12}J_{13}^{2}J_{14}^{2}}
	+\tfrac{A(0,-1,1)q^{-1}J_{0}J_{5}J_{8}^{2}J_{12}}
		{J_{1}^{2}J_{2}^{2}J_{3}^{2}J_{4}^{2}J_{7}J_{9}J_{10}^{2}J_{11}^{2}J_{13}^{2}J_{14}^{2}}
	+\tfrac{A(0,-1,1)q^{-1}J_{0}J_{5}J_{8}}
		{J_{1}^{2}J_{2}^{2}J_{3}^{2}J_{4}^{2}J_{9}^{2}J_{10}J_{11}^{2}J_{13}^{2}J_{14}}
	+\tfrac{A(0,1,-1)q^{-1}J_{0}J_{5}J_{6}J_{10}^{2}}
		{J_{1}^{2}J_{2}^{2}J_{3}^{2}J_{4}^{2}J_{8}^{2}J_{9}^{2}J_{11}^{2}J_{12}J_{13}^{2}J_{14}}
	\\&\quad
	+\tfrac{A(-4,0,-3)q^{2}J_{0}J_{5}^{2}J_{6}^{2}J_{8}}
		{J_{1}^{2}J_{2}^{2}J_{3}^{2}J_{4}^{2}J_{7}J_{9}^{2}J_{11}^{2}J_{12}^{2}J_{13}^{2}J_{14}^{2}}
	+\tfrac{A(-54,-27,-37)q^{-1}J_{0}J_{7}J_{8}^{2}}
		{J_{1}^{2}J_{2}^{2}J_{3}^{2}J_{4}J_{5}^{2}J_{11}^{2}J_{12}^{2}J_{13}^{2}J_{14}^{2}}
	+\tfrac{A(38,18,26)q^{-1}J_{0}J_{7}^{2}J_{8}^{2}}
		{J_{1}^{2}J_{2}^{2}J_{3}^{2}J_{4}J_{5}^{2}J_{9}J_{10}^{2}J_{11}^{2}J_{13}^{2}J_{14}^{2}}
	\\&\quad
	+\tfrac{A(51,25,35)q^{-1}J_{0}J_{6}J_{8}^{2}J_{9}^{2}}
		{J_{1}^{2}J_{2}^{2}J_{3}^{2}J_{4}J_{5}^{2}J_{7}J_{10}J_{11}^{2}J_{12}^{2}J_{13}^{2}J_{14}^{2}}
	+\tfrac{A(-8,-4,-6)q^{-1}J_{0}J_{6}J_{8}^{2}}
		{J_{1}^{2}J_{2}^{2}J_{3}^{2}J_{4}J_{5}^{2}J_{10}J_{11}^{2}J_{12}^{2}J_{13}J_{14}^{2}}
	+\tfrac{A(-26,-13,-18)q^{-1}J_{0}J_{6}J_{7}^{2}J_{8}}
		{J_{1}^{2}J_{2}^{2}J_{3}^{2}J_{4}J_{5}^{2}J_{9}^{2}J_{10}^{2}J_{11}^{2}J_{12}J_{13}J_{14}}
	.
\end{align*}
\end{theorem}
It is trivial to verify that the $q^{\alpha}S(k,7;\tau/7)$ appearing in the definitions of 
the $R2_i(q)$ are all series with integral  powers of $q$.
One advantage to writing the identity in this form is that we can also read off
formulas involving the rank differences,
\begin{align*}
	R2_{r,s,c}(d;q)
	&=
	\sum_{n=0}^\infty \left(
		N_2(r,c,cn+d) - N_2(s,c,cn+d)
	\right)q^n
	.
\end{align*}
To explain this, we let $\zeta_c=\zeta_c^1$ and use that
$1+\z{c}{}+\z{c}{2}+\dots+\z{c}{c-1}=0$
to find that
\begin{align*}
	R2(\z{c}{};q)
	&=
		\sum_{n=0}^\infty
		\sum_{r=0}^{c-1} N_2(r,c,n)\z{c}{r} q^n
	=
		\sum_{r=1}^{c-1}
		\z{c}{r}		
		\sum_{d=0}^{c-1}		
		q^dR2_{r,0,c}(d,q^c)
	.
\end{align*}
Since $N_2(m,n)=N_2(-m,n)$, we know $R2_{r,0,c}(d,q^c)=R2_{c-r,0,c}(d,q^c)$.
When $c$ is odd we have
\begin{align*}	
	R2(\z{c}{};q)
	&=
		\sum_{r=1}^\frac{c-1}{2}
		(\z{c}{r}+\z{c}{c-r})	
		\sum_{d=0}^{c-1}
		q^d
		R2_{r,0,c}(d;q^c)	
,
\end{align*}
whereas for even $c$ we have
\begin{align*}	
	R2(\z{c}{};q)
	&=
		\z{c}{\frac{c}{2}} \sum_{d=0}^{c-1} q^d R2_{\frac{c}{2},0,c}(d;q^c)	
		+
		\sum_{r=1}^\frac{c-2}{2}
		(\z{c}{r}+\z{c}{c-r})	
		\sum_{d=0}^{c-1}
		q^d
		R2_{r,0,c}(d;q^c)	
.
\end{align*}
In the case of $c$ being an odd prime, we have that
$\z{c}{}$, $\z{c}{2}$, $\dots$, $\z{c}{c-1}$ are linearly independent 
over $\mathbb{Q}$, so if
\begin{align*}
	R2(\z{c}{};q)
	&=
	\sum_{r=1}^{\frac{c-1}{2}}
	(\z{c}{r}+\z{c}{c-r})	
	\sum_{d=0}^{c-1}
		q^d S_r(d;q^{c})
,	
\end{align*}
and each $S_r(d;q)$ is a series in $q$ with rational coefficients, then
$S_r(d;q)=R2_{r,0,c}(d;q)$.
As an example, from the formula for $R2_0(q)$, we have that
\begin{align*}
	&R2_{1,0,7}(0;q)
	=
	3S(0,7,\tau/7)
	+3S(3,7,\tau/7)
	-\tfrac{51q^{-1}J_{0}J_{6}^{2}J_{7}J_{8}^{2}}
		{J_{1}^{2}J_{2}^{2}J_{3}^{2}J_{4}^{2}J_{5}^{2}J_{10}^{2}J_{11}^{2}J_{12}J_{13}^{2}J_{14}^{2}}
	+\tfrac{69q^{-1}J_{0}J_{6}^{2}J_{7}^{2}J_{8}}
		{J_{1}^{2}J_{2}^{2}J_{3}^{2}J_{4}^{2}J_{5}^{2}J_{9}^{2}J_{11}J_{12}^{2}J_{13}^{2}J_{14}^{2}}
	\\&\quad
	+\tfrac{27q^{-1}J_{0}J_{6}^{2}J_{7}^{2}J_{8}}
		{J_{1}^{2}J_{2}^{2}J_{3}^{2}J_{4}^{2}J_{5}^{2}J_{9}J_{10}J_{11}^{2}J_{12}^{2}J_{13}^{2}J_{14}}
	+\tfrac{28q^{-1}J_{0}J_{6}J_{7}J_{8}J_{10}}
		{J_{1}^{2}J_{2}^{2}J_{3}^{2}J_{4}^{2}J_{5}J_{9}J_{11}^{2}J_{12}^{2}J_{13}^{2}J_{14}^{2}}
	-\tfrac{45q^{-1}J_{0}J_{6}J_{7}^{2}J_{8}}
		{J_{1}^{2}J_{2}^{2}J_{3}^{2}J_{4}^{2}J_{5}J_{9}^{2}J_{10}J_{11}^{2}J_{13}^{2}J_{14}^{2}}
	+ \cdots .
\end{align*}
We note it is in fact more convenient to verify the identities for the 
$R2_{r,0,7}(d;q)$, as then we are working with integer coefficients rather
than coefficients from $\mathbb{Z}[\zeta_7]$.
When $c$ is composite, we can also deduce an identity based on the 
minimal polynomial for $\z{c}{}$. 

The rest of the article is organized as follows. In Section 3 
we recall the basics of modular and harmonic Maass forms and
introduce
the functions studied by Zwegers in \cite{Zwegers2, Zwegers}; these functions allow us to
relate our functions to harmonic Maass forms. In Section 4 we work out the
transformation formulas for our functions and prove Theorem
\ref{TheoremModularity}. In Section 5 we give formulas for
the orders at cusps. In Section 6 we use the results of Sections 4 and 5
to prove Theorem \ref{TheoremDissection}.

\section{Modular Forms, Harmonic Maass Forms, and Zwegers' Functions}

We first recall some basic terminology and results for modular and Maass forms.
For further details, one may consult \cite{BringmannOno, Ono1, Rankin, Zwegers}. 
We have that $\SLTwo$ is the multiplicative group of $2\times 2$ integer matrices with 
determinant $1$. The principal congruence subgroup of level $N$ is 
\begin{align*}
	\Gamma(N) &= 
	\left\{
		\TwoTwoMatrix{\alpha}{\beta}{\gamma}{\delta}\in \SLTwo
		:
		\alpha\equiv\delta\equiv 1\pmod{N}, 
		\beta\equiv\gamma\equiv 0\pmod{N}
	\right\}.
\end{align*}
A subgroup $\Gamma$ of $\SLTwo$ is called a congruence subgroup
if $\Gamma\supseteq\Gamma(N)$ for some $N$. Two congruence subgroups we
use are
\begin{align*}
	\Gamma_0(N) &= 
	\left\{
		\TwoTwoMatrix{\alpha}{\beta}{\gamma}{\delta}\in \SLTwo
		: 
		\gamma\equiv 0\pmod{N}
	\right\}
	,\\
	\Gamma_1(N) &= 
	\left\{
		\TwoTwoMatrix{\alpha}{\beta}{\gamma}{\delta}\in \SLTwo
		: 
		\alpha\equiv\delta\equiv 1\pmod{N}, 		
		\gamma\equiv 0\pmod{N}
	\right\}
.
\end{align*}
We recall $\SLTwo$ acts on $\mathcal{H}$ via Mobius transformations, that is
if $B=\SmallMatrix{\alpha}{\beta}{\gamma}{\delta}$ then
$B\tau=\frac{\alpha\tau+\beta}{\gamma\tau+\delta}$.
Additionally we let $j(B,\tau)=\gamma\tau+\delta$. One can verify that
$j(BC,\tau)=j(B,C\tau)\cdot j(C,\tau)$.

We recall a weakly holomorphic modular form of integral weight $k$
on a congruence subgroup $\Gamma$ of $\SLTwo$ is a holomorphic
function $f$ on $\mathcal{H}$ such that
\begin{enumerate}
\item if $B=\SmallMatrix{\alpha}{\beta}{\gamma}{\delta}\in\Gamma$, then
$f(B\tau)=j(B,\tau)^kf(\tau)$,
\item if $B\in\SLTwo$ then $j(B,\tau)^{-k}f(B\tau)$ has an expansion of the form
$\sum_{n=n_0}^\infty a_n\exp(2\pi in\tau/N)$, where $n_0\in\mathbb{Z}$ and $N$ is a positive integer.
\end{enumerate}
When $k$ is a half integer, we require $\Gamma\subset\Gamma_0(4)$ and replace 
$(1)$ with
$f(A\tau)=\Jac{\gamma}{\delta}^{2k}\epsilon(\delta)^{-2k}(\gamma\tau+\delta)^kf(\tau)$.
Here $\Jac{m}{n}$ is the Jacobi symbol extended to all integers $n$ by
\begin{align*}
	\Jac{0}{\pm 1} &= 1
	,&
	\Jac{m}{n}
	&=
		\PieceTwo{\Jac{m}{|n|}}
		{-\Jac{m}{|n|}}
		{\mbox{ if } m>0 \mbox{ or } n > 0,   }
		{\mbox{ if } m < 0 \mbox{ and } n < 0,}	
\end{align*}
and $\epsilon(\delta)$ is $1$ when $\delta\equiv 1\pmod{4}$ and is $i$ otherwise.
For a modular form with multiplier, we replace $(1)$ with
$f(B\tau)=\psi(B) j(B,\tau)^k f(\tau)$ where $|\psi(B)|=1$.

A harmonic weak Maass form satisfies the transformation law in $(1)$.
The condition of holomorphic is replaced with being smooth and 
annihilated by the weight $k$ hyperbolic Laplacian operator,
\begin{align*}
	\Delta_k
	&= 
	-y^2\left( \frac{\partial^2}{\partial x^2}+\frac{\partial^2}{\partial y^2}  \right)
	+
	iky\left( \frac{\partial}{\partial x}+i\frac{\partial}{\partial y}  \right)
,
\end{align*}
where here and throughout the article $\tau=x+iy$.
We note that with
$\frac{\partial}{\partial\tau} 
= \frac{1}{2}(\frac{\partial}{\partial x}-i\frac{\partial}{\partial y})$
and
$\frac{\partial}{\partial\overline{\tau}} 
= \frac{1}{2}(\frac{\partial}{\partial x}+i\frac{\partial}{\partial y})$,
we have
$\Delta_k = -4y^{2-k}\frac{\partial}{\partial\tau}y^k 
\frac{\partial}{\partial\overline{\tau}}$.
Condition $(2)$ is replaced with  $j(B,\tau)^{-k}f(B\tau)$ having at most linear
exponential growth as $y\rightarrow \infty$ in the form that there exists a polynomial
$P_B(\tau)\in\mathbb{C}[q^{-\frac{1}{N}}]$ such that
$j(B,\tau)^{-k}f(B\tau)-P_B(\tau)=O(e^{-\epsilon y})$, with $\epsilon>0$, as 
$y\rightarrow \infty$.

If $f$ is a harmonic weak Maass form of weight $2-k$, with $k>1$, on $\Gamma_1(N)$, 
then $f$
can be written as $f=f^{+} + f^-$, where $f^+$ and $f^-$ have expansions of the
form
\begin{align*}
	f^+(\tau)
	&=
	\sum_{n=n_0}^\infty a(n)q^n
	,&
	f^-(\tau)
	&=
	\sum_{n=1}^\infty b(n) \Gamma(k-1, 4\pi ny) q^{-n}
.
\end{align*}
Here $\Gamma$ is the incomplete Gamma function given by
$\Gamma(\alpha,\beta) = \int_\beta^\infty e^{-t} t^{\alpha-1}dt$.
We call $f^+$ the holomorphic part and $f^-$ the non-holomorphic part. The 
non-holomorphic part is often written instead as an integral of the form
\begin{align*}
	f^-(\tau)
	&=
	\int_{-\overline{\tau}}^\infty
	g(w) (-i(w+\tau))^{k-2} dw
	,
\end{align*}
where $g=\xi_{2-k}(f)$ and 
$\xi_k=2iy^k\overline{\frac{\partial}{\partial\overline{\tau}}}$.
When $f$ is a harmonic Maass form on some other congruence subgroup, similar 
expansions exist, but with $q$ replaced by a fractional power of $q$.

We now recall the various functions needed in working with harmonic Maass forms.
For $u,v,z\in \mathbb{C}$, $\tau\in\mathcal{H}$,
$u,v\not\in\mathbb{Z}+\tau\mathbb{Z}$, and $\ell$ a positive integer we let
\begin{align*}
	\vartheta(z;\tau)
	&=
		\sum_{n\in \mathbb{Z}+\frac{1}{2}}
		e^{ \pi in^2\tau + 2\pi in\left(z+\frac{1}{2}\right) }
	=
		-iq^{\frac18}e^{-\pi iz}\aqprod{e^{2\pi iz},e^{-2\pi iz}q,q}{q}{\infty}
	,\\
	\mu(u,v;\tau)
	&=
		\frac{e^{\pi iu}}{\vartheta(v;\tau)}
		\sum_{n=-\infty}^\infty
		\frac{ (-1)^n e^{\pi in(n+1)\tau + 2\pi inv} }
			{ 1 - e^{2\pi in\tau + 2\pi iu} }
	,\\
	A_{\ell}(u,v;\tau)
	&=
		e^{ \pi i\ell u }
		\sum_{n=-\infty}^\infty
		\frac{ (-1)^{\ell n} e^{ \pi i \ell n(n+1)\tau + 2\pi inv} }
			{ 1 - e^{ 2\pi in\tau + 2\pi iu } }
.
\end{align*}
With
$a=\mbox{Im}(z)/y$, we define
\begin{align*}
	E(z) 
	&=
		2\int_0^z e^{-\pi w^2}dw
	,\\
	R(z;\tau)
	&=
		\sum_{n\in \mathbb{Z}+\frac{1}{2}}
		\left( \mbox{sgn}(n) - E( (n+a)\sqrt{2y} )  \right)
		(-1)^{n-\frac{1}{2}}	
		e^{ -\pi in^2\tau - 2\pi inz }	
.
\end{align*}
We remark that $R(z;\tau)$ is not related to the rank generating function 
$R(\zeta;q)$ discussed in the introduction.
For $a,b\in\mathbb{R}$ we set
\begin{align*}
	g_{a,b}(\tau) 
	&=
	\sum_{n\in\mathbb{Z}+a}
	n e^{ \pi in^2\tau + 2\pi inb }
.
\end{align*}
Finally for $u,v\notin\mathbb{Z}+\tau\mathbb{Z}$ we set 
\begin{align*}
\widetilde{\mu}(u,v;\tau)
&=
	\mu(u,v;\tau) + \frac{i}{2}R(u-v;\tau)	
,\\
\widehat{A}_{\ell}(u,v;\tau)
&=
	A_{\ell}(u,v;\tau)
	+
	\frac{i}{2}\sum_{k=0}^{\ell-1}
	e^{2\pi iku}\vartheta\left( v+k\tau +\frac{(\ell-1)}{2}; \ell\tau \right)
	R\left( \ell u -v-k\tau -\frac{(\ell-1)}{2}; \ell\tau \right)
.
\end{align*}
Zwegers studied most of these functions in his
revolutionary PhD thesis \cite{Zwegers},
giving their transformation
formulas which are essential to our results. The functions 
$A_\ell(u,v;\tau)$ and $\widehat{A}_\ell(u,v;\tau)$ are 
the level $\ell$ Appell functions from \cite{Zwegers2}.

\begin{proposition}\label{PropRankToMu}
Suppose $\zeta=e^{2\pi iz}$, then
\begin{align*}
R2(\zeta;q)
&=
q^{\frac{1}{8}}(\zeta^{-1}-1)\frac{\eta(2\tau)}{\eta(\tau)\eta(4\tau)}
A_2\left(z,-\tau-\frac{1}{2};2\tau\right)
\end{align*}
\end{proposition}
\begin{proof}
We have that
\begin{align*}
R2(\zeta;q)
&=
	\frac{\aqprod{-q}{q^2}{\infty}}{\aqprod{q^2}{q^2}{\infty}}
	\left(1+
		\sum_{n=1}^\infty \frac{ (1-\zeta)(1-\zeta^{-1})(-1)^nq^{2n^2+n}(1+q^{2n}) }
			{(1-\zeta q^{2n})(1-\zeta^{-1}q^{2n})}
	\right)	
\\
&=
	\frac{\aqprod{-q}{q^2}{\infty}}{\aqprod{q^2}{q^2}{\infty}}
	\left(1+
		\sum_{n=1}^\infty (-1)^nq^{2n^2+n}
		\left( 
			\frac{(1-\zeta)}{(1-\zeta q^{2n})}
			+
			\frac{(1-\zeta^{-1})}{(1-\zeta^{-1}q^{2n})}
		\right)	
	\right)
\\
&=
	\frac{\aqprod{-q}{q^2}{\infty}}{\aqprod{q^2}{q^2}{\infty}}
	(1-\zeta)
	\sum_{n=-\infty}^\infty 
		\frac{ (-1)^n q^{2n^2+n}}
		{(1-\zeta q^{2n})}
\\
&=
	(\zeta^{-1}-1)
	\frac{\aqprod{-q}{q^2}{\infty}}{\aqprod{q^2}{q^2}{\infty}}
	A_2\left(z,-\tau-\frac{1}{2}; 2\tau \right)
.
\end{align*}
To finish the proof, we note that
\begin{align*}
\frac{\aqprod{-q}{q^2}{\infty}}{\aqprod{q^2}{q^2}{\infty}}
&=
\frac{\aqprod{q^2}{q^2}{\infty}}{\aqprod{q}{q}{\infty}\aqprod{q^4}{q^4}{\infty}}
=
q^{\frac{1}{8}}\frac{\eta(2\tau)}{\eta(\tau)\eta(4\tau)}
.
\end{align*}
\end{proof}
We note that we could instead only use $\mu$. In particular one can easily verify that
$R2(\zeta;q)=i(1-\zeta)(\zeta^{-1}\mu(2z,-\tau;4\tau)-\mu(2z,\tau;4\tau) )$ 
and more generally that
$\widehat{A}_2(u,v;\tau)=
\vartheta\left(v+\frac{1}{2};2\tau\right)\tmu\left(2u,v+\frac{1}{2};2\tau\right)
+e^{2\pi iu}\vartheta\left(v+\tau+\frac{1}{2};2\tau\right)\tmu\left(2u,v+\tau+\frac{1}{2};2\tau\right)$.
We choose to use $\widehat{A}_2$ as this form makes it immediately apparent that
we should expect $\widetilde{\mathcal{R}2}(a,c;\tau)$ to have the same multiplier
as $\frac{\eta(2\tau)}{\eta(\tau)\eta(4\tau)}$ on a certain subgroup.

\begin{proposition}\label{PropositionN2TildeToMuTilde}
For $a$ and $c$ integers, $c>0$, and $c\nmid 2a$, we have that
\begin{align*}
\mathcal{R}2(a,c;\tau)
&=
	(\z{c}{-a}-1)
	\frac{\eta(2\tau)}{\eta(\tau)\eta(4\tau)}
	A_2\left(\frac{a}{c}, -\tau-\frac{1}{2}; 2\tau \right)
,\\
\widetilde{\mathcal{R}2}(a,c;\tau)
&=
	(\z{c}{-a}-1)
	\frac{\eta(2\tau)}{\eta(\tau)\eta(4\tau)}
	\widehat{A}_2\left(\frac{a}{c}, -\tau-\frac{1}{2}; 2\tau \right)
.
\end{align*}
\end{proposition}
\begin{proof}
We see that the first identity is a trivial corollary of the previous proposition.
For the second identity, given
the definitions of
$\widetilde{\mathcal{R}2}$ and $\widehat{A}_2$,
we see we must show that
\begin{align*}
&
\frac{(\z{c}{-a}-1)\eta(2\tau)}{\eta(\tau)\eta(4\tau)}
\left(
	\frac{i}{2}\vartheta(-\tau;4\tau)R\left(\tfrac{2a}{c} + \tau;4\tau \right)
	+
	\frac{i}{2}\z{c}{a}\vartheta(\tau;4\tau)R\left(\tfrac{2a}{c}- \tau;4\tau \right)
\right)
\\
&=
	-
	i(1-\z{c}{a})\z{2c}{-a}
	\left(
		e^{\frac{\pi i}{4}}	
		\int_{-\overline{\tau}}^{i\infty}			
		\frac{g_{\frac{3}{4}, \frac{1}{2}-\frac{2a}{c}} (4w) dw }{\sqrt{-i(w+\tau)}}					
		+
		e^{-\frac{\pi i}{4}}	
		\int_{-\overline{\tau}}^{i\infty}			
		\frac{g_{\frac{1}{4}, \frac{1}{2}-\frac{2a}{c}} (4w) dw }{\sqrt{-i(w+\tau)}}		
	\right)
.
\end{align*}
Using that
\begin{align*}
\frac{\eta(\tau)\eta(4\tau)}{\eta(2\tau)} &= q^{\frac{1}{8}}\aqprod{q,q^3,q^4}{q^4}{\infty}
,&
\vartheta(\tau;4\tau) &= -i\aqprod{q,q^3,q^4}{q^4}{\infty}
,&
\vartheta(-\tau;4\tau) &= i\aqprod{q,q^3,q^4}{q^4}{\infty}
,
\end{align*}
we find that
\begin{align*}
&\frac{(\z{c}{-a}-1)\eta(2\tau)}{\eta(\tau)\eta(4\tau)}
\left(
	\frac{i}{2}\vartheta(-\tau;4\tau)R\left(\tfrac{2a}{c} + \tau;4\tau \right)
	+
	\frac{i}{2}\z{c}{a}\vartheta(\tau;4\tau)R\left(\tfrac{2a}{c}- \tau;4\tau \right)
\right)
\\&=
\frac{(\z{c}{-a}-1)q^{-\frac{1}{8}}}{2}
\left( 
	-R\left(\tfrac{2a}{c} + \tau;4\tau \right)
	+
	\z{c}{a}R\left(\tfrac{2a}{c}- \tau;4\tau \right)
\right).
\end{align*}
Now, by Theorem 1.16 of \cite{Zwegers}, we have that
\begin{align*}
	R( \tfrac{2a}{c}-\tau; 4\tau )
	&=
		-\exp\left(\tfrac{4\pi i\tau}{16} + \tfrac{2\pi i}{4}(\tfrac{-2a}{c}+\tfrac{1}{2}) \right)	
		\int_{-\overline{4\tau}}^{i\infty}	
		\frac{g_{\frac{1}{4}, \frac{1}{2} - \frac{2a}{c}}(z) dz }
			{\sqrt{-i(z+4\tau)}}
	=
		-2q^{\frac{1}{8}}e^{\frac{\pi i}{4}} \z{2c}{-a}  	
		\int_{-\overline{\tau}}^{i\infty}	
		\frac{g_{\frac{1}{4}, \frac{1}{2} - \frac{2a}{c}}(4w) dw }
			{\sqrt{-i(w+\tau)}}
	,
\end{align*}
and similarly
\begin{align*}
	R( \tfrac{2a}{c}+\tau; 4\tau )
	&=
		-\exp\left(\tfrac{4\pi i\tau}{16} - \tfrac{2\pi i}{4}(\tfrac{-2a}{c}+\tfrac{1}{2}) \right)	
		\int_{-\overline{4\tau}}^{i\infty}	
		\frac{g_{\frac{3}{4}, \frac{1}{2} - \frac{2a}{c}}(z) dz }
			{\sqrt{-i(z+4\tau)}}
	=
		-2q^{\frac{1}{8}}e^{\frac{-\pi i}{4}} \z{2c}{a}  	
		\int_{-\overline{\tau}}^{i\infty}	
		\frac{g_{\frac{3}{4}, \frac{1}{2} - \frac{2a}{c}}(4w) dw }
			{\sqrt{-i(w+\tau)}}
	.
\end{align*}
The proposition then follows.
\end{proof}

\begin{proposition}\label{PropositionSTildeToMuTilde}
Suppose $k$ and $c$ are integers, $0\le k<c$, and $1+4k-c\not\equiv0\pmod{4c}$. 
Then 
\begin{align*}
\mathcal{S}(k;c;\tau)
&=
	q^{-\frac{(1+4k-2c)^2}{8}}
	\mu\left( (1+4k-c)c\tau , c^2\tau + \frac{(c-1)}{2}  ; 4c^2\tau \right)
\\
&=
	\left\{\begin{array}{ll}
		\frac{i^{c} \eta(2c^2\tau)}{\eta(c^2\tau)\eta(4c^2\tau)}
		q^{-\frac{(3c-4k-1)(c-4k-1)}{8}}		
		A_1\left( (1+4k-c)c\tau, c^2\tau + \frac{(c-1)}{2}; 4c^2\tau  \right)
		& \mbox{ if c is odd,}
		\\[1.5ex]
		\frac{i^{c} \eta(c^2\tau)}{\eta(2c^2\tau)^2}
		q^{-\frac{(3c-4k-1)(c-4k-1)}{8}}		
		A_1\left( (1+4k-c)c\tau, c^2\tau + \frac{(c-1)}{2}; 4c^2\tau  \right)
		& \mbox{ if c is even,}
	\end{array}\right.
\\
\widetilde{\mathcal{S}}(k;c;\tau)
&=
	q^{-\frac{(1+4k-2c)^2}{8}}
	\tmu\left( (1+4k-c)c\tau , c^2\tau + \frac{(c-1)}{2}  ; 4c^2\tau \right)
\\
&=
	\left\{\begin{array}{ll}
		\frac{i^{c} \eta(2c^2\tau)}{\eta(c^2\tau)\eta(4c^2\tau)}
		q^{-\frac{(3c-4k-1)(c-4k-1)}{8}}		
		\widehat{A}_1\left( (1+4k-c)c\tau, c^2\tau + \frac{(c-1)}{2}; 4c^2\tau  \right)
		& \mbox{ if c is odd,}
		\\[1.5ex]
		\frac{i^{c} \eta(c^2\tau)}{\eta(2c^2\tau)^2}
		q^{-\frac{(3c-4k-1)(c-4k-1)}{8}}		
		\widehat{A}_1\left( (1+4k-c)c\tau, c^2\tau + \frac{(c-1)}{2}; 4c^2\tau  \right)
		& \mbox{ if c is even.}
	\end{array}\right.
\end{align*}
\end{proposition}
\begin{proof}
We find that
\begin{align*}
\mathcal{S}(k;c;\tau)
&=
	q^{-\frac{(1+4k-2c)^2}{8}}
	\mu\left( (1+4k-c)c\tau , c^2\tau + \frac{(c-1)}{2}  ; 4c^2\tau \right)
\end{align*}
follows directly from definitions. Next, we note that when $c$ is odd
\begin{align*}
\frac{1}{\aqprod{(-1)^{c-1}q^{c^2},(-1)^{c-1}q^{3c^2},q^{4c^2}}{q^{4c^2}}{\infty}}
&=
\frac{1}{\aqprod{q^{c^2},q^{3c^2},q^{4c^2}}{q^{4c^2}}{\infty}}
=
\frac{\aqprod{q^{2c^2}}{q^{2c^2}}{\infty}}{\aqprod{q^{c^2}}{q^{c^2}}{\infty} \aqprod{q^{4c^2}}{q^{4c^2}}{\infty}}
\\
&=
\frac{q^{\frac{c^2}{8}} \eta(2c^2\tau)}{ \eta(c^2\tau) \eta(4c^2\tau) }
,
\end{align*}
and when $c$ is even
\begin{align*}
\frac{1}{\aqprod{(-1)^{c-1}q^{c^2},(-1)^{c-1}q^{3c^2},q^{4c^2}}{q^{4c^2}}{\infty}}
&=
\frac{1}{\aqprod{-q^{c^2},-q^{3c^2},q^{4c^2}}{q^{4c^2}}{\infty}}
=
\frac{\aqprod{q^{c^2},q^{3c^2}}{q^{4c^2}}{\infty} }
	{\aqprod{q^{4c^2}}{q^{4c^2}}{\infty} \aqprod{q^{2c^2},q^{6c^2}}{q^{8c^2}}{\infty}}
\\
&=
\frac{\aqprod{q^{c^2}}{q^{c^2}}{\infty}}{\aqprod{q^{2c^2}}{q^{2c^2}}{\infty}^2 }
=
\frac{q^{\frac{c^2}{8}} \eta(c^2\tau)}{ \eta(2c^2\tau)^2}
.
\end{align*}
The identities for $\mathcal{S}(k,c;\tau)$ then follow from the definition of $A_1(u,v;\tau)$.

To establish the identities for $\widetilde{\mathcal{S}}(k,c;\tau)$ we must verify that
for $0\le k<c$ we have
\begin{align*}
	\frac{i}{2}q^{-\frac{(1+4k-2c)^2}{8}} 
	R\left( (1+4k-2c)c\tau - \tfrac{c-1}{2} ; 4c^2\tau  \right)
	&=
	(-1)^{k+1} i^{1+c} e^{-\frac{\pi i}{4}} c
	\int_{-\overline{\tau}}^\infty
	\frac{  g_{ \frac{1+4k}{4c}, \frac{c}{2}     }(4c^2w)   }
	{\sqrt{-i(w+\tau)}}dw
.
\end{align*}
For this we use Theorem 1.6 of \cite{Zwegers} to determine that
\begin{align*}
	&\frac{i}{2}q^{-\frac{(1+4k-2c)^2}{8}} 
	R\left( (1+4k-2c)c\tau - \tfrac{c-1}{2}  ;4c^2\tau \right)
	\\
	&=
		-\frac{i}{2}q^{-\frac{(1+4k-2c)^2}{8}} 
		\exp\left(  \pi i \tau\tfrac{(1+4k-2c)^2}{4} -  \pi i\tfrac{(1+4k-2c)}{4}				
		\right)	
		\int_{-\overline{4c^2\tau}}^{i\infty}
		\frac{g_{ \frac{1+4k}{4c}, \frac{c}{2}  }(z)  }{\sqrt{-i(z+4c^2\tau)}}
			dz
	\\			
	&=
		(-1)^{k+1} i^{1+c} e^{-\frac{\pi i}{4}} c	
		\int_{-\overline{\tau}}^{i\infty}
		\frac{g_{ \frac{1+4k}{4c}, \frac{c}{2}  }(4c^2w)  }{\sqrt{-i(4c^2w+4c^2\tau)}}
			dw
.			
\end{align*}

\end{proof}

It is through these identities that we take the definitions of
$\mathcal{S}(k,c;\tau)$ and $\widetilde{\mathcal{S}}(k,c;\tau)$
when $k$ is an arbitrary integer. That is to say, they are defined in terms of 
$\mu$ and $\tmu$ rather than series and integrals.

Working with $\tmu(u,v;\tau)$ and $\widehat{A}_{\ell}(u,v;\tau)$ is advantageous in that  
the transformations under the action of the modular group
$\SLTwo$ is known and quite elegant.
With this we are able to determine a simple multiplier for
$\widetilde{\mathcal{R}2}(a,c;\tau)$
and $\widetilde{\mathcal{S}}(k,c;\tau)$.
For this reason we do not need to replace
$\tau$ by $8\tau$ to have a strong understanding of these functions,
even with the factor $q^{-\frac{1}{8}}$ in the definitions of
$\mathcal{R}2(a,c;\tau)$
and $\widetilde{\mathcal{R}2}(a,c;\tau)$.
Additionally the transformations of $\tmu$ allow us to
easily determine the behavior of $\widetilde{\mathcal{R}2}(a,c;\tau)$
and $\widetilde{\mathcal{S}}(k,c;\tau)$ at the cusps.

\section{Transformations Formulas}

For a matrix  $B=\SmallMatrix{\alpha}{\beta}{\gamma}{\delta}\in\SLTwo$
we have $\nu(B)$, the $\eta$-multiplier, defined by
\begin{align*}
	\eta(B\tau) &= \nu(B)\sqrt{\gamma\tau+\delta} \,\eta(\tau)
,
\end{align*}
where $\eta(\tau)$ is Dedekind's eta-function,
\begin{align*}
	\eta(\tau) &= q^{\frac{1}{24}}\aqprod{q}{q}{\infty}
.
\end{align*}
A convenient form for the $\eta$-multiplier, which can be found 
as Theorem 2 in Chapter 4 of \cite{Knopp}, is
\begin{align}
	\label{EqEtaMultipler}
	\nu(B)
	&=
	\left\{
	\begin{array}{ll}
		\big(\frac{\delta}{|\gamma|} \big)
		\exp\left(\frac{\pi i}{12}\left(
			(\alpha+\delta)\gamma - \beta\delta(\gamma^2-1) - 3\gamma		
		\right)\right)
		&
		\mbox{ if } \gamma \equiv 1 \pmod{2},
		\\				
		\Jac{\gamma}{\delta}
		\exp\left(\frac{\pi i}{12}\left(
			(\alpha+\delta)\gamma - \beta\delta(\gamma^2-1) + 3\delta - 3 - 3\gamma\delta		
		\right)\right)
		&
		\mbox{ if } \delta\equiv 1\pmod{2}.
	\end{array}
	\right.
\end{align}
For an integer $m$ we let
\begin{align*}
 	B_{m} &= \TwoTwoMatrix{\alpha}{m\beta}{\gamma/m}{\delta}
	.
\end{align*}
The utility of $B_m$ is in the fact that $mB\tau=B_m(m\tau)$.

Our transformation formulas for $\widetilde{\mathcal{R}2}(a,c;\tau)$ 
and $\widetilde{\mathcal{S}}(k,c;\tau)$
are easily deduced by the
transformations of $\tmu(u,v;\tau)$ and $\widehat{A}_\ell(u,v;\tau)$. The following 
essential properties are from
Theorem 1.11 of \cite{Zwegers} and Theorem 2.2 of \cite{Zwegers2}. If $k,l,m,n$ are integers then
\begin{align}
	\label{EqTildeMuElliptic}
	\tmu\left(u+k\tau+l, v+m\tau+n; \tau \right)
	&=
	(-1)^{k+l+m+n}
	e^{
		\pi i\tau(k-m)^2 + 2\pi i(k-m)(u-v)	
	}
	\tmu(u,v;\tau)
	,\\
	\label{EqEllipticTransformationAl}
	\widehat{A}_\ell(u+k\tau+l,v+m\tau+n;\tau)
	&=
	(-1)^{\ell(k+l)}
	e^{2\pi iu(\ell k -m)}e^{-2\pi ivk}q^{\frac{\ell k^2}{2}-km}\widehat{A}_\ell(u,v;\tau)
	,
\end{align}
and
if $B=\SmallMatrix{\alpha}{\beta}{\gamma}{\delta}\in\SLTwo$ then
\begin{align}
	\label{EqTildeMuModular}
	\tmu\left(
		\frac{u}{\gamma\tau+\delta},
		\frac{v}{\gamma\tau+\delta};
		B\tau		
	\right)
	&=
	\nu(B)^{-3}
	\exp\left(
		-\frac{\pi i\gamma(u-v)^2}{\gamma\tau+\delta}
	\right)	
	\sqrt{\gamma\tau+\delta}\,
	\tmu\left(u, v;\tau\right)	
	,\\
	\label{EqModularTransformationAl}
	\widehat{A}_\ell\left(\frac{u}{\gamma\tau+\delta}, \frac{v}{\gamma\tau+\delta};B\tau\right)
	&=
	(\gamma\tau+\delta)e^{\frac{\pi i\gamma}{(\gamma\tau+\delta)}\left(-\ell u^2+2uv\right)}
	\widehat{A}_\ell(u,v;\tau)
	.
\end{align}

The propositions and formulas of this section are organized as follows.
We begin by investigating the transformation formula for 
$\widetilde{\mathcal{R}2}(a,c;\tau)$ that follows
from \eqref{EqModularTransformationAl} and refine this result by restricting to
a smaller congruence subgroup of $\SLTwo$ that
yields a simple multiplier. We then verify
$\widetilde{\mathcal{R}2}(a,c;\tau)$ is annihilated by $\Delta_{\frac{1}{2}}$
and that $\widetilde{\mathcal{R}2}(a,c;\tau)$ behaves correctly at the cusps.
With this we will have established part (1) of Theorem \ref{TheoremModularity}.
We proceed by determining a formula for 
the non-holomorphic part of
$\widetilde{\mathcal{R}2}(a,c;\tau)$ that allows us to determine the 
non-holomorpic parts of the $c$-dissections of 
$\widetilde{\mathcal{R}2}(a,c;\tau)$. We then recognize $\widetilde{\mathcal{S}}(k,c;\tau)$ as the 
functions whose non-holomorphic parts agree with those of the $c$-dissections
of $\widetilde{\mathcal{R}2}(a,c;\tau)$
and whose holomorphic parts are $q$-series with exponents taking on a single
value modulo $c$.
After this we give $\widetilde{\mathcal{S}}(k,c;\tau)$
the same treatment as $\widetilde{\mathcal{R}2}(a,c;\tau)$.
Lastly we quickly deduce the necessary transformation formulas for various
generalized eta quotients using the work of Biagioli in \cite{Biagioli}.

\begin{proposition}\label{PropositionTransformationR2Tilde}
Suppose $a$ and $c$ are integers, $c>0$, and $c\nmid 2a$.
If $B=\begin{psmallmatrix}\alpha&\beta\\\gamma&\delta\end{psmallmatrix}
\in\Gamma_0(4)\cap\Gamma_0(2c)\cap\Gamma_1(c)$, then
\begin{align*}
\widetilde{\mathcal{R}2}(a,c;B\tau)
&=	
\frac{\nu(B_2)}{\nu(B)\nu(B_4)}
e^{\frac{\pi ia}{c}\left( \frac{2a\gamma}{c}-\frac{a\gamma\delta}{c}-\frac{\gamma}{2}+\alpha-1  \right)}
\sqrt{\gamma\tau+\delta}
\widetilde{\mathcal{R}2}(a,c;\tau).
\end{align*}
In particular, if $B=\begin{psmallmatrix}\alpha&\beta\\\gamma&\delta\end{psmallmatrix}
\in\Gamma_0(4)\cap\Gamma_0(2c^2)\cap\Gamma_1(2c)$, then
\begin{align*}
\widetilde{\mathcal{R}2}(a,c;B\tau)
&=	
\frac{\nu(B_2)}{\nu(B)\nu(B_4)}
\sqrt{\gamma\tau+\delta}
\widetilde{\mathcal{R}2}(a,c;\tau).
\end{align*}
\end{proposition}
\begin{proof}
To begin we let
\begin{align*}
\widehat{F}(a,c;\tau)&=\widehat{A}_2\left(\frac{a}{c},-\tau-\frac{1}{2};2\tau\right),
\end{align*}
so that
\begin{align*}
\widetilde{\mathcal{R}2}(a,c;\tau)
&=
(\z{c}{-a}-1)\frac{\eta(2\tau)}{\eta(\tau)\eta(4\tau)}\widehat{F}(a,c;\tau)
.
\end{align*}
We note that $B_2=\begin{psmallmatrix}\alpha&2\beta\\\gamma/2&\delta\end{psmallmatrix}\in\SLTwo$, 
so that we may apply \eqref{EqModularTransformationAl}
to find that
\begin{align*}
\widehat{F}(a,c;B\tau)
&=
	\widehat{A}_2\left( \frac{a}{c}, -B\tau -\frac{1}{2}; B_2(2\tau) \right)
\\
&=
	(\gamma\tau+\delta)
	\exp\left( \frac{\pi i\gamma}{2(\gamma\tau+\delta)}\left(
		-\frac{2a^2(\gamma\tau+\delta)^2}{c^2}
		-\frac{2a(\gamma\tau+\delta)}{c}\left( (\alpha\tau+\beta) + \frac{(\gamma\tau+\delta)}{2}  \right)
	\right)\right)
	\\&\quad\cdot
	\widehat{A}_2\left(
		\frac{a(\gamma\tau+\delta)}{c},
		-(\alpha\tau+\beta) - \frac{(\gamma\tau+\delta)}{2}
		; 2\tau
	\right)
\\
&=
	(\gamma\tau+\delta)
	e^{-\frac{\pi ia\gamma}{2c}\left( \frac{2a\delta}{c}+2\beta+\delta  \right)}
	q^{-\frac{a\gamma}{2c}\left( \frac{a\gamma}{c}+\alpha+\frac{\gamma}{2}  \right)}
	\widehat{A}_2\left( 
		\frac{a\gamma}{c}\tau+\frac{a\delta}{c},
		-\frac{(2\alpha+\gamma)}{2}\tau - \beta-\frac{\delta}{2};
		2\tau  
	\right)
\\
&=
	(\gamma\tau+\delta)
	e^{-\frac{\pi ia\gamma}{2c}\left( \frac{2a\delta}{c}+1  \right)}
	q^{-\frac{a\gamma}{2c}\left( \frac{a\gamma}{c}+\alpha+\frac{\gamma}{2}  \right)}
	\widehat{A}_2\left( 
		\frac{a\gamma}{c}\tau+\frac{a\delta}{c},
		-\frac{(2\alpha+\gamma)}{2}\tau - \beta-\frac{\delta}{2};
		2\tau  
	\right)
.
\end{align*}
Since
$\gamma\equiv 0 \pmod{2c}$,
$\delta\equiv 1\pmod{c}$,
$\delta\equiv 1\pmod{2}$, and
$\alpha-1+\frac{\gamma}{2}\equiv 0\pmod{2}$
we can apply \eqref{EqEllipticTransformationAl} to get that
\begin{align*}
&\widehat{A}_2\left( 
	\frac{a\gamma}{c}\tau+\frac{a\delta}{c},
	-\frac{2a+\gamma}{2}\tau - \beta-\frac{\delta}{2};
	2\tau  
\right)
\\&=
	\widehat{A}_2\left( 
		\frac{a}{c} + \frac{a\gamma}{2c}2\tau+\frac{a(\delta-1)}{c},
		-\tau-\frac{1}{2}
		-\left( \frac{\alpha-1}{2} + \frac{\gamma}{4}\right)2\tau - \beta-\frac{(\delta-1)}{2};
		2\tau  
	\right)
\\
&=
	e^{\frac{2\pi ia}{c}\left( \frac{a\gamma}{c}+\frac{(\alpha-1)}{2}+\frac{\gamma}{4}  \right)}
	e^{2\pi i\left(\tau+\frac{1}{2}\right)\frac{a\gamma}{2c} }
	q^{2\left( \frac{a^2\gamma^2}{4c^2} + \frac{a\gamma}{2c}\left(\frac{(\alpha-1)}{2}+\frac{\gamma}{4}\right)  \right)}
	\widehat{A}_2\left(\frac{a}{c},-\tau-\frac{1}{2};2\tau\right)
\\
&=
	e^{\frac{2\pi ia}{c}\left( \frac{a\gamma}{c}+\frac{(\alpha-1)}{2}+\frac{\gamma}{2}  \right)}
	q^{\frac{a\gamma}{2c}\left( \frac{a\gamma}{c} + \alpha + \frac{\gamma}{2}  \right)}
	\widehat{F}\left(a,c;\tau\right)
\\
&=
	e^{\frac{2\pi ia}{c}\left( \frac{a\gamma}{c}+\frac{(\alpha-1)}{2}  \right)}
	q^{\frac{a\gamma}{2c}\left( \frac{a\gamma}{c} + \alpha + \frac{\gamma}{2}  \right)}
	\widehat{F}\left(a,c;\tau\right)
.
\end{align*}
For $B\in\Gamma_0(4)\cap\Gamma_0(2c)\cap\Gamma_1(c)$ this yields that
\begin{align*}
\widetilde{\mathcal{R}2}(a,c;B\tau)
&=	
\sqrt{\gamma\tau+\delta}
\frac{\nu(B_2)}{\nu(B)\nu(B_4)}
e^{\frac{\pi ia}{c}\left( \frac{2a\gamma}{c}-\frac{a\gamma\delta}{c}-\frac{\gamma}{2}+\alpha-1  \right)}
\widetilde{\mathcal{R}2}(a,c;\tau).
\end{align*}
To complete the proof we simply note that
\begin{align*}
e^{\frac{\pi ia}{c}\left( \frac{2a\gamma}{c}-\frac{a\gamma\delta}{c}-\frac{\gamma}{2}+\alpha-1  \right)}
&=1
\end{align*}
for $B\in\Gamma_0(4)\cap\Gamma_0(2c^2)\cap\Gamma_1(2c)$.
\end{proof}
The rest of the transformation formulas have similar proofs. We apply 
\eqref{EqTildeMuModular} or \eqref{EqModularTransformationAl}
followed by \eqref{EqTildeMuElliptic} or \eqref{EqEllipticTransformationAl} and 
reduce the various powers of $q$ and roots of unity. For this reason we will 
omit the majority of these calculations as they are routine and straightforward, 
but rather long.

\begin{proposition}\label{PropR2HarmonicMaassForm}
Suppose $a$ and $c$ integers, $c>0$, and $c\nmid 2a$.
Then $\widetilde{\mathcal{R}2}(a,c;\tau)$ is a harmonic Maass form of weight
$\frac{1}{2}$ on $\Gamma_0(4)\cap\Gamma_0(2c)\cap\Gamma_1(c)$ with multiplier
$\frac{\nu(B_2)}{\nu(B)\nu(B_4)} \cdot 
e^{\frac{\pi ia}{c}\left(\frac{2a\gamma}{c}-\frac{a\gamma\delta}{c}-\frac{\gamma}{2}+\alpha-1\right)}$
for $B=\begin{psmallmatrix}\alpha&\beta\\\gamma&\delta\end{psmallmatrix}$.
Furthermore, $\widetilde{\mathcal{R}2}(a,c;8\tau)$ is a harmonic Maass form of 
weight $\frac{1}{2}$ on $\Gamma_0(256)\cap\Gamma_0(16c^2)\cap\Gamma_1(8)\cap\Gamma_1(2c)$.
\end{proposition}
\begin{proof}
The transformation formula for $\widetilde{\mathcal{R}2}(a,c;\tau)$ 
is Proposition \ref{PropositionTransformationR2Tilde}. Since
$\mathcal{R}2(a,c;\tau)$ is holomorphic in $\tau$, to show
that $\Delta_{\frac{1}{2}}$ annihilates $\widetilde{\mathcal{R}2}(a,c;\tau)$,
it is sufficient to verify that 
$\sqrt{y}\frac{\partial}{\partial\overline{\tau}}q^{-\frac{1}{8}}R(\frac{2a}{c}\pm\tau;4\tau)$
is holomorphic in $\overline{\tau}$. However this follows by Lemma 1.8 of \cite{Zwegers}.

While the growth condition at the cusps follows from the fact that 
$\frac{\eta(2\tau)}{\eta(\tau)\eta(4\tau)}$ is a modular form,
the transformation formula
for $\widehat{A}_2$, and properties of $R$, we will later require a
more precise bound for the holomorphic part. To establish this bound, it is
more convenient to work with $\tmu$. In particular,
we suppose $\alpha$ and $\gamma$ are integers with $\gcd(\alpha,\gamma)=1$
and take 
$B=\begin{psmallmatrix}\alpha&\beta\\\gamma&\delta\end{psmallmatrix}\in\SLTwo$.
We set $g=\gcd(\gamma,4)$, $m=\frac{4\alpha}{g}$, and $r=\frac{\gamma}{g}$.
Since $\gcd(m,r)=1$, we take
$L=\begin{psmallmatrix}m&n\\r&s\end{psmallmatrix}\in\SLTwo$ and set
\begin{align*}
C &= L^{-1}\begin{pmatrix}4\alpha&4\beta\\\gamma&\delta\end{pmatrix}
=\begin{pmatrix}g&4\beta s -\delta n\\0& 4/g \end{pmatrix}.
\end{align*}
This gives that $4B\tau=LC\tau$ and 
$(\gamma\tau+\delta)=j(LC,\tau)=j(L,C\tau)\cdot j(C,\tau)=j(L,C\tau)\cdot\frac{4}{g}$.
By \eqref{EqTildeMuModular} we find that
\begin{align*}
&(\gamma\tau+\delta)^{-\frac{1}{2}}\widetilde{\mathcal{R}2}(a,c;B\tau)
\\
&=
(\gamma\tau+\delta)^{-\frac{1}{2}}i(1-\z{c}{a})\exp\left(-\frac{\pi iLC\tau}{16}\right)
\left(
\z{c}{a}\tmu\left(\frac{2a}{c},-\frac{LC\tau}{4};LC\tau\right)
-\tmu\left(\frac{2a}{c},\frac{LC\tau}{4};LC\tau\right)
\right)
\\
&=
\epsilon_1\exp\left(2\pi iC\tau\left( -\frac{m^2}{32}-\frac{2a^2r^2}{c^2}-\frac{amr}{2c}  \right)\right)
\tmu\left(\frac{2arC\tau+2as}{c},\frac{-mC\tau-s}{4};C\tau \right)
\\&\quad+
\epsilon_2\exp\left(2\pi iC\tau\left( -\frac{m^2}{32}-\frac{2a^2r^2}{c^2}+\frac{amr}{2c}  \right)\right)
\tmu\left(\frac{2arC\tau+2as}{c},\frac{mC\tau+s}{4};C\tau \right)
,
\end{align*}
where $\epsilon_1$ and $\epsilon_2$ are non-zero constants.
Upon noting $e^{2\pi iC\tau}=\epsilon_3 q^{\frac{g^2}{4}}$,
where $\epsilon_3$ is a fixed root of unity, we find that 
$(\gamma\tau+\delta)^{-\frac{1}{2}}\widetilde{\mathcal{R}2}(a,c;B\tau)$
has at most linear exponential growth as $y\rightarrow \infty$.

To determine the transformation formula for $\widetilde{\mathcal{R}2}(a,c;8\tau)$,
we first suppose $B=\begin{psmallmatrix}\alpha&\beta\\\gamma&\delta\end{psmallmatrix}
\in \Gamma_0(256)\cap\Gamma_0(16c^2)\cap\Gamma_1(8)\cap\Gamma_1(2c)$.
Since $B_{8}\in \Gamma_0(4)\cap\Gamma_0(2c^2)\cap\Gamma_1(2c)$,
we have by Proposition \ref{PropositionTransformationR2Tilde}
that
$\widetilde{\mathcal{R}2}(a,c;8\tau)$ has multiplier
$\frac{\nu(B_{16})}{\nu(B_8)\nu(B_{32})}$. A direct calculation using
\eqref{EqEtaMultipler} reveals that on $\Gamma_0(32)$
\begin{align*}
\frac{\nu(B_{16})}{\nu(B_8)\nu(B_{32})}
&=
\Jac{\gamma}{\delta}\exp\left(\frac{\pi i}{4}\left(\frac{\gamma\delta}{32}-\delta+1\right)\right).
\end{align*}
However, on $\Gamma_0(256)\cap\Gamma_1(8)$ we have
$
\exp\left(\frac{\pi i}{4}\left(\frac{\gamma\delta}{32}-\delta+1\right)\right)
=1$,
so that $\widetilde{\mathcal{R}2}(a,c;8\tau)$ transforms as claimed.
\end{proof}
Upon noting 
that $\Gamma_0(256)\cap\Gamma_0(16c^2)\cap\cap\Gamma_1(8)\cap\Gamma_1(2c)$
can be abbreviated to 
$\Gamma_0(256t^2)\cap\Gamma_1(8t)$ where
$t=\frac{c}{gcd(4,c)}$,
we have now proved part $(1)$ of Theorem \ref{TheoremModularity}.

\begin{proposition}
Suppose $a$ and $c$ are integers, $c>0$, and $c\nmid 2a$.
The non-holomorphic part of $\widetilde{\mathcal{R}2}(a,c;\tau)$
is
\begin{align*}
	&\frac{(1-\z{c}{a})\z{c}{-a}}{2\sqrt{\pi}}
	\sum_{n=-\infty}^\infty
		(-1)^n {\normalfont\mbox{sgn}}(n+\tfrac{1}{4})
		(\z{c}{-2an} - \z{c}{2an+a})   
		q^{-2n^2-n-\frac{1}{8}} 
		\Gamma(\tfrac{1}{2}, 8\pi(n+\tfrac{1}{4})^2 y  )	
	\\
	&=
	\frac{i^{1-c} (1-\z{c}{a})\z{c}{-a}}{2}
	\sum_{k=0}^{c-1}
		(-1)^k (\z{c}{-2ak} - \z{c}{2ak+a})
		q^{-\frac{(1+4k-2c)^2}{8} }  
		R( c(1+4k-2c)\tau + \tfrac{1-c}{2}; 4c^2\tau)
	.
\end{align*}
\end{proposition}
\begin{proof}
The non-holomorphic part of $\widetilde{\mathcal{R}2}(a,c;\tau)$ is
\begin{align*}
	-i(1-\z{c}{a})\z{2c}{-a}\left(
		\exp\left(\tfrac{\pi i}{4}\right)
		\int_{-\overline{\tau}}^{i\infty}
		\frac{ g_{\frac{3}{4},\frac{1}{2}-\frac{2a}{c}}(4w)}
			{\sqrt{-i(w+\tau)}}dw
		+
		\exp\left(\tfrac{-\pi i}{4}\right)
		\int_{-\overline{\tau}}^{i\infty}
		\frac{ g_{\frac{1}{4},\frac{1}{2}-\frac{2a}{c}}(4w)}
			{\sqrt{-i(w+\tau)}}dw
	\right)	
	.
\end{align*}
We find that
\begin{align*}
	\int_{-\overline{\tau}}^{i\infty}
	\frac{ g_{\frac{3}{4},\frac{1}{2}-\frac{2a}{c}}(4w)}
		{\sqrt{-i(w+\tau)}}dw
	&=
	\sum_{n=-\infty}^\infty
		(n+\tfrac{3}{4})
	\int_{-\overline{\tau}}^{i\infty}
		\frac{e^{\pi i(n+\frac{3}{4})^2 4w + 2\pi i(n+\frac{3}{4})(\frac{1}{2}-\frac{2a}{c})  } }
		{\sqrt{-i(w+\tau)}}dw	
	\\
	&=
	\sum_{n=-\infty}^\infty
		\frac{(-1)^{n+1} 
			\z{c}{-2an}
			\z{2c}{-3a}
			e^{ \frac{3\pi i}{4} }
		}{4\pi i(n+\frac{3}{4})}
	\int_{8\pi(n+\frac{3}{4})^2 y }^{\infty}
		\frac{e^{-t} e^{-4\pi i(n+\frac{3}{4})^2\tau } }
		{\sqrt{\frac{t}{4\pi (n+\frac{3}{4})^2} }}dt	
	\\
	&=
	\frac{i e^{\frac{3\pi i}{4}} \z{2c}{-3a} }
		{2\sqrt{\pi}}
	\sum_{n=-\infty}^\infty
		(-1)^n \z{c}{-2an} q^{-2n^2-3n - \frac{9}{8}} 
		\mbox{sgn}(n+\tfrac{3}{4})
		\Gamma(\tfrac{1}{2}, 8\pi(n+\tfrac{3}{4})^2 y  )	
,
\end{align*}
where we have used the substitution 
$w=-\frac{t}{4\pi i(n+\frac{3}{4})^2}-\tau$.
Similarly we calculate that
\begin{align*}
	\int_{-\overline{\tau}}^{i\infty}
	\frac{ g_{\frac{1}{4},\frac{1}{2}-\frac{2a}{c}}(4w)}
		{\sqrt{-i(w+\tau)}}dw
	&=
	\frac{i e^{\frac{\pi i}{4}} \z{2c}{-a} }
		{2\sqrt{\pi}}
	\sum_{n=-\infty}^\infty
		(-1)^n \z{c}{-2an} q^{-2n^2-n - \frac{1}{8}} 
		\mbox{sgn}(n+\tfrac{1}{4})
		\Gamma(\tfrac{1}{2}, 8\pi(n+\tfrac{1}{4})^2 y  ).	
\end{align*}
With $n\mapsto -n-1$ we find that
\begin{align*}
	&
	\sum_{n=-\infty}^\infty
		(-1)^n \z{c}{-2an} q^{-2n^2-3n - \frac{9}{8}} 
		\mbox{sgn}(n+\tfrac{3}{4})
		\Gamma(\tfrac{1}{2}, 8\pi(n+\tfrac{3}{4})^2 y  )	
	\\
	&=
	\sum_{n=-\infty}^\infty
		(-1)^{n} \z{c}{2an+2a} q^{-2n^2-n - \frac{1}{8}} 
		\mbox{sgn}(n+\tfrac{1}{4})
		\Gamma(\tfrac{1}{2}, 8\pi(n+\tfrac{1}{4})^2 y  )	
.
\end{align*}
Thus the non-holomorphic part of $\widetilde{\mathcal{R}2}(a,c;\tau)$ is
\begin{align*}
	&\frac{(1-\z{c}{a})\z{c}{-a}}{2\sqrt{\pi}}
	\sum_{n=-\infty}^\infty
		(-1)^n {\normalfont\mbox{sgn}}(n+\tfrac{1}{4})
		(\z{c}{-2an} - \z{c}{2an+a})   
		q^{-2n^2-n-\frac{1}{8}} 
		\Gamma(\tfrac{1}{2}, 8\pi(n+\tfrac{1}{4})^2 y  )	
	\\
	&=		
	\frac{(1-\z{c}{a})\z{c}{-a}}{2\sqrt{\pi}}
	\sum_{k=0}^{c-1}
		(-1)^k (\z{c}{-2ak} - \z{c}{2ak+a})
	\sum_{n=-\infty}^\infty
		(-1)^{cn}
		{\normalfont\mbox{sgn}}(cn+k+\tfrac{1}{4})
		q^{-2(cn+k+\frac{1}{4})^2 } 
		\Gamma(\tfrac{1}{2}, 8\pi(cn+k+\tfrac{1}{4})^2 y  )	
	\\
	&=		
	\frac{(1-\z{c}{a})\z{c}{-a}}{2\sqrt{\pi}}
	\sum_{k=0}^{c-1}
		(-1)^k (\z{c}{-2ak} - \z{c}{2ak+a})
	\sum_{n=-\infty}^\infty
		(-1)^{cn}
		{\normalfont\mbox{sgn}}(n+\tfrac{1+4k}{4c})
		q^{-2c^2(n+\frac{1+4k}{4c})^2 } 
	\int_{8\pi c^2(n+\frac{1+4k}{4c})^2 y}^\infty		
		\frac{e^{-t}}{\sqrt{t}} dt		
	\\
	&=		
	-ic(1-\z{c}{a})\z{c}{-a}
	\sum_{k=0}^{c-1}
		(-1)^k (\z{c}{-2ak} - \z{c}{2ak+a})
	\sum_{n=-\infty}^\infty
		(-1)^{cn}  (n+\tfrac{1+4k}{4c})
	\int_{-\overline{\tau}}^{i\infty}		
		\frac{e^{4\pi iwc^2(n+\frac{1+4k}{4c})^2 } }  
		{\sqrt{  -i(w+\tau) }} dw		
	\\
	&=		
	-ic e^{-\frac{\pi i}{4}}(1-\z{c}{a})\z{c}{-a}
	\sum_{k=0}^{c-1}
		(\z{c}{-2ak} - \z{c}{2ak+a})
	\int_{-\overline{\tau}}^{i\infty}
	\frac{ g_{\frac{1+4k}{4c}, \frac{c}{2}}(4c^2w) }
	{\sqrt{  -i(w+\tau) }} dw	
	\\
	&=		
	-\frac{i e^{-\frac{\pi i}{4}}(1-\z{c}{a})\z{c}{-a}}{2}
	\sum_{k=0}^{c-1}
		(\z{c}{-2ak} - \z{c}{2ak+a})
	\int_{-\overline{4c^2\tau}}^{i\infty}
	\frac{ g_{\frac{1+4k-2c}{4c}+\frac{1}{2}, \frac{c-1}{2}+\frac{1}{2}}(z) }
	{\sqrt{  -i(z+4c^2\tau) }} dz	
	\\	
	&=		
	\frac{i^{1-c} (1-\z{c}{a})\z{c}{-a}}{2}
	\sum_{k=0}^{c-1}
		(-1)^k (\z{c}{-2ak} - \z{c}{2ak+a})
		q^{-\frac{(1+4k-2c)^2}{8} }  
		R( c(1+4k-2c)\tau + \tfrac{1-c}{2}; 4c^2\tau)
.
\end{align*}
We note that in the second to last equality we have used
Theorem 1.16 of \cite{Zwegers}, which is valid since
$ -\frac{1}{2} < \frac{1+4k-2c}{4c} < \frac{1}{2}$
for $0\le k \le c-1$.
\end{proof}

We now see the definitions of $\mathcal{S}(k,c;\tau)$ and 
$\widetilde{\mathcal{S}}(k,c;\tau)$ are
well motivated. We recall that Proposition \ref{PropositionSTildeToMuTilde}
states
\begin{align*}
\widetilde{\mathcal{}S}(k;c;\tau)
&=
	q^{-\frac{(1+4k-2c)^2}{8}}
	\tmu\left( (1+4k-c)c\tau , c^2\tau + \frac{(c-1)}{2}  ; 4c^2\tau \right)
.
\end{align*}
From this it is apparent that we have chosen $\widetilde{\mathcal{S}}(k,c;\tau)$
so that the non-holomorphic part is 
$\frac{i}{2}q^{-\frac{(1+4k-2c)^2}{8}}R( c(1+4k-2c)\tau + \tfrac{1-c}{2}; 4c^2\tau)$,
which matches with the terms of the non-holomorphic part 
of $\widetilde{\mathcal{R}2}(a,c;\tau)$.
Furthermore, $\z{c}{-2ak}-\z{c}{2ak+a}=0$ when $1+4k-c\equiv 0\pmod{4c}$.

\begin{proposition}\label{PropTransformationForSTilde}
Suppose $k$ and $c$ are integers, $c>0$, and $1+4k-c\not\equiv0\pmod{4c}$.
If $B\in\Gamma_0(\gcd(c,2)4c^2)\cap\Gamma_1(4c)$,
then
\begin{align*}
\widetilde{\mathcal{S}}(k,c;B\tau)
&=
\frac{\nu(B_2)}{\nu(B)\nu(B_4)}
\sqrt{\gamma\tau+\delta}\widetilde{\mathcal{S}}(k,c;\tau)
.
\end{align*}
\end{proposition}
\begin{proof}
To begin we let 
$$
\widehat{G}(k,c;\tau)
=
q^{-\frac{(3c-4k-1)(c-4k-1)}{8}}
\widehat{A}_1\left( (1+4k-c)c\tau, c^2\tau + \frac{(c-1)}{2};4c^2\tau  \right),
$$
so that
\begin{align*}
\widetilde{\mathcal{S}}(k,c;B\tau)
&=
\left\{ \begin{array}{ll}
	\frac{\eta(2c^2\tau)}{\eta(c^2\tau)\eta(4c^2\tau)}\widehat{G}(k,c;\tau)
	& \mbox{ if $c$ is odd,}
	\\[1ex]
	\frac{\eta(c^2\tau)}{\eta(2c^2\tau)^2}\widehat{G}(k,c;\tau)
	& \mbox{ if $c$ is even.}
\end{array}\right.
\end{align*}
Using \eqref{EqEllipticTransformationAl}, on $\Gamma_0(4c^2)$ we have that
\begin{align*}
&\widehat{G}(k,c;B\tau)
\\
&=
	\exp\left( -\frac{\pi iB\tau}{4} (3c-4k-1)(c-4k-1) \right)
	\\&\quad\cdot
	\widehat{A}_1\left(
		(1+4k-c)c\frac{(\alpha\tau+\beta)}{(\gamma\tau+\delta)},
		\frac{c^2(\alpha\tau+\beta)}{(\gamma\tau+\delta)}+\frac{(c-1)}{2}; 
		B_{4c^2}(4c^2\tau)	
	\right)
\\
&=
	(\gamma\tau+\delta)
	\exp\left( -\frac{\pi iB\tau}{4} (3c-4k-1)(c-4k-1) \right)
	\\&\quad\cdot
	\exp\left( 
		\frac{\pi i\gamma
			\left(
				-(1+4k-c)^2c^2(\alpha\tau+\beta)^2
				+2(1+4k-c)c(\alpha\tau+\beta)\left( c^2(\alpha\tau+\beta)+\frac{(c-1)(\gamma\tau+\delta)}{2}  \right)
			\right)
		}
		{4c^2(\gamma\tau+\delta)}
	\right)
	\\&\quad\cdot
	\widehat{A}_1\left(
		(1+4k-c)c(\alpha\tau+\beta),
		c^2(\alpha\tau+\beta)+\frac{(c-1)(\gamma\tau+\delta)}{2}; 
		4c^2\tau	
	\right)
\\
&=
	\exp\left( -\frac{\pi i\alpha\beta}{4}(3c-4k-1)(c-4k-1)  \right)
	q^{-\frac{\alpha^2}{8}(3c-4k-1)(c-4k-1) + \frac{\alpha\gamma}{8c}(1+4k-c)(c-1)}
	\\&\quad\cdot
	(\gamma\tau+\delta)
	\widehat{A}_1\left(
		(1+4k-c)c(\alpha\tau+\beta),
		c^2(\alpha\tau+\beta)+\frac{(c-1)(\gamma\tau+\delta)}{2}; 
		4c^2\tau	
	\right)
.	
\end{align*}
Now by \eqref{EqEllipticTransformationAl} we have that
\begin{align*}
&\widehat{A}_1\left(
	(1+4k-c)c(\alpha\tau+\beta),
	c^2(\alpha\tau+\beta)+\frac{(c-1)(\gamma\tau+\delta)}{2}; 
	4c^2\tau	
\right)
\\
&=
	\widehat{A}_1\left(
		(1+4k-c)c\tau +\tfrac{(1+4k-c)(\alpha-1)}{4c}4c^2\tau +  (1+4k-c)c\beta	,
		c^2\tau + \tfrac{(c-1)}{2}  + \left(\tfrac{(\alpha-1)}{4}+\tfrac{(c-1)\gamma}{8c^2} \right)4c^2\tau   
		+c^2\beta 
		\right.\\&\quad\left.	
		+ \tfrac{(c-1)(\delta-1)}{2}; 
		4c^2\tau	
	\right)
\\
&=
	(-1)^{\frac{(1+4k-c)(\alpha-1)}{4c}}
	\exp\left( 2\pi i(1+4k-c)c\tau\left( 
		\frac{(1+4k-c)(\alpha-1)}{4c}  	-\frac{(\alpha-1)}{4}	-\frac{(c-1)\gamma}{8c^2}
	\right)\right)
	\\&\quad\cdot
	\exp\left( -2\pi i\left(c^2\tau+\frac{(c-1)}{2}\right)\frac{(1+4k-c)(\alpha-1)}{4c} \right)
	\\&\quad\cdot
	\exp\left( 8\pi ic^2\tau\left(
		\frac{(1+4k-c)^2(\alpha-1)^2}{32c^2} 
		-
		\frac{(1+4k-c)(\alpha-1)}{4c}\left( \frac{(\alpha-1)}{4}+\frac{(c-1)\gamma}{8c^2}  \right)
	\right)\right)
	\\&\quad\cdot
	\widehat{A}_1\left( (1+4k-c)c\tau, c^2\tau + \tfrac{(c-1)}{2}; 4c^2\tau	\right)
\\
&=
	(-1)^{\frac{(1+4k-c)(\alpha-1)}{4}}
	q^{ -\frac{(1-\alpha^2)}{8}(3c-4k-1)(c-4k-1) - \frac{\alpha\gamma}{8c}(1+4k-c)(c-1)   }
	\widehat{A}_1\left( (1+4k-c)c\tau, c^2\tau + \tfrac{(c-1)}{2}; 4c^2\tau	\right)
,
\end{align*}
and so
\begin{align*}
\widehat{G}(k,c;B\tau)
&=
(-1)^{\frac{(1+4k-c)(\alpha-1)}{4}}
\exp\left( -\frac{\pi i\alpha\beta}{4}(3c-4k-1)(c-4k-1)   \right)
(\gamma\tau+\delta)
\widehat{G}(k,c;\tau)
\\
&=
\exp\left( -\frac{\pi i\alpha\beta}{4}(3c-4k-1)(c-4k-1)   \right)
(\gamma\tau+\delta)
\widehat{G}(k,c;\tau)
.
\end{align*}

We now must separately consider the cases when $c$ is odd or even. When $c$
is odd, $(3c-4k-1)(c-4k-1)$ is divisible by $8$. As such,
\begin{align*}
\widehat{G}(k,c;B\tau)
&=
(\gamma\tau+\delta)
\widehat{G}(k,c;\tau)
.
\end{align*}
By Theorem 1.64 of \cite{Ono1},
$\frac{\eta(2\tau)\eta(c^2\tau)\eta(4c^2\tau)}{\eta(\tau)\eta(4\tau)\eta(2c^2\tau)}$
is a modular function on $\Gamma_0(4c^2)$ with trivial multiplier. Thus
\begin{align*}
\frac{\nu(B_2)}{\nu(B)\nu(B_4)}
&=
\frac{\nu(B_{2c^2})}{\nu(B_{c^2})\nu(B_{4c^2})},
\end{align*}
and the case when $c$ is odd follows.

When $c$ is even, we instead have that
\begin{align*}
\widehat{G}(k,c;B\tau)
&=
\exp\left( -\frac{\pi i\beta}{4}(3c^2+1)   \right)
(\gamma\tau+\delta)
\widehat{G}(k,c;\tau)
.
\end{align*}
A direct, but lengthy, calculation with \eqref{EqEtaMultipler} reveals that
on $\Gamma_0(8c^2)\cap \Gamma(4c)$ we have
\begin{align*}
\frac{\nu(B)\nu(B_4)\nu(B_{c^2})}{\nu(B_2)\nu(B_{2c^2})^2}
&=
\exp\left( \frac{\pi i\beta(1-c^2)}{4}  \right)
.
\end{align*}
With this the case when $c$ is even follows.
\end{proof}

\begin{proposition}\label{PropositionSTildeHarmonicMaassForm}
Suppose $k$ and $c$ are integers, $c>0$, and $1+4k-c\not\equiv0\pmod{4c}$. 
Then $\widetilde{\mathcal{S}}(k,c;\tau)$ 
is a harmonic Maass form on $\Gamma_0(\gcd(c,2)4c^2)\cap\Gamma_1(4c)$ with multiplier
$\frac{\nu(B_2)}{\nu(B)\nu(B_4)}$.
\end{proposition}
\begin{proof}
We first verify that $\widetilde{\mathcal{S}}(k,c;\tau)$ is annihilated by $\Delta_{\frac{1}{2}}$.
For this we need only verify that
\begin{align*}
	\sqrt{y}\frac{\partial}{\partial\overline{\tau}}
	q^{ -\frac{(1+4k-2c)^2}{8}}
	R( (1+4k-c)c\tau + \tfrac{(1-c)}{2} ; 4c^2\tau  )
\end{align*}
is holomorphic in $\overline{\tau}$.
However this follows from Lemma 1.8 of \cite{Zwegers}.

We next verify that $\widetilde{\mathcal{S}}(k,c;\tau)$ 
has at worst linear exponential growth at the
cusps. For this, suppose $\alpha/\gamma$ is in reduced form. We then
take $B=\SmallMatrix{\alpha}{\beta}{\gamma}{\delta}\in\SLTwo$.
We set $g=\gcd(4c^2,\gamma)$, $m=4c^2\alpha/g$, and $r=\gamma/g$.
Since $\gcd(m,r)=1$ we can take
$L=\SmallMatrix{m}{n}{r}{s}\in\SLTwo$, and set
\begin{align*}
 	C
 	&=
 		L^{-1}\TwoTwoMatrix{4c^2\alpha}{4c^2\beta}{\gamma}{\delta}
	=
		\TwoTwoMatrix{g}{4c^2\beta s - \delta n}{0}{4c^2/g}
.		
\end{align*}
We note that $j(L,C\tau) = \frac{g}{4c^2}(\gamma\tau+\delta)$.
Using \eqref{EqTildeMuModular} we compute that
\begin{align*}
&(\gamma\tau+\delta)^{-\frac{1}{2}}\widetilde{S}(k,c;B\tau)
\\
&=
(\gamma\tau+\delta)^{-\frac{1}{2}}
\exp\left( -\frac{\pi iLC\tau(1+4k-2c)^2}{16c^2} \right)
\tmu\left( \frac{(1+4k-2c)LC\tau}{4c}, \frac{LC\tau}{4}+\frac{(c-1)}{2}; LC\tau  \right)
\\
&=
\epsilon_1\cdot
\exp\left( \pi iC\tau\left(
	-\frac{m^2(1+4k-2c)^2}{16c^2}
	+\frac{mr(1+4k-2c)(c-1)}{4c}
	-\frac{r^2(c-1)^2}{4}
\right)\right)
\\&\quad\cdot
\tmu\left(
	\frac{(1+4k-c)(mC\tau +n)}{4c}, \frac{(mC\tau+n)}{4} + \frac{(c-1)(rC\tau+s)}{2}; C\tau
\right)
,
\end{align*}
where $\epsilon_1$ is a non-zero constant.
Since $e^{2\pi iC\tau} = \epsilon_2 q^{\frac{g^2}{4c^2}}$, where $\epsilon_2$
is a fixed root of unity,
we see that
$(\gamma\tau+\delta)^{-\frac{1}{2}}\widetilde{\mathcal{S}}(k,c;B\tau)$ 
has at worst linear
exponential growth as $y\rightarrow \infty$ by the definition of $\tmu$.
\end{proof}

Since $\widetilde{\mathcal{R}2}(a,c;\tau)$ and 
$\widetilde{\mathcal{S}}(k,c;\tau)$ have the same weight and multiplier
on $\Gamma_0(\gcd(c,2))4c^2)\cap\Gamma_1(4c)$,
taking the correct difference
of these terms will cancel the non-holomorphic parts and result in a modular form.
This establishes parts (2), (3), and (4) of Theorem \ref{TheoremModularity}.

To accommodate the generalized eta-quotients in the dissection of
$R2(\z{7}{};q)$ in Theorem \ref{TheoremDissection}, 
we use the following.
For $N$ a positive integer and $\rho$ an integer with $N\nmid\rho$,
we let
$$
	f_{N,\rho}(\tau)
	=
		q^{\frac{(N-2\rho)^2}{8N}}
		\aqprod{q^\rho,q^{N-\rho},q^N}{q^N}{\infty}.
$$ 
We note that $f_{N,\rho}(\tau)=f_{N,-\rho}(\tau)=f_{N,\rho+N}(\tau)$,
and so Lemma 2.1 of \cite{Biagioli} can be stated as 
\begin{align}\label{EqBiagioliTransform}
	f_{N,\rho}(B\tau)
	&=
		(-1)^{\rho\beta + \Floor{\frac{\rho\alpha}{N}} + \Floor{\frac{\rho}{N}} }
		\exp\left(\frac{\pi i\alpha\beta\rho^2}{N}\right)
		\nu( B_N )^{3}
		\sqrt{\gamma\tau+\delta}\,		
		f_{N,\rho}(\tau)
,
\end{align}
for $B=\SmallMatrix{\alpha}{\beta}{\gamma}{\delta}\in\Gamma_1(N)$.

\begin{proposition}\label{PropositionGeneralTranformationForGetaQuotient}
Suppose $c>0$ is an integer and
\begin{align*}
	f(\tau)
	&=
		\eta(4c^2\tau)^{r_0}
		\prod_{k=1}^{2c} f_{4c^2, ck}(\tau)^{r_k}
.
\end{align*}
If $B=\SmallMatrix{\alpha}{\beta}{\gamma}{\delta}\in\Gamma_0(4c^2)\cap\Gamma_1(4c)$,
then
\begin{align*}
	f(B\tau)
	&=
	(-1)^{ (c\beta + \frac{(\alpha-1)}{4c} + \frac{\beta(\alpha-1)}{4})S + \beta T }
	\exp\left( \tfrac{\pi i\beta S}{4}  \right)
	\nu( B_{4c^2} ) ^{ r_0 + 3R}
	(\gamma\tau+\delta)^{ \frac{r_0+R}{2}}
	f(\tau)
	,
\end{align*}
where
\begin{align*}
	R 
	&=
		\sum_{i=1}^{2c} r_i
	,\qquad\qquad
	S 
	=
		\sum_{\substack{i=1,\\ i\equiv 1\pmod{2} }}^{2c} r_i
	,\qquad\qquad
	T
	=
		\sum_{\substack{i=1,\\ i\equiv 2\pmod{4} }}^{2c} r_i		
.
\end{align*}
\end{proposition}
\begin{proof}
By (\ref{EqBiagioliTransform}) we have
\begin{align*}
	f_{4c^2,ck}(B\tau)
	&=
		(-1)^{kc\beta + \Floor{\frac{kc\alpha}{4c^2}} + \Floor{\frac{kc}{4c^2}} }
		\exp\left( \tfrac{ \pi i\alpha\beta(kc)^2 }{4c^2}  \right)
		\nu( B_{4c^2} )^{3}
		\sqrt{\gamma\tau+\delta}
		f_{4c^2,ck}(\tau)
	.
\end{align*}
Since $\alpha\equiv 1\pmod{4c}$ and $1\le k\le 2c$, we can write
$\alpha = 1 + \frac{(\alpha-1)}{4c}4c$ to deduce that
\begin{align*}
	\Floor{\frac{kc\alpha}{4c^2}}
	&=
	\frac{(\alpha-1)k}{4c}
	.
\end{align*}
Thus
\begin{align*}
	(-1)^{kc\beta + \Floor{\frac{kc\alpha}{4c^2}} + \Floor{\frac{kc}{4c^2}} }
	&=
	(-1)^{ k (c\beta + \frac{(\alpha-1)}{4c} ) }	
	.
\end{align*}
Next we have that
\begin{align*}
	\exp\left( \tfrac{\pi i\alpha\beta(kc)^2}{4c^2} \right)
	&=
	(-1)^{k \tfrac{\beta(\alpha-1)}{4}}
	\exp\left( \tfrac{\pi i\beta k^2}{4} \right)
.
\end{align*}
So in fact
\begin{align*}
	f_{4c^2,ck}(B\tau)
	&=
	(-1)^{k( c\beta + \frac{(\alpha-1)}{4c} + \frac{\beta(\alpha-1)}{4} )}
	\exp\left( \tfrac{\pi i \beta k^2}{4} \right)
	\nu( B_{4c^2} )^3
	\sqrt{\gamma\tau+\delta}
	f_{4c^2,ck}(\tau)
	.
\end{align*}
Thus
\begin{align*}
	f(B\tau)
	&=
		\left(
			\prod_{k=1}^{2c}
				(-1)^{r_k k( c\beta + \frac{(\alpha-1)}{4c} + \frac{\beta(\alpha-1)}{4} ) }	
				\exp\left(  \tfrac{\pi i\beta k^2 r_k}{4}  \right)
		\right)	
		\nu( B_{4c^2} )^{r_0 + 3R}
		(\gamma\tau+\delta)^{\frac{r_0+R}{2}}
		f(\tau)
	\\
	&=
		(-1)^{ (c\beta + \frac{(\alpha-1)}{4c} + \frac{\beta(\alpha-1)}{4} )S }	
		(-1)^{\beta T}
		\exp\left(  \tfrac{\pi i\beta S}{4}  \right)
		\nu( B_{4c^2} )^{r_0 + 3R}
		(\gamma\tau+\delta)^{\frac{r_0+R}{2}}
		f(\tau)
	.		
\end{align*}
\end{proof}

\begin{corollary}\label{CorollaryProductsModular}
Suppose $c>0$ is an odd integer,
\begin{align*}
	f(\tau)
	&=
		\eta(4c^2\tau)^{r_0}
		\prod_{k=1}^{2c} f_{4c^2, ck}(\tau)^{r_k}
	,
\end{align*}
and
\begin{align*}
	R 
	&=
		\sum_{i=1}^{2c} r_i
	,\qquad\qquad
	S 
	=
		\sum_{\substack{i=1,\\ i\equiv 1\pmod{2} }}^{2c} r_i
	,\qquad\qquad
	T
	=
		\sum_{\substack{i=1,\\ i\equiv 2\pmod{4} }}^{2c} r_i		
.
\end{align*}
Suppose that $r_0+R=1$, $R\equiv -2 \pmod{12}$, $S\equiv -1 \pmod{8}$,
and $T\equiv 0\pmod{2}$.
If $B=\SmallMatrix{\alpha}{\beta}{\gamma}{\delta}\in\Gamma_0(4c^2)\cap\Gamma_1(4c)$,
then
\begin{align*}
	f(B\tau)
	&=
	(-1)^{ \beta + \frac{(\alpha-1)}{4} + \frac{\beta(\alpha-1)}{4}}
	\exp\left( -\tfrac{\pi i\beta}{4}  \right)
	\nu( B_{4} )^{ -3}
	\sqrt{\gamma\tau+\delta}
	f(\tau)
	.
\end{align*}
Furthermore if
\begin{align*}
	g(\tau)
	&=
		\frac{\eta(4\tau)\eta(\tau)}{\eta(2\tau)}f(\tau)
	,&
	h(\tau)
	&=
		\frac{\eta(4c^2\tau)\eta(c^2\tau)}{\eta(2c^2\tau)}f(\tau)
	,
\end{align*}
then $g(\tau)$ and $h(\tau)$ are weight $1$ weakly holomorphic modular forms on $\Gamma_0(4c^2)\cap\Gamma_1(4c)$.
\end{corollary}
\begin{proof}
The only simplification for the transformation formula of $f(\tau)$
that does not follow immediately from
Proposition \ref{PropositionGeneralTranformationForGetaQuotient}
and the conditions on $R$, $S$, and $T$ is that
$\nu( B_{4c^2} )^{r_0+3R} = \nu( B_{4} )^{ -3}$.
However this follows from the fact that since
$c$ is odd, we have that $\frac{\eta(4c^2\tau)^3}{\eta(4\tau)^3}$
is a modular function on $\Gamma_0(4c^2)$
by Theorem 1.64 of \cite{Ono1}.
From this we have that $\nu( B_{4c^2} )^{3} = 	\nu( B_{4} )^{ 3}$.
But we also know $\nu( B_{4c^2} )^{3}$ is a $24^{th}$ root of unity and
$r_0+3R = 1+2R\equiv -3 \pmod{24}$, so that
$\nu( B_{4c^2} )^{r_0+3R} = \nu( B_{4c^2} )^{-3} = 	\nu( B_{4} )^{ -3}$.

Nothing that $\frac{\eta(2\tau)}{\eta(\tau)\eta(4\tau)}=f_{4,1}(\tau)$, we use
\eqref{EqBiagioliTransform} to deduce that, on $\Gamma_1(4)$,
\begin{align*}
\frac{\nu(B_2)}{\nu(B)\nu(B_4)}
&=
(-1)^{ \beta + \frac{(\alpha-1)}{4} + \frac{\beta(\alpha-1)}{4}}
\exp\left( \tfrac{\pi i\beta}{4}  \right)
\nu( B_{4} )^{ 3}.
\end{align*}
As such, $g(\tau)$ has trivial multiplier on $\Gamma_0(4c^2)\cap\Gamma_1(4c)$.
As noted in the proof of Proposition \ref{PropTransformationForSTilde},
on $\Gamma_0(4c^2)$ we have that
\begin{align*}
\frac{\nu(B_2)}{\nu(B)\nu(B_4)}
&=
\frac{\nu(B_{2c^2})}{\nu(B_{c^2})\nu(B_{4c^2})}
,
\end{align*}
and so $h(\tau)$ also has trivial multiplier.
\end{proof}

We could easily deduce an analogous result for even $c$, however we omit this
as we would not make use of it.

\section{Orders at Cusps}


We recall for a modular form $f$ on some congruence subgroup $\Gamma$, 
the invariant order at
$i\infty$ is the least power of $q$ appearing in the $q$-expansion at 
$i\infty$. That is, if
\begin{align*}
	f(\tau) &= \sum_{m=m_0}^\infty a(m)\exp(2\pi i \tau m/N )
,
\end{align*}
and $a(m_0)\not=0$, then the invariant order is $m_0/N$. For a modular form,
this is always a finite number.
For a harmonic weak Maass form, we cannot take such an expansion,
however we can do so for the holomorphic part. 
As such we define the invariant order at $i\infty$ of a harmonic weak Maass form to be the
least exponent of $q$ appearing in the holomorphic part.
If $f$ is a modular form of weight $k$,
$\gcd(\alpha,\gamma)=1$, and 
$B=\SmallMatrix{\alpha}{\beta}{\gamma}{\delta}\in\SLTwo$,
then the invariant order of $f$ at the cusp $\frac{\alpha}{\gamma}$ is the invariant
order at $i\infty$ of $j(B,\tau)^{-k}f(B\tau)$. In the same fashion, 
if $f$ is a harmonic weak Maass form, then we define the invariant order of $f$ 
at the cusp $\frac{\alpha}{\gamma}$ to be the invariant
order at $i\infty$ of $j(B,\tau)^{-k}f(B\tau)$. This value 
is independent of the choice of $B$.

To compute bounds on the orders of the functions handled in this article,
we use Corollary 6.2 from \cite{JenningsShaffer2}.
For a real number $w$, we let $\Floor{w}$ denote the greatest 
integer less than or equal to $w$ and $\Fractional{w}$ the fractional part of 
$w$. That is, 
$w=\Floor{w}+\Fractional{w}$, $\Floor{w}\in\mathbb{Z}$, and 
$0\le\Fractional{w}<1$.

\begin{proposition}\label{PropTMuOrders}
If $f(\tau)=q^{\alpha}\tmu(u_1\tau+u_2,w_1\tau+w_2;\tau)$ is a harmonic weak Maass form,
with $u_i,w_i\in\mathbb{R}$, then the lowest power of 
$q$ appearing in the expansion of the holomorphic part of
$f(\tau)$ is at least $\alpha+\tilde{\nu}(u_1,w_1)$, where
\begin{align*}
	\widetilde{\nu}(u,w)
	&=
		\frac{1}{2}\left( \Floor{u}-\Floor{w} \right)^2
		+
		\left(\Floor{u}-\Floor{w}\right)\left(\Fractional{u}-\Fractional{w}\right)		
		+
		k( u , w )
	,
	\\
	k(u, w)
	&=
		\left\{
		\begin{array}{cc}
			\vspace{5pt}
			\nu(\Fractional{u},\Fractional{w}) 
				& \mbox{ if } \Fractional{u}-\Fractional{w}\not=\pm\frac{1}{2} ,
			\\
			\min\left( \frac{1}{8},\nu(\Fractional{u},\Fractional{w}) \right)  
				&\mbox{ if } \Fractional{u}-\Fractional{w}=\pm\frac{1}{2} ,	
		\end{array}\right.
	\\
	\nu(u,w)
	&=
		\left\{
		\begin{array}{cc}
			\vspace{5pt}			
			\frac{u+w}{2}-\frac{1}{8} & \mbox{ if } u+w \le 1 ,
			\\
			\frac{7}{8}-\frac{u+w}{2} & \mbox{ if } u+w > 1.
		\end{array}\right.
\end{align*}
\end{proposition}

In the proofs of Propositions \ref{PropR2HarmonicMaassForm} and 
\ref{PropositionSTildeHarmonicMaassForm} we computed 
$\widetilde{\mathcal{R}2}(a,c;B\tau)$ and 
$\widetilde{S}(k,c;B\tau)$, for $B\in\SLTwo$, in terms of
$\tmu$. With these formulas, along with Proposition \ref{PropTMuOrders},
we can immediately read off bounds on the orders at cusps.

\begin{proposition}\label{PropositionOrderForR2Tilde}
Suppose $a$ and $c$ are integers, $c>0$, and $c\nmid 2a$.
If $\alpha$ and $\gamma$ are integers with $\gcd(\alpha,\gamma)=1$,
then the invariant order of 
$\widetilde{\mathcal{R}2}(a,c;\tau)$ at the cusp $\frac{\alpha}{\gamma}$ is at least
\begin{align*}
	\frac{g^2}{4}
	\left(  
		-\frac{m^2}{32} 
		- \frac{2a^2 r^2}{c^2}  
		+
		\min\left(
			-\frac{amr}{2c} + \widetilde{\nu}\left( \tfrac{2ar}{c}, -\tfrac{m}{4}  \right),
			\frac{amr}{2c} + \widetilde{\nu}\left( \tfrac{2ar}{c}, \tfrac{m}{4}  \right)
		\right)
	\right)
,
\end{align*}
where $g=\gcd(4,\gamma)$, $m=4\alpha/g$, and $r=\gamma/g$.
\end{proposition}

\begin{proposition}\label{PropositionOrderForSTilde}
Suppose $k$ and $c$ are integers, $c>0$, and $1+4k-c\not\equiv0\pmod{4c}$.
If $\alpha$ and $\gamma$ are integers with $\gcd(\alpha,\gamma)=1$,
then the invariant order of 
$\widetilde{\mathcal{S}}(k,c;\tau)$ at the cusp $\frac{\alpha}{\gamma}$ is at least
\begin{align*}
\frac{g^2}{4c^2}\left(
	-\frac{m^2(1+4k-2c)^2}{32c^2}
	+\frac{mr(1+4k-2c)(c-1)}{8c}	
	-\frac{r^2(c-1)^2}{8}
	+\widetilde{\nu}\left( \frac{(1+4k-c)m}{4c}, \frac{m}{4}+\frac{(c-1)r}{2} \right)	
\right)
,
\end{align*}
where $g=\gcd(4c^2,\gamma)$, $m=4c^2\alpha/g$, and $r=\gamma/g$.
\end{proposition}

The following is Lemma 3.2 of \cite{Biagioli}.
\begin{proposition}\label{PropProductOrders}
Suppose $\alpha$ and $\gamma$ are integers with $\gcd(\alpha,\gamma)=1$,
then the invariant order of $f_{N,\rho}(\tau)$ at the cusp 
$\frac{\alpha}{\gamma}$ is
$
	\frac{\gcd(N,\gamma)^2}{2N}
	\left(
		\Fractional{\frac{\alpha\rho}{\gcd(N,\gamma)} } 
		-
		\frac{1}{2}
	\right)^2
$.
\end{proposition}

\section{Proof Of Theorem \ref{TheoremDissection} }

To begin we give the definitions of the remaining 
$R2_i(q)$ from the statement of Theorem 2.2. We let
\begin{align*}
	&R2_1(q)
	=
	\tfrac{A(-8,1,-3)q^{-3}J_{0}J_{7}J_{8}J_{10}^{2}}
		{J_{1}^{2}J_{2}^{2}J_{3}^{2}J_{4}^{2}J_{5}^{2}J_{11}^{2}J_{12}^{2}J_{13}^{2}J_{14}^{2}}
	+\tfrac{A(10,-7,2)q^{-3}J_{0}J_{7}^{2}J_{8}}
		{J_{1}^{2}J_{2}^{2}J_{3}^{2}J_{4}^{2}J_{5}^{2}J_{9}J_{11}^{2}J_{13}^{2}J_{14}^{2}}
	+\tfrac{A(0,-1,0)q^{-3}J_{0}J_{7}^{2}J_{8}^{2}}
		{J_{1}^{2}J_{2}^{2}J_{3}^{2}J_{4}^{2}J_{5}^{2}J_{9}^{2}J_{10}J_{11}J_{12}J_{13}^{2}J_{14}}
	\\&
	+\tfrac{A(-5,11,0)q^{-3}J_{0}J_{7}^{2}J_{8}^{2}}
		{J_{1}^{2}J_{2}^{2}J_{3}^{2}J_{4}^{2}J_{5}^{2}J_{9}J_{10}^{2}J_{11}^{2}J_{12}J_{13}^{2}}
	+\tfrac{A(4,10,3)q^{-3}J_{0}J_{6}J_{8}J_{9}^{2}J_{10}}
		{J_{1}^{2}J_{2}^{2}J_{3}^{2}J_{4}^{2}J_{5}^{2}J_{7}J_{11}^{2}J_{12}^{2}J_{13}^{2}J_{14}^{2}}
	+\tfrac{A(0,-1,0)q^{-3}J_{0}J_{6}J_{8}J_{10}}
		{J_{1}^{2}J_{2}^{2}J_{3}^{2}J_{4}^{2}J_{5}^{2}J_{11}^{2}J_{12}^{2}J_{13}J_{14}^{2}}
	+\tfrac{A(1,2,1)q^{-3}J_{0}J_{6}J_{8}J_{9}}
		{J_{1}^{2}J_{2}^{2}J_{3}^{2}J_{4}^{2}J_{5}^{2}J_{10}J_{11}^{2}J_{13}^{2}J_{14}^{2}}
	\\&	
	+\tfrac{A(-1,1,0)q^{-3}J_{0}J_{6}J_{8}^{2}}
		{J_{1}^{2}J_{2}^{2}J_{3}^{2}J_{4}^{2}J_{5}^{2}J_{9}J_{10}J_{12}J_{13}^{2}J_{14}^{2}}
	+\tfrac{A(8,-15,1)q^{-3}J_{0}J_{6}J_{8}^{2}}
		{J_{1}^{2}J_{2}^{2}J_{3}^{2}J_{4}^{2}J_{5}^{2}J_{10}^{2}J_{11}J_{12}J_{13}^{2}J_{14}}
	+\tfrac{A(-2,-17,-1)q^{-3}J_{0}J_{6}J_{7}J_{10}}
		{J_{1}^{2}J_{2}^{2}J_{3}^{2}J_{4}^{2}J_{5}^{2}J_{9}J_{11}J_{12}J_{13}^{2}J_{14}^{2}}
	+\tfrac{A(-5,19,1)q^{-3}J_{0}J_{6}J_{7}}
		{J_{1}^{2}J_{2}^{2}J_{3}^{2}J_{4}^{2}J_{5}^{2}J_{11}^{2}J_{12}J_{13}^{2}J_{14}}
	\\&	
	+\tfrac{A(4,-11,0)q^{-3}J_{0}J_{6}J_{7}J_{8}}
		{J_{1}^{2}J_{2}^{2}J_{3}^{2}J_{4}^{2}J_{5}^{2}J_{9}^{2}J_{12}^{2}J_{13}^{2}J_{14}}
	+\tfrac{A(1,9,2)q^{-3}J_{0}J_{6}J_{7}J_{8}J_{10}}
		{J_{1}^{2}J_{2}^{2}J_{3}^{2}J_{4}^{2}J_{5}^{2}J_{9}^{2}J_{11}^{2}J_{12}^{2}J_{14}^{2}}
	+\tfrac{A(-11,8,-4)q^{-3}J_{0}J_{6}J_{7}J_{8}}
		{J_{1}^{2}J_{2}^{2}J_{3}^{2}J_{4}^{2}J_{5}^{2}J_{9}J_{10}J_{11}^{2}J_{13}J_{14}^{2}}
	+\tfrac{A(-3,-2,-2)q^{-3}J_{0}J_{6}J_{7}J_{8}}
		{J_{1}^{2}J_{2}^{2}J_{3}^{2}J_{4}^{2}J_{5}^{2}J_{9}J_{10}J_{11}J_{12}^{2}J_{13}^{2}}
	\\&	
	+\tfrac{A(5,-11,0)q^{-3}J_{0}J_{6}J_{7}J_{8}J_{14}}
		{J_{1}^{2}J_{2}^{2}J_{3}^{2}J_{4}^{2}J_{5}^{2}J_{10}^{2}J_{11}^{2}J_{12}^{2}J_{13}^{2}}
	+\tfrac{A(4,2,3)q^{-3}J_{0}J_{6}J_{7}J_{8}^{2}}
		{J_{1}^{2}J_{2}^{2}J_{3}^{2}J_{4}^{2}J_{5}^{2}J_{9}^{2}J_{10}^{2}J_{11}J_{12}J_{13}J_{14}}
	+\tfrac{A(3,-11,0)q^{-2}J_{0}J_{6}J_{7}J_{8}^{2}}
		{J_{1}^{2}J_{2}^{2}J_{3}^{2}J_{4}^{2}J_{5}^{2}J_{10}J_{11}^{2}J_{12}J_{13}^{2}J_{14}^{2}}
	\\&	
	+\tfrac{A(3,3,1)q^{-3}J_{0}J_{6}J_{7}^{2}J_{10}^{2}}
		{J_{1}^{2}J_{2}^{2}J_{3}^{2}J_{4}^{2}J_{5}^{2}J_{8}J_{9}^{2}J_{11}J_{12}^{2}J_{13}^{2}J_{14}}
	+\tfrac{A(-3,-3,-1)q^{-3}J_{0}J_{6}J_{7}^{2}J_{10}}
		{J_{1}^{2}J_{2}^{2}J_{3}^{2}J_{4}^{2}J_{5}^{2}J_{8}J_{9}J_{11}^{2}J_{12}^{2}J_{13}^{2}}
	+\tfrac{A(3,2,1)q^{-3}J_{0}J_{6}J_{7}^{2}}
		{J_{1}^{2}J_{2}^{2}J_{3}^{2}J_{4}^{2}J_{5}^{2}J_{9}^{2}J_{11}^{2}J_{12}J_{13}J_{14}}
	+\tfrac{A(-3,-2,-2)q^{-3}J_{0}J_{6}J_{7}^{2}J_{12}}
		{J_{1}^{2}J_{2}^{2}J_{3}^{2}J_{4}^{2}J_{5}^{2}J_{9}J_{10}^{2}J_{11}^{2}J_{13}^{2}J_{14}}
	\\&	
	+\tfrac{A(-3,11,0)q^{-2}J_{0}J_{6}J_{7}^{2}J_{8}}
		{J_{1}^{2}J_{2}^{2}J_{3}^{2}J_{4}^{2}J_{5}^{2}J_{9}J_{11}^{2}J_{12}^{2}J_{13}^{2}J_{14}}
	+\tfrac{A(-1,0,-1)q^{-2}J_{0}J_{6}J_{7}^{2}J_{8}^{2}}
		{J_{1}^{2}J_{2}^{2}J_{3}^{2}J_{4}^{2}J_{5}^{2}J_{9}^{2}J_{10}J_{11}^{2}J_{12}J_{13}J_{14}^{2}}
	+\tfrac{A(1,-18,-3)q^{-3}J_{0}J_{6}^{2}J_{9}}
		{J_{1}^{2}J_{2}^{2}J_{3}^{2}J_{4}^{2}J_{5}^{2}J_{7}J_{11}J_{12}J_{13}^{2}J_{14}^{2}}
	\\&	
	+\tfrac{A(-5,11,0)q^{-3}J_{0}J_{6}^{2}J_{9}^{2}}
		{J_{1}^{2}J_{2}^{2}J_{3}^{2}J_{4}^{2}J_{5}^{2}J_{7}J_{10}J_{11}^{2}J_{12}J_{13}^{2}J_{14}}
	+\tfrac{A(-4,16,1)q^{-3}J_{0}J_{6}^{2}J_{8}J_{11}}
		{J_{1}^{2}J_{2}^{2}J_{3}^{2}J_{4}^{2}J_{5}^{2}J_{7}J_{9}J_{12}^{2}J_{13}^{2}J_{14}^{2}}
	+\tfrac{A(-1,5,0)q^{-3}J_{0}J_{6}^{2}J_{8}}
		{J_{1}^{2}J_{2}^{2}J_{3}^{2}J_{4}^{2}J_{5}^{2}J_{7}J_{10}J_{12}^{2}J_{13}^{2}J_{14}}
	+\tfrac{A(4,-16,-1)q^{-3}J_{0}J_{6}^{2}J_{8}}
		{J_{1}^{2}J_{2}^{2}J_{3}^{2}J_{4}^{2}J_{5}^{2}J_{7}J_{11}^{2}J_{12}^{2}J_{14}^{2}}
	\\&	
	+\tfrac{A(-2,-1,-1)q^{-3}J_{0}J_{6}^{2}J_{8}J_{9}}
		{J_{1}^{2}J_{2}^{2}J_{3}^{2}J_{4}^{2}J_{5}^{2}J_{7}J_{10}^{2}J_{11}^{2}J_{13}J_{14}^{2}}
	+\tfrac{A(3,3,1)q^{-3}J_{0}J_{6}^{2}J_{9}J_{10}^{2}}
		{J_{1}^{2}J_{2}^{2}J_{3}^{2}J_{4}^{2}J_{5}^{2}J_{8}^{2}J_{11}^{2}J_{12}J_{13}^{2}J_{14}^{2}}
	+\tfrac{A(0,-1,1)q^{-3}J_{0}J_{6}^{2}J_{10}^{2}}
		{J_{1}^{2}J_{2}^{2}J_{3}^{2}J_{4}^{2}J_{5}^{2}J_{8}J_{9}J_{12}^{2}J_{13}^{2}J_{14}^{2}}
	+\tfrac{A(-2,-4,-1)q^{-3}J_{0}J_{6}^{2}J_{10}}
		{J_{1}^{2}J_{2}^{2}J_{3}^{2}J_{4}^{2}J_{5}^{2}J_{8}J_{11}J_{12}^{2}J_{13}^{2}J_{14}}
	\\&	
	+\tfrac{A(3,2,2)q^{-3}J_{0}J_{6}^{2}J_{9}}
		{J_{1}^{2}J_{2}^{2}J_{3}^{2}J_{4}^{2}J_{5}^{2}J_{8}J_{11}^{2}J_{12}^{2}J_{13}^{2}}
	+\tfrac{A(-1,3,-1)q^{-3}J_{0}J_{6}^{2}}
		{J_{1}^{2}J_{2}^{2}J_{3}^{2}J_{4}^{2}J_{5}^{2}J_{10}J_{11}^{2}J_{12}J_{13}J_{14}}
	+\tfrac{A(3,3,1)q^{-3}J_{0}J_{6}^{2}J_{7}J_{10}}
		{J_{1}^{2}J_{2}^{2}J_{3}^{2}J_{4}^{2}J_{5}^{2}J_{8}J_{9}^{2}J_{11}J_{12}^{2}J_{13}J_{14}}
	+\tfrac{A(1,0,1)q^{-2}J_{0}J_{6}^{2}J_{7}J_{8}^{2}}
		{J_{1}^{2}J_{2}^{2}J_{3}^{2}J_{4}^{2}J_{5}^{2}J_{9}^{2}J_{10}^{2}J_{11}^{2}J_{12}J_{14}^{2}}
	\\&	
	+\tfrac{A(-1,0,-1)q^{-3}J_{0}J_{8}^{2}}
		{J_{1}^{2}J_{2}^{2}J_{3}^{2}J_{4}^{2}J_{5}J_{7}J_{11}J_{12}J_{13}^{2}J_{14}^{2}}
	+\tfrac{A(0,-1,0)q^{-3}J_{0}J_{8}^{2}J_{9}}
		{J_{1}^{2}J_{2}^{2}J_{3}^{2}J_{4}^{2}J_{5}J_{7}J_{10}J_{11}^{2}J_{12}J_{13}^{2}J_{14}}
	+\tfrac{A(-1,-1,-1)J_{0}J_{6}J_{7}J_{8}^{2}}
		{J_{1}^{2}J_{2}^{2}J_{3}^{2}J_{4}^{2}J_{9}^{2}J_{10}J_{11}^{2}J_{12}J_{13}^{2}J_{14}^{2}}
	\\&	
	+\tfrac{A(2,1,1)qJ_{0}J_{5}J_{6}^{2}J_{8}^{2}}
		{J_{1}^{2}J_{2}^{2}J_{3}^{2}J_{4}^{2}J_{7}J_{9}J_{10}^{2}J_{11}^{2}J_{12}J_{13}^{2}J_{14}^{2}}
	,\\
	&R2_2(q)
	=
	\tfrac{A(5,-1,5)q^{-3}J_{0}J_{7}J_{8}^{2}}
		{J_{1}^{2}J_{2}^{2}J_{3}^{2}J_{4}^{2}J_{5}^{2}J_{11}^{2}J_{12}J_{13}^{2}J_{14}^{2}}
	+\tfrac{A(0,2,2)q^{-3}J_{0}J_{7}^{2}J_{8}J_{10}^{2}}
		{J_{1}^{2}J_{2}^{2}J_{3}^{2}J_{4}^{2}J_{5}^{2}J_{9}^{2}J_{11}J_{12}^{2}J_{13}^{2}J_{14}^{2}}
	+\tfrac{A(-12,-6,-16)q^{-3}J_{0}J_{7}^{2}J_{8}J_{10}}
		{J_{1}^{2}J_{2}^{2}J_{3}^{2}J_{4}^{2}J_{5}^{2}J_{9}J_{11}^{2}J_{12}^{2}J_{13}^{2}J_{14}}
	\\&	
	+\tfrac{A(1,-4,1)q^{-3}J_{0}J_{7}^{2}J_{8}^{2}}
		{J_{1}^{2}J_{2}^{2}J_{3}^{2}J_{4}^{2}J_{5}^{2}J_{9}^{2}J_{11}^{2}J_{12}J_{13}J_{14}^{2}}
	+\tfrac{A(13,8,14)q^{-3}J_{0}J_{7}^{2}J_{8}^{2}J_{12}}
		{J_{1}^{2}J_{2}^{2}J_{3}^{2}J_{4}^{2}J_{5}^{2}J_{9}J_{10}^{2}J_{11}^{2}J_{13}^{2}J_{14}^{2}}
	+\tfrac{A(3,0,6)q^{-3}J_{0}J_{6}J_{8}^{2}J_{9}^{2}}
		{J_{1}^{2}J_{2}^{2}J_{3}^{2}J_{4}^{2}J_{5}^{2}J_{7}J_{10}J_{11}^{2}J_{12}J_{13}^{2}J_{14}^{2}}
	+\tfrac{A(2,4,1)q^{-3}J_{0}J_{6}J_{8}J_{10}}
		{J_{1}^{2}J_{2}^{2}J_{3}^{2}J_{4}^{2}J_{5}^{2}J_{11}J_{12}^{2}J_{13}^{2}J_{14}^{2}}
	\\&	
	+\tfrac{A(3,5,4)q^{-3}J_{0}J_{6}J_{8}J_{9}}
		{J_{1}^{2}J_{2}^{2}J_{3}^{2}J_{4}^{2}J_{5}^{2}J_{11}^{2}J_{12}^{2}J_{13}^{2}J_{14}}
	+\tfrac{A(-7,-2,-11)q^{-3}J_{0}J_{6}J_{8}^{2}}
		{J_{1}^{2}J_{2}^{2}J_{3}^{2}J_{4}^{2}J_{5}^{2}J_{10}J_{11}^{2}J_{12}J_{13}J_{14}^{2}}
	+\tfrac{A(4,-5,9)q^{-3}J_{0}J_{6}J_{7}J_{8}J_{10}}
		{J_{1}^{2}J_{2}^{2}J_{3}^{2}J_{4}^{2}J_{5}^{2}J_{9}^{2}J_{11}J_{12}^{2}J_{13}J_{14}^{2}}
	+\tfrac{A(-11,1,-13)q^{-3}J_{0}J_{6}J_{7}J_{8}}
		{J_{1}^{2}J_{2}^{2}J_{3}^{2}J_{4}^{2}J_{5}^{2}J_{9}J_{10}J_{11}J_{13}^{2}J_{14}^{2}}
	\\&	
	+\tfrac{A(-9,-5,-9)q^{-3}J_{0}J_{6}J_{7}J_{8}}
		{J_{1}^{2}J_{2}^{2}J_{3}^{2}J_{4}^{2}J_{5}^{2}J_{9}J_{11}^{2}J_{12}^{2}J_{13}J_{14}}
	+\tfrac{A(0,-1,1)q^{-3}J_{0}J_{6}J_{7}J_{8}}
		{J_{1}^{2}J_{2}^{2}J_{3}^{2}J_{4}^{2}J_{5}^{2}J_{10}^{2}J_{11}^{2}J_{13}^{2}J_{14}}
	+\tfrac{A(2,3,1)q^{-3}J_{0}J_{6}J_{7}J_{8}^{2}}
		{J_{1}^{2}J_{2}^{2}J_{3}^{2}J_{4}^{2}J_{5}^{2}J_{9}^{2}J_{10}^{2}J_{12}J_{13}^{2}J_{14}}
	+\tfrac{A(-6,-4,-5)q^{-3}J_{0}J_{6}J_{7}J_{8}^{2}}
		{J_{1}^{2}J_{2}^{2}J_{3}^{2}J_{4}^{2}J_{5}^{2}J_{9}^{2}J_{10}J_{11}^{2}J_{12}J_{14}^{2}}
	\\&	
	+\tfrac{A(-3,-2,-2)q^{-3}J_{0}J_{6}J_{7}^{2}J_{10}}
		{J_{1}^{2}J_{2}^{2}J_{3}^{2}J_{4}^{2}J_{5}^{2}J_{8}J_{9}J_{11}^{2}J_{13}^{2}J_{14}^{2}}
	+\tfrac{A(10,9,7)q^{-3}J_{0}J_{6}J_{7}^{2}}
		{J_{1}^{2}J_{2}^{2}J_{3}^{2}J_{4}^{2}J_{5}^{2}J_{9}^{2}J_{11}J_{12}J_{13}^{2}J_{14}}
	+\tfrac{A(3,1,5)q^{-3}J_{0}J_{6}J_{7}^{2}}
		{J_{1}^{2}J_{2}^{2}J_{3}^{2}J_{4}^{2}J_{5}^{2}J_{9}J_{10}J_{11}^{2}J_{12}J_{13}^{2}}
	+\tfrac{A(-11,-6,-14)q^{-2}J_{0}J_{6}J_{7}^{2}J_{8}^{2}}
		{J_{1}^{2}J_{2}^{2}J_{3}^{2}J_{4}^{2}J_{5}^{2}J_{9}^{2}J_{10}J_{11}J_{12}J_{13}^{2}J_{14}^{2}}
	\\&	
	+\tfrac{A(13,8,14)q^{-2}J_{0}J_{6}J_{7}^{2}J_{8}^{2}}
		{J_{1}^{2}J_{2}^{2}J_{3}^{2}J_{4}^{2}J_{5}^{2}J_{9}J_{10}^{2}J_{11}^{2}J_{12}J_{13}^{2}J_{14}}
	+\tfrac{A(5,7,6)q^{-3}J_{0}J_{6}^{2}J_{8}}
		{J_{1}^{2}J_{2}^{2}J_{3}^{2}J_{4}^{2}J_{5}^{2}J_{7}J_{11}J_{12}^{2}J_{13}J_{14}^{2}}
	+\tfrac{A(3,2,4)q^{-3}J_{0}J_{6}^{2}J_{8}J_{9}}
		{J_{1}^{2}J_{2}^{2}J_{3}^{2}J_{4}^{2}J_{5}^{2}J_{7}J_{10}J_{11}^{2}J_{12}^{2}J_{13}J_{14}}
	\\&	
	+\tfrac{A(-1,-6,0)q^{-3}J_{0}J_{6}^{2}J_{8}^{2}J_{11}}
		{J_{1}^{2}J_{2}^{2}J_{3}^{2}J_{4}^{2}J_{5}^{2}J_{7}J_{9}J_{10}^{2}J_{12}J_{13}^{2}J_{14}^{2}}
	+\tfrac{A(-4,-1,-6)q^{-3}J_{0}J_{6}^{2}J_{8}^{2}}
		{J_{1}^{2}J_{2}^{2}J_{3}^{2}J_{4}^{2}J_{5}^{2}J_{7}J_{10}^{2}J_{11}^{2}J_{12}J_{14}^{2}}
	+\tfrac{A(-1,1,0)q^{-3}J_{0}J_{6}^{2}J_{10}^{2}}
		{J_{1}^{2}J_{2}^{2}J_{3}^{2}J_{4}^{2}J_{5}^{2}J_{8}J_{11}^{2}J_{12}^{2}J_{13}J_{14}^{2}}
	+\tfrac{A(-4,-7,-5)q^{-3}J_{0}J_{6}^{2}J_{9}}
		{J_{1}^{2}J_{2}^{2}J_{3}^{2}J_{4}^{2}J_{5}^{2}J_{8}J_{11}^{2}J_{13}^{2}J_{14}^{2}}
	\\&	
	+\tfrac{A(1,-2,1)q^{-3}J_{0}J_{6}^{2}J_{7}J_{10}^{2}}
		{J_{1}^{2}J_{2}^{2}J_{3}^{2}J_{4}^{2}J_{5}^{2}J_{8}^{2}J_{9}J_{11}J_{12}J_{13}^{2}J_{14}^{2}}
	+\tfrac{A(-3,-2,-2)q^{-3}J_{0}J_{6}^{2}J_{7}J_{10}}
		{J_{1}^{2}J_{2}^{2}J_{3}^{2}J_{4}^{2}J_{5}^{2}J_{8}^{2}J_{11}^{2}J_{12}J_{13}^{2}J_{14}}
	+\tfrac{A(3,2,2)q^{-3}J_{0}J_{6}^{2}J_{7}J_{10}}
		{J_{1}^{2}J_{2}^{2}J_{3}^{2}J_{4}^{2}J_{5}^{2}J_{8}J_{9}^{2}J_{12}^{2}J_{13}^{2}J_{14}}
	+\tfrac{A(3,2,2)q^{-3}J_{0}J_{6}^{2}J_{7}J_{10}^{2}}
		{J_{1}^{2}J_{2}^{2}J_{3}^{2}J_{4}^{2}J_{5}^{2}J_{8}J_{9}^{2}J_{11}^{2}J_{12}^{2}J_{14}^{2}}
	\\&	
	+\tfrac{A(0,-1,1)q^{-3}J_{0}J_{6}^{2}J_{7}J_{12}^{2}}
		{J_{1}^{2}J_{2}^{2}J_{3}^{2}J_{4}^{2}J_{5}^{2}J_{8}J_{10}^{2}J_{11}^{2}J_{13}^{2}J_{14}^{2}}
	+\tfrac{A(0,1,-1)q^{-3}J_{0}J_{6}^{2}J_{7}J_{14}}
		{J_{1}^{2}J_{2}^{2}J_{3}^{2}J_{4}^{2}J_{5}^{2}J_{8}J_{10}J_{11}^{2}J_{12}^{2}J_{13}^{2}}
	+\tfrac{A(2,1,0)q^{-3}J_{0}J_{6}^{2}J_{7}}
		{J_{1}^{2}J_{2}^{2}J_{3}^{2}J_{4}^{2}J_{5}^{2}J_{9}^{2}J_{10}J_{11}J_{12}J_{13}J_{14}}
	\\&	
	+\tfrac{A(-2,-2,0)q^{-2}J_{0}J_{6}^{2}J_{7}J_{8}^{2}}
		{J_{1}^{2}J_{2}^{2}J_{3}^{2}J_{4}^{2}J_{5}^{2}J_{9}^{2}J_{10}^{2}J_{11}J_{12}J_{13}J_{14}^{2}}
	+\tfrac{A(1,1,1)q^{-1}J_{0}J_{6}^{2}J_{7}^{2}J_{8}}
		{J_{1}^{2}J_{2}^{2}J_{3}^{2}J_{4}^{2}J_{5}^{2}J_{9}J_{11}^{2}J_{12}^{2}J_{13}^{2}J_{14}^{2}}
	+\tfrac{A(-1,1,-2)q^{-3}J_{0}J_{8}J_{10}^{2}}
		{J_{1}^{2}J_{2}^{2}J_{3}^{2}J_{4}^{2}J_{5}J_{9}J_{11}^{2}J_{12}^{2}J_{13}J_{14}^{2}}
	+\tfrac{A(-1,-1,0)q^{-3}J_{0}J_{7}J_{10}}
		{J_{1}^{2}J_{2}^{2}J_{3}^{2}J_{4}^{2}J_{5}J_{9}J_{11}^{2}J_{12}J_{13}^{2}J_{14}}
	\\&	
	+\tfrac{A(-1,-1,-1)q^{-1}J_{0}J_{6}J_{7}^{2}J_{8}^{2}}
		{J_{1}^{2}J_{2}^{2}J_{3}^{2}J_{4}^{2}J_{5}J_{9}^{2}J_{10}^{2}J_{11}^{2}J_{12}J_{13}^{2}J_{14}}
	+\tfrac{A(-1,-1,-1)J_{0}J_{6}^{2}J_{7}^{2}J_{8}}
		{J_{1}^{2}J_{2}^{2}J_{3}^{2}J_{4}^{2}J_{5}J_{9}^{2}J_{11}^{2}J_{12}^{2}J_{13}^{2}J_{14}^{2}}
	+\tfrac{A(1,1,0)qJ_{0}J_{5}^{2}J_{6}J_{8}^{2}}
		{J_{1}^{2}J_{2}^{2}J_{3}^{2}J_{4}^{2}J_{7}J_{9}^{2}J_{10}J_{11}^{2}J_{12}J_{13}^{2}J_{14}^{2}}
	,\\
	&R2_3(q)
	=
	\tfrac{A(2,2,1)q^{-3}J_{0}J_{7}^{2}J_{8}^{2}}
		{J_{1}^{2}J_{2}^{2}J_{3}^{2}J_{4}^{2}J_{5}^{2}J_{9}^{2}J_{11}J_{12}J_{13}^{2}J_{14}^{2}}
	+\tfrac{A(-2,0,-2)q^{-3}J_{0}J_{7}^{2}J_{8}^{2}}
		{J_{1}^{2}J_{2}^{2}J_{3}^{2}J_{4}^{2}J_{5}^{2}J_{9}J_{10}J_{11}^{2}J_{12}J_{13}^{2}J_{14}}
	+\tfrac{A(5,0,2)q^{-3}J_{0}J_{6}J_{8}^{2}}
		{J_{1}^{2}J_{2}^{2}J_{3}^{2}J_{4}^{2}J_{5}^{2}J_{10}J_{11}J_{12}J_{13}^{2}J_{14}^{2}}
	\\&	
	+\tfrac{A(-1,-3,0)q^{-3}J_{0}J_{6}J_{8}^{2}J_{9}}
		{J_{1}^{2}J_{2}^{2}J_{3}^{2}J_{4}^{2}J_{5}^{2}J_{10}^{2}J_{11}^{2}J_{12}J_{13}^{2}J_{14}}
	+\tfrac{A(-1,0,-2)q^{-3}J_{0}J_{6}J_{7}J_{10}}
		{J_{1}^{2}J_{2}^{2}J_{3}^{2}J_{4}^{2}J_{5}^{2}J_{11}^{2}J_{12}J_{13}^{2}J_{14}^{2}}
	+\tfrac{A(5,2,4)q^{-3}J_{0}J_{6}J_{7}J_{8}J_{10}}
		{J_{1}^{2}J_{2}^{2}J_{3}^{2}J_{4}^{2}J_{5}^{2}J_{9}^{2}J_{12}^{2}J_{13}^{2}J_{14}^{2}}
	+\tfrac{A(5,4,5)q^{-3}J_{0}J_{6}J_{7}J_{8}}
		{J_{1}^{2}J_{2}^{2}J_{3}^{2}J_{4}^{2}J_{5}^{2}J_{9}J_{11}J_{12}^{2}J_{13}^{2}J_{14}}
	\\&	
	+\tfrac{A(3,4,2)q^{-3}J_{0}J_{6}J_{7}J_{8}}
		{J_{1}^{2}J_{2}^{2}J_{3}^{2}J_{4}^{2}J_{5}^{2}J_{10}J_{11}^{2}J_{12}^{2}J_{13}^{2}}
	+\tfrac{A(-9,-4,-7)q^{-3}J_{0}J_{6}J_{7}J_{8}^{2}}
		{J_{1}^{2}J_{2}^{2}J_{3}^{2}J_{4}^{2}J_{5}^{2}J_{9}^{2}J_{10}J_{11}J_{12}J_{13}J_{14}^{2}}
	+\tfrac{A(-4,-5,0)q^{-3}J_{0}J_{6}J_{7}J_{8}^{2}}
		{J_{1}^{2}J_{2}^{2}J_{3}^{2}J_{4}^{2}J_{5}^{2}J_{9}J_{10}^{2}J_{11}^{2}J_{12}J_{13}J_{14}}
	\\&	
	+\tfrac{A(-3,-3,-1)q^{-3}J_{0}J_{6}J_{7}^{2}J_{10}^{2}}
		{J_{1}^{2}J_{2}^{2}J_{3}^{2}J_{4}^{2}J_{5}^{2}J_{8}J_{9}J_{11}^{2}J_{12}^{2}J_{13}^{2}J_{14}}
	+\tfrac{A(10,7,6)q^{-3}J_{0}J_{6}J_{7}^{2}J_{12}}
		{J_{1}^{2}J_{2}^{2}J_{3}^{2}J_{4}^{2}J_{5}^{2}J_{9}J_{10}J_{11}^{2}J_{13}^{2}J_{14}^{2}}
	+\tfrac{A(5,3,3)q^{-3}J_{0}J_{6}J_{7}^{2}J_{8}}
		{J_{1}^{2}J_{2}^{2}J_{3}^{2}J_{4}^{2}J_{5}^{2}J_{9}^{2}J_{10}^{2}J_{11}J_{13}^{2}J_{14}}
	+\tfrac{A(-5,-2,-4)q^{-3}J_{0}J_{6}J_{7}^{2}J_{8}}
		{J_{1}^{2}J_{2}^{2}J_{3}^{2}J_{4}^{2}J_{5}^{2}J_{9}^{2}J_{10}J_{11}^{2}J_{12}^{2}J_{13}}
	\\&	
	+\tfrac{A(3,2,2)q^{-2}J_{0}J_{6}J_{7}^{2}J_{8}J_{10}}
		{J_{1}^{2}J_{2}^{2}J_{3}^{2}J_{4}^{2}J_{5}^{2}J_{9}J_{11}^{2}J_{12}^{2}J_{13}^{2}J_{14}^{2}}
	+\tfrac{A(1,3,0)q^{-3}J_{0}J_{6}^{2}J_{9}^{2}}
		{J_{1}^{2}J_{2}^{2}J_{3}^{2}J_{4}^{2}J_{5}^{2}J_{7}J_{11}^{2}J_{12}J_{13}^{2}J_{14}^{2}}
	+\tfrac{A(0,4,0)q^{-3}J_{0}J_{6}^{2}J_{8}}
		{J_{1}^{2}J_{2}^{2}J_{3}^{2}J_{4}^{2}J_{5}^{2}J_{7}J_{12}^{2}J_{13}^{2}J_{14}^{2}}
	+\tfrac{A(9,13,4)q^{-3}J_{0}J_{6}^{2}J_{8}J_{9}}
		{J_{1}^{2}J_{2}^{2}J_{3}^{2}J_{4}^{2}J_{5}^{2}J_{7}J_{10}J_{11}J_{12}^{2}J_{13}^{2}J_{14}}
	\\&	
	+\tfrac{A(0,-4,0)q^{-3}J_{0}J_{6}^{2}J_{8}^{2}}
		{J_{1}^{2}J_{2}^{2}J_{3}^{2}J_{4}^{2}J_{5}^{2}J_{7}J_{10}^{2}J_{11}J_{12}J_{13}J_{14}^{2}}
	+\tfrac{A(-1,1,0)q^{-3}J_{0}J_{6}^{2}J_{10}^{2}}
		{J_{1}^{2}J_{2}^{2}J_{3}^{2}J_{4}^{2}J_{5}^{2}J_{8}J_{11}J_{12}^{2}J_{13}^{2}J_{14}^{2}}
	+\tfrac{A(0,1,-1)q^{-3}J_{0}J_{6}^{2}J_{9}J_{10}}
		{J_{1}^{2}J_{2}^{2}J_{3}^{2}J_{4}^{2}J_{5}^{2}J_{8}J_{11}^{2}J_{12}^{2}J_{13}^{2}J_{14}}
	+\tfrac{A(-7,-13,-2)q^{-3}J_{0}J_{6}^{2}}
		{J_{1}^{2}J_{2}^{2}J_{3}^{2}J_{4}^{2}J_{5}^{2}J_{11}^{2}J_{12}J_{13}J_{14}^{2}}
	\\&	
	+\tfrac{A(-12,-11,-6)q^{-3}J_{0}J_{6}^{2}J_{8}}
		{J_{1}^{2}J_{2}^{2}J_{3}^{2}J_{4}^{2}J_{5}^{2}J_{9}J_{10}^{2}J_{13}^{2}J_{14}^{2}}
	+\tfrac{A(4,1,2)q^{-3}J_{0}J_{6}^{2}J_{8}}
		{J_{1}^{2}J_{2}^{2}J_{3}^{2}J_{4}^{2}J_{5}^{2}J_{10}^{2}J_{11}^{2}J_{12}^{2}J_{13}}
	+\tfrac{A(-3,-2,-2)q^{-3}J_{0}J_{6}^{2}J_{7}J_{10}}
		{J_{1}^{2}J_{2}^{2}J_{3}^{2}J_{4}^{2}J_{5}^{2}J_{8}J_{9}J_{11}^{2}J_{12}^{2}J_{13}J_{14}}
	+\tfrac{A(-2,1,-2)q^{-3}J_{0}J_{6}^{2}J_{7}}
		{J_{1}^{2}J_{2}^{2}J_{3}^{2}J_{4}^{2}J_{5}^{2}J_{9}^{2}J_{10}J_{12}J_{13}^{2}J_{14}}
	\\&	
	+\tfrac{A(-10,-12,-5)q^{-2}J_{0}J_{6}^{2}J_{7}J_{8}}
		{J_{1}^{2}J_{2}^{2}J_{3}^{2}J_{4}^{2}J_{5}^{2}J_{9}J_{11}^{2}J_{12}^{2}J_{13}J_{14}^{2}}
	+\tfrac{A(3,3,1)q^{-3}J_{0}J_{6}^{2}J_{7}^{2}J_{10}}
		{J_{1}^{2}J_{2}^{2}J_{3}^{2}J_{4}^{2}J_{5}^{2}J_{8}^{2}J_{9}^{2}J_{11}J_{12}J_{13}^{2}J_{14}}
	+\tfrac{A(7,11,3)q^{-2}J_{0}J_{6}^{2}J_{7}^{2}}
		{J_{1}^{2}J_{2}^{2}J_{3}^{2}J_{4}^{2}J_{5}^{2}J_{9}J_{10}J_{11}^{2}J_{12}J_{13}^{2}J_{14}}
	\\&	
	+\tfrac{A(0,-1,0)q^{-3}J_{0}J_{7}J_{8}^{2}J_{10}}
		{J_{1}^{2}J_{2}^{2}J_{3}^{2}J_{4}^{2}J_{5}J_{6}J_{9}J_{11}^{2}J_{12}J_{13}^{2}J_{14}^{2}}
	+\tfrac{A(-1,1,0)q^{-3}J_{0}J_{8}^{2}J_{9}}
		{J_{1}^{2}J_{2}^{2}J_{3}^{2}J_{4}^{2}J_{5}J_{7}J_{11}^{2}J_{12}J_{13}^{2}J_{14}^{2}}
	+\tfrac{A(-2,-5,-1)q^{-3}J_{0}J_{8}J_{10}}
		{J_{1}^{2}J_{2}^{2}J_{3}^{2}J_{4}^{2}J_{5}J_{11}^{2}J_{12}^{2}J_{13}^{2}J_{14}}
	+\tfrac{A(1,3,0)q^{-3}J_{0}J_{8}^{2}J_{12}}
		{J_{1}^{2}J_{2}^{2}J_{3}^{2}J_{4}^{2}J_{5}J_{10}^{2}J_{11}^{2}J_{13}^{2}J_{14}^{2}}
	\\&	
	+\tfrac{A(0,-1,0)q^{-2}J_{0}J_{7}^{2}J_{8}^{2}}
		{J_{1}^{2}J_{2}^{2}J_{3}^{2}J_{4}^{2}J_{5}J_{9}^{2}J_{10}J_{11}^{2}J_{12}J_{13}^{2}J_{14}}
	+\tfrac{A(1,0,0)J_{0}J_{6}^{2}J_{7}^{2}J_{8}^{2}}
		{J_{1}^{2}J_{2}^{2}J_{3}^{2}J_{4}^{2}J_{5}J_{9}^{2}J_{10}^{2}J_{11}^{2}J_{12}J_{13}^{2}J_{14}^{2}}
	+\tfrac{A(-4,-2,-2)qJ_{0}J_{5}J_{6}^{2}J_{8}}
		{J_{1}^{2}J_{2}^{2}J_{3}^{2}J_{4}^{2}J_{9}^{2}J_{11}^{2}J_{12}^{2}J_{13}^{2}J_{14}^{2}}
	,\\
	&R2_4(q)
	=
	A(0,1,-1)q^{-\frac{4}{7}}S(1,7,\tau/7)
	+A(0,1,-1)q^{-\frac{4}{7}}S(2,7,\tau/7)
	+\tfrac{A(7,-4,7)q^{2}J_{0}J_{12}}
		{J_{1}J_{2}J_{3}J_{4}J_{5}J_{6}J_{7}J_{8}J_{9}J_{10}^{2}J_{11}^{2}J_{13}^{2}}
	\\&
	+\tfrac{A(-14,0,-10)q^{3}J_{0}J_{10}^{2}}
		{J_{1}J_{2}J_{3}J_{4}J_{5}J_{6}J_{7}J_{8}J_{9}J_{11}^{2}J_{12}J_{13}^{2}J_{14}^{2}}
	+\tfrac{A(5,-12,11)q^{2}J_{0}J_{12}}
		{J_{1}J_{2}J_{3}J_{4}J_{5}J_{6}J_{7}J_{8}J_{9}^{2}J_{10}J_{11}J_{13}^{2}J_{14}}
	+\tfrac{A(-6,13,-14)q^{2}J_{0}}
		{J_{1}J_{2}J_{3}J_{4}J_{5}J_{6}J_{7}J_{8}J_{9}^{2}J_{11}^{2}J_{12}J_{13}}
	\\&
	+\tfrac{A(5,13,-3)q^{2}J_{0}J_{12}^{2}}
		{J_{1}J_{2}J_{3}J_{4}J_{5}J_{6}J_{7}J_{8}^{2}J_{9}J_{11}^{2}J_{13}^{2}J_{14}^{2}}
	+\tfrac{A(-3,-2,-2)q^{2}J_{0}J_{10}J_{14}}
		{J_{1}J_{2}J_{3}J_{4}J_{5}J_{6}J_{7}J_{8}^{2}J_{9}J_{11}^{2}J_{12}^{2}J_{13}^{2}}
	+\tfrac{A(-1,1,0)q^{2}J_{0}J_{10}^{2}}
		{J_{1}J_{2}J_{3}J_{4}J_{5}J_{6}J_{7}J_{8}^{2}J_{9}^{2}J_{11}J_{12}^{2}J_{13}^{2}}
	\\&
	+\tfrac{A(3,-4,4)q^{3}J_{0}}
		{J_{1}J_{2}J_{3}J_{4}J_{5}J_{6}J_{7}J_{9}J_{11}^{2}J_{12}^{2}J_{13}^{2}}
	+\tfrac{A(-3,1,-3)q^{2}J_{0}J_{14}^{2}}
		{J_{1}J_{2}J_{3}J_{4}J_{5}J_{6}J_{7}J_{9}^{2}J_{10}^{2}J_{11}^{2}J_{12}^{2}J_{13}}
	+\tfrac{A(-1,1,0)q^{3}J_{0}J_{10}}
		{J_{1}J_{2}J_{3}J_{4}J_{5}J_{6}J_{7}J_{9}^{2}J_{11}J_{12}^{2}J_{13}^{2}J_{14}}
	\\&
	+\tfrac{A(7,-5,8)q^{3}J_{0}J_{8}}
		{J_{1}J_{2}J_{3}J_{4}J_{5}J_{6}J_{7}J_{9}^{2}J_{10}J_{11}^{2}J_{12}J_{13}J_{14}}
	+\tfrac{A(7,0,5)q^{4}J_{0}J_{8}^{2}}
		{J_{1}J_{2}J_{3}J_{4}J_{5}J_{6}J_{7}J_{9}J_{10}J_{11}^{2}J_{12}^{2}J_{13}^{2}J_{14}}
	+\tfrac{A(7,0,5)qJ_{0}J_{14}^{2}}
		{J_{1}J_{2}J_{3}J_{4}J_{5}J_{6}J_{7}^{2}J_{8}J_{9}J_{10}^{2}J_{11}^{2}J_{12}J_{13}}
	\\&
	+\tfrac{A(4,-1,4)q^{2}J_{0}J_{10}}
		{J_{1}J_{2}J_{3}J_{4}J_{5}J_{6}J_{7}^{2}J_{8}J_{9}J_{11}J_{12}J_{13}^{2}J_{14}}
	+\tfrac{A(-11,5,-9)qJ_{0}J_{14}}
		{J_{1}J_{2}J_{3}J_{4}J_{5}J_{6}J_{7}^{2}J_{8}J_{9}^{2}J_{10}J_{11}J_{12}J_{13}}
	+\tfrac{A(4,1,2)qJ_{0}J_{12}}
		{J_{1}J_{2}J_{3}J_{4}J_{5}J_{6}J_{7}^{2}J_{8}J_{9}^{2}J_{10}J_{11}^{2}J_{14}}
	\\&	
	+\tfrac{A(0,-3,1)qJ_{0}J_{12}}
		{J_{1}J_{2}J_{3}J_{4}J_{5}J_{6}J_{7}^{2}J_{8}J_{9}^{2}J_{10}^{2}J_{13}^{2}}
	+\tfrac{A(15,-11,19)q^{2}J_{0}J_{10}^{2}}
		{J_{1}J_{2}J_{3}J_{4}J_{5}J_{6}J_{7}^{2}J_{8}J_{9}^{2}J_{12}J_{13}^{2}J_{14}^{2}}
	+\tfrac{A(-4,9,-9)q^{2}J_{0}}
		{J_{1}J_{2}J_{3}J_{4}J_{5}J_{6}J_{7}^{2}J_{8}J_{11}^{2}J_{12}J_{13}^{2}}
	\\&	
	+\tfrac{A(-7,-1,-4)qJ_{0}J_{12}^{2}}
		{J_{1}J_{2}J_{3}J_{4}J_{5}J_{6}J_{7}^{2}J_{8}^{2}J_{9}J_{10}J_{11}J_{13}^{2}J_{14}}
	+\tfrac{A(1,-1,0)qJ_{0}J_{14}^{2}}
		{J_{1}J_{2}J_{3}J_{4}J_{5}J_{6}J_{7}^{2}J_{8}^{2}J_{9}J_{11}J_{12}^{2}J_{13}^{2}}
	+\tfrac{A(-4,3,-5)qJ_{0}}
		{J_{1}J_{2}J_{3}J_{4}J_{5}J_{6}J_{7}^{2}J_{8}^{2}J_{9}J_{11}^{2}J_{13}}
	\\&
	+\tfrac{A(3,-4,4)J_{0}J_{14}^{2}}
		{J_{1}J_{2}J_{3}J_{4}J_{5}J_{6}J_{7}^{2}J_{8}^{2}J_{9}^{2}J_{10}^{2}J_{11}^{2}}
	+\tfrac{A(-2,12,-9)qJ_{0}J_{12}^{2}}
		{J_{1}J_{2}J_{3}J_{4}J_{5}J_{6}J_{7}^{2}J_{8}^{2}J_{9}^{2}J_{13}^{2}J_{14}^{2}}
	+\tfrac{A(15,-10,16)qJ_{0}J_{10}}
		{J_{1}J_{2}J_{3}J_{4}J_{5}J_{6}J_{7}^{2}J_{8}^{2}J_{9}^{2}J_{11}J_{13}J_{14}}
	\\&
	+\tfrac{A(6,-2,6)qJ_{0}J_{10}J_{14}}
		{J_{1}J_{2}J_{3}J_{4}J_{5}J_{6}J_{7}^{2}J_{8}^{2}J_{9}^{2}J_{12}^{2}J_{13}^{2}}
	+\tfrac{A(-13,-1,-10)q^{2}J_{0}J_{10}^{2}}
		{J_{1}J_{2}J_{3}J_{4}J_{5}J_{6}J_{7}^{2}J_{8}^{2}J_{11}^{2}J_{13}^{2}J_{14}^{2}}
	+\tfrac{A(8,-8,11)q^{2}J_{0}J_{14}}
		{J_{1}J_{2}J_{3}J_{4}J_{5}J_{6}J_{7}^{2}J_{9}J_{10}J_{11}J_{12}^{2}J_{13}^{2}}
	\\&
	+\tfrac{A(5,-12,11)q^{2}J_{0}}
		{J_{1}J_{2}J_{3}J_{4}J_{5}J_{6}J_{7}^{2}J_{9}J_{10}J_{11}^{2}J_{13}J_{14}}
	+\tfrac{A(-15,11,-19)q^{2}J_{0}}
		{J_{1}J_{2}J_{3}J_{4}J_{5}J_{6}J_{7}^{2}J_{9}^{2}J_{11}J_{13}J_{14}^{2}}
	+\tfrac{A(5,-7,8)q^{2}J_{0}}
		{J_{1}J_{2}J_{3}J_{4}J_{5}J_{6}J_{7}^{2}J_{9}^{2}J_{12}^{2}J_{13}^{2}}
	\\&
	+\tfrac{A(1,5,-2)q^{2}J_{0}J_{8}}
		{J_{1}J_{2}J_{3}J_{4}J_{5}J_{6}J_{7}^{2}J_{9}^{2}J_{10}^{2}J_{11}J_{12}J_{13}}
	+\tfrac{A(2,12,-6)q^{3}J_{0}J_{12}}
		{J_{1}J_{2}J_{3}J_{4}J_{5}J_{6}J_{8}J_{9}^{2}J_{11}^{2}J_{13}^{2}J_{14}^{2}}
	+\tfrac{A(-7,0,-5)q^{2}J_{0}J_{14}}
		{J_{1}J_{2}J_{3}J_{4}J_{5}J_{6}J_{8}^{2}J_{9}^{2}J_{10}J_{11}^{2}J_{13}^{2}}
	\\&	
	+\tfrac{A(-4,0,-3)q^{4}J_{0}J_{10}^{2}}
		{J_{1}J_{2}J_{3}J_{4}J_{5}J_{6}J_{9}^{2}J_{11}^{2}J_{12}^{2}J_{13}^{2}J_{14}^{2}}
	+\tfrac{A(3,-1,3)qJ_{0}J_{14}^{2}}
		{J_{1}J_{2}J_{3}J_{4}J_{5}J_{6}^{2}J_{7}J_{8}J_{9}J_{10}J_{11}^{2}J_{12}J_{13}^{2}}
	+\tfrac{A(2,-6,5)qJ_{0}J_{14}}
		{J_{1}J_{2}J_{3}J_{4}J_{5}J_{6}^{2}J_{7}J_{8}J_{9}^{2}J_{11}J_{12}J_{13}^{2}}
	\\&
	+\tfrac{A(-13,5,-12)qJ_{0}J_{10}}
		{J_{1}J_{2}J_{3}J_{4}J_{5}J_{6}^{2}J_{7}J_{8}^{2}J_{9}J_{11}^{2}J_{13}^{2}}
	+\tfrac{A(0,-3,1)J_{0}J_{14}^{2}}
		{J_{1}J_{2}J_{3}J_{4}J_{5}J_{6}^{2}J_{7}J_{8}^{2}J_{9}^{2}J_{10}J_{11}^{2}J_{13}}
	+\tfrac{A(-3,4,-4)J_{0}J_{12}^{2}J_{14}}
		{J_{1}J_{2}J_{3}J_{4}J_{5}J_{6}^{2}J_{7}J_{8}^{2}J_{9}^{2}J_{10}^{2}J_{11}J_{13}^{2}}
	\\&
	+\tfrac{A(-5,2,-4)q^{2}J_{0}}
		{J_{1}J_{2}J_{3}J_{4}J_{5}J_{6}^{2}J_{7}J_{9}J_{11}^{2}J_{13}^{2}J_{14}}
	+\tfrac{A(0,1,-1)q^{-1}J_{0}J_{12}^{2}J_{14}}
		{J_{1}J_{2}J_{3}J_{4}J_{5}J_{6}^{2}J_{7}^{2}J_{8}^{2}J_{9}^{2}J_{10}^{2}J_{11}^{2}}
	+\tfrac{A(0,1,-1)q^{2}J_{0}J_{14}}
		{J_{1}J_{2}J_{3}J_{4}J_{5}J_{7}J_{8}^{2}J_{9}^{2}J_{10}^{2}J_{11}^{2}J_{13}}
	\\&
	+\tfrac{A(-3,-1,-3)q^{3}J_{0}J_{10}^{2}}
		{J_{1}J_{2}J_{3}J_{4}J_{5}J_{7}^{2}J_{8}^{2}J_{11}^{2}J_{12}^{2}J_{13}^{2}J_{14}}
	+\tfrac{A(4,0,3)q^{5}J_{0}J_{8}}
		{J_{1}J_{2}J_{3}J_{4}J_{5}J_{9}^{2}J_{10}J_{11}^{2}J_{12}J_{13}^{2}J_{14}^{2}}
	+\tfrac{A(3,-4,4)q^{3}J_{0}J_{6}J_{14}^{2}}
		{J_{1}J_{2}J_{3}J_{4}J_{5}J_{7}J_{8}^{2}J_{9}^{2}J_{10}^{2}J_{11}^{2}J_{12}^{2}J_{13}}
	,\\
	&R2_5(q)
	=
	\tfrac{A(21,15,5)q^{-3}J_{0}J_{7}^{2}J_{8}^{2}}
		{J_{1}^{2}J_{2}^{2}J_{3}^{2}J_{4}^{2}J_{5}^{2}J_{9}J_{11}^{2}J_{12}J_{13}^{2}J_{14}^{2}}
	+\tfrac{A(-13,-8,-2)q^{-3}J_{0}J_{6}J_{8}^{2}J_{9}}
		{J_{1}^{2}J_{2}^{2}J_{3}^{2}J_{4}^{2}J_{5}^{2}J_{10}J_{11}^{2}J_{12}J_{13}^{2}J_{14}^{2}}
	+\tfrac{A(-4,0,-2)q^{-3}J_{0}J_{6}J_{7}J_{8}J_{10}}
		{J_{1}^{2}J_{2}^{2}J_{3}^{2}J_{4}^{2}J_{5}^{2}J_{9}J_{11}J_{12}^{2}J_{13}^{2}J_{14}^{2}}
	\\&	
	+\tfrac{A(-4,-6,-1)q^{-3}J_{0}J_{6}J_{7}J_{8}}
		{J_{1}^{2}J_{2}^{2}J_{3}^{2}J_{4}^{2}J_{5}^{2}J_{11}^{2}J_{12}^{2}J_{13}^{2}J_{14}}
	+\tfrac{A(23,13,9)q^{-3}J_{0}J_{6}J_{7}J_{8}^{2}}
		{J_{1}^{2}J_{2}^{2}J_{3}^{2}J_{4}^{2}J_{5}^{2}J_{9}J_{10}J_{11}^{2}J_{12}J_{13}J_{14}^{2}}
	+\tfrac{A(-25,-15,-6)q^{-3}J_{0}J_{6}J_{7}^{2}J_{8}}
		{J_{1}^{2}J_{2}^{2}J_{3}^{2}J_{4}^{2}J_{5}^{2}J_{9}^{2}J_{10}J_{11}J_{13}^{2}J_{14}^{2}}
	+\tfrac{A(-22,-17,-9)q^{-3}J_{0}J_{6}J_{7}^{2}J_{8}}
		{J_{1}^{2}J_{2}^{2}J_{3}^{2}J_{4}^{2}J_{5}^{2}J_{9}^{2}J_{11}^{2}J_{12}^{2}J_{13}J_{14}}
	\\&	
	+\tfrac{A(-13,-6,-6)q^{-3}J_{0}J_{6}J_{7}^{2}J_{8}}
		{J_{1}^{2}J_{2}^{2}J_{3}^{2}J_{4}^{2}J_{5}^{2}J_{9}J_{10}^{2}J_{11}^{2}J_{13}^{2}J_{14}}
	+\tfrac{A(-5,-2,-3)q^{-3}J_{0}J_{6}^{2}J_{8}J_{9}}
		{J_{1}^{2}J_{2}^{2}J_{3}^{2}J_{4}^{2}J_{5}^{2}J_{7}J_{11}J_{12}^{2}J_{13}^{2}J_{14}^{2}}
	+\tfrac{A(-2,-1,-2)q^{-3}J_{0}J_{6}^{2}J_{8}J_{9}^{2}}
		{J_{1}^{2}J_{2}^{2}J_{3}^{2}J_{4}^{2}J_{5}^{2}J_{7}J_{10}J_{11}^{2}J_{12}^{2}J_{13}^{2}J_{14}}
	+\tfrac{A(9,6,3)q^{-3}J_{0}J_{6}^{2}J_{9}J_{10}^{2}}
		{J_{1}^{2}J_{2}^{2}J_{3}^{2}J_{4}^{2}J_{5}^{2}J_{8}J_{11}^{2}J_{12}^{2}J_{13}^{2}J_{14}^{2}}
	\\&	
	+\tfrac{A(-21,-11,-7)q^{-3}J_{0}J_{6}^{2}J_{8}}
		{J_{1}^{2}J_{2}^{2}J_{3}^{2}J_{4}^{2}J_{5}^{2}J_{9}J_{11}J_{12}^{2}J_{13}J_{14}^{2}}
	+\tfrac{A(9,8,2)q^{-3}J_{0}J_{6}^{2}J_{8}}
		{J_{1}^{2}J_{2}^{2}J_{3}^{2}J_{4}^{2}J_{5}^{2}J_{10}^{2}J_{11}J_{13}^{2}J_{14}^{2}}
	+\tfrac{A(-8,-6,-1)q^{-3}J_{0}J_{6}^{2}J_{8}}
		{J_{1}^{2}J_{2}^{2}J_{3}^{2}J_{4}^{2}J_{5}^{2}J_{10}J_{11}^{2}J_{12}^{2}J_{13}J_{14}}
	+\tfrac{A(15,10,5)q^{-3}J_{0}J_{6}^{2}J_{8}^{2}J_{11}}
		{J_{1}^{2}J_{2}^{2}J_{3}^{2}J_{4}^{2}J_{5}^{2}J_{9}^{2}J_{10}^{2}J_{12}J_{13}^{2}J_{14}^{2}}
	\\&	
	+\tfrac{A(11,3,5)q^{-3}J_{0}J_{6}^{2}J_{8}^{2}}
		{J_{1}^{2}J_{2}^{2}J_{3}^{2}J_{4}^{2}J_{5}^{2}J_{9}J_{10}^{2}J_{11}^{2}J_{12}J_{14}^{2}}
	+\tfrac{A(3,2,2)q^{-3}J_{0}J_{6}^{2}J_{7}J_{10}^{2}}
		{J_{1}^{2}J_{2}^{2}J_{3}^{2}J_{4}^{2}J_{5}^{2}J_{8}J_{9}J_{11}^{2}J_{12}^{2}J_{13}J_{14}^{2}}
	+\tfrac{A(10,10,2)q^{-3}J_{0}J_{6}^{2}J_{7}}
		{J_{1}^{2}J_{2}^{2}J_{3}^{2}J_{4}^{2}J_{5}^{2}J_{8}J_{11}^{2}J_{13}^{2}J_{14}^{2}}
	+\tfrac{A(-15,-14,-4)q^{-3}J_{0}J_{6}^{2}J_{7}}
		{J_{1}^{2}J_{2}^{2}J_{3}^{2}J_{4}^{2}J_{5}^{2}J_{9}^{2}J_{12}J_{13}^{2}J_{14}^{2}}
	\\&	
	+\tfrac{A(19,10,7)q^{-3}J_{0}J_{6}^{2}J_{7}}
		{J_{1}^{2}J_{2}^{2}J_{3}^{2}J_{4}^{2}J_{5}^{2}J_{9}J_{10}J_{11}J_{12}J_{13}^{2}J_{14}}
	+\tfrac{A(4,1,0)q^{-3}J_{0}J_{6}^{2}J_{7}J_{8}}
		{J_{1}^{2}J_{2}^{2}J_{3}^{2}J_{4}^{2}J_{5}^{2}J_{9}^{2}J_{10}^{2}J_{11}J_{13}J_{14}^{2}}
	+\tfrac{A(0,1,2)q^{-3}J_{0}J_{6}^{2}J_{7}J_{8}}
		{J_{1}^{2}J_{2}^{2}J_{3}^{2}J_{4}^{2}J_{5}^{2}J_{9}^{2}J_{10}^{2}J_{12}^{2}J_{13}^{2}}
	\\&	
	+\tfrac{A(-7,-8,-2)q^{-2}J_{0}J_{6}^{2}J_{7}J_{8}^{2}}
		{J_{1}^{2}J_{2}^{2}J_{3}^{2}J_{4}^{2}J_{5}^{2}J_{9}J_{10}^{2}J_{11}J_{12}J_{13}^{2}J_{14}^{2}}
	+\tfrac{A(6,4,1)q^{-3}J_{0}J_{6}^{2}J_{7}^{2}J_{10}^{2}}
		{J_{1}^{2}J_{2}^{2}J_{3}^{2}J_{4}^{2}J_{5}^{2}J_{8}^{2}J_{9}^{2}J_{11}J_{12}J_{13}^{2}J_{14}^{2}}
	+\tfrac{A(-9,-7,-2)q^{-3}J_{0}J_{6}^{2}J_{7}^{2}J_{10}}
		{J_{1}^{2}J_{2}^{2}J_{3}^{2}J_{4}^{2}J_{5}^{2}J_{8}^{2}J_{9}J_{11}^{2}J_{12}J_{13}^{2}J_{14}}
	\\&	
	+\tfrac{A(2,2,1)q^{-3}J_{0}J_{6}^{2}J_{7}^{2}}
		{J_{1}^{2}J_{2}^{2}J_{3}^{2}J_{4}^{2}J_{5}^{2}J_{8}J_{9}^{2}J_{11}^{2}J_{13}J_{14}^{2}}
	+\tfrac{A(9,7,2)q^{-3}J_{0}J_{6}^{2}J_{7}^{2}}
		{J_{1}^{2}J_{2}^{2}J_{3}^{2}J_{4}^{2}J_{5}^{2}J_{8}J_{9}^{2}J_{11}J_{12}^{2}J_{13}^{2}}
	+\tfrac{A(6,4,1)q^{-3}J_{0}J_{6}^{2}J_{7}^{2}J_{12}^{2}}
		{J_{1}^{2}J_{2}^{2}J_{3}^{2}J_{4}^{2}J_{5}^{2}J_{8}J_{9}J_{10}^{2}J_{11}^{2}J_{13}^{2}J_{14}^{2}}
	+\tfrac{A(-6,-4,-1)q^{-3}J_{0}J_{6}^{2}J_{7}^{2}J_{14}}
		{J_{1}^{2}J_{2}^{2}J_{3}^{2}J_{4}^{2}J_{5}^{2}J_{8}J_{9}J_{10}J_{11}^{2}J_{12}^{2}J_{13}^{2}}
	\\&	
	+\tfrac{A(7,8,2)q^{-2}J_{0}J_{6}^{2}J_{7}^{2}J_{8}}
		{J_{1}^{2}J_{2}^{2}J_{3}^{2}J_{4}^{2}J_{5}^{2}J_{9}J_{10}^{2}J_{11}^{2}J_{12}^{2}J_{13}^{2}}
	+\tfrac{A(-4,-4,0)q^{-3}J_{0}J_{8}J_{10}^{2}}
		{J_{1}^{2}J_{2}^{2}J_{3}^{2}J_{4}^{2}J_{5}J_{11}^{2}J_{12}^{2}J_{13}^{2}J_{14}^{2}}
	+\tfrac{A(11,8,4)q^{-3}J_{0}J_{7}J_{8}J_{10}^{2}}
		{J_{1}^{2}J_{2}^{2}J_{3}^{2}J_{4}^{2}J_{5}J_{9}^{2}J_{11}^{2}J_{12}^{2}J_{13}J_{14}^{2}}
	+\tfrac{A(-4,-1,-2)q^{-3}J_{0}J_{7}J_{8}}
		{J_{1}^{2}J_{2}^{2}J_{3}^{2}J_{4}^{2}J_{5}J_{9}J_{11}^{2}J_{13}^{2}J_{14}^{2}}
	\\&	
	+\tfrac{A(1,1,0)q^{-3}J_{0}J_{7}J_{8}^{2}}
		{J_{1}^{2}J_{2}^{2}J_{3}^{2}J_{4}^{2}J_{5}J_{9}J_{10}^{2}J_{11}^{2}J_{12}J_{13}^{2}}
	+\tfrac{A(1,-1,0)q^{-3}J_{0}J_{6}J_{7}J_{10}}
		{J_{1}^{2}J_{2}^{2}J_{3}^{2}J_{4}^{2}J_{5}J_{8}J_{9}J_{11}^{2}J_{12}^{2}J_{13}^{2}}
	+\tfrac{A(-5,-2,-3)q^{-3}J_{0}J_{6}^{2}J_{8}J_{9}}
		{J_{1}^{2}J_{2}^{2}J_{3}^{2}J_{4}^{2}J_{5}J_{7}^{2}J_{10}^{2}J_{11}^{2}J_{13}J_{14}^{2}}
	+\tfrac{A(-2,-1,-2)J_{0}J_{5}J_{6}J_{8}^{2}}
		{J_{1}^{2}J_{2}^{2}J_{3}^{2}J_{4}^{2}J_{9}^{2}J_{10}J_{11}^{2}J_{12}J_{13}^{2}J_{14}^{2}}
	\\&	
	+\tfrac{A(5,2,3)qJ_{0}J_{5}^{2}J_{6}^{2}J_{8}^{2}}
		{J_{1}^{2}J_{2}^{2}J_{3}^{2}J_{4}^{2}J_{7}^{2}J_{9}J_{10}^{2}J_{11}^{2}J_{12}J_{13}^{2}J_{14}^{2}}	
	,\\
	&R2_6(q)
	=
	A(-1,1,0)q^{-\frac{6}{7}}S(4,7,\tau/7)
	+A(1,-1,0)q^{-\frac{6}{7}}S(6,7,\tau/7)
	+\tfrac{A(-30,4,-15)q^{-1}J_{0}J_{6}^{2}J_{7}J_{8}}
		{J_{1}^{2}J_{2}^{2}J_{3}^{2}J_{4}^{2}J_{5}J_{9}J_{11}^{2}J_{12}^{2}J_{13}^{2}J_{14}^{2}}
	\\&
	+\tfrac{A(10,2,5)q^{-1}J_{0}J_{6}^{2}J_{7}^{2}J_{8}}
		{J_{1}^{2}J_{2}^{2}J_{3}^{2}J_{4}^{2}J_{5}J_{9}^{2}J_{10}^{2}J_{11}^{2}J_{13}^{2}J_{14}^{2}}
	+\tfrac{A(-22,3,-5)q^{-1}J_{0}J_{6}J_{7}J_{8}^{2}}
		{J_{1}^{2}J_{2}^{2}J_{3}^{2}J_{4}^{2}J_{9}^{2}J_{10}^{2}J_{11}^{2}J_{12}J_{13}^{2}J_{14}}
	+\tfrac{A(7,-1,1)q^{-1}J_{0}J_{6}^{2}J_{8}^{2}}
		{J_{1}^{2}J_{2}^{2}J_{3}^{2}J_{4}^{2}J_{7}J_{9}J_{10}^{2}J_{11}J_{12}J_{13}^{2}J_{14}^{2}}
	+\tfrac{A(34,-7,21)q^{-1}J_{0}J_{6}^{2}}
		{J_{1}^{2}J_{2}^{2}J_{3}^{2}J_{4}^{2}J_{9}J_{11}^{2}J_{12}J_{13}^{2}J_{14}^{2}}
	\\&
	+\tfrac{A(27,4,11)q^{-1}J_{0}J_{6}^{2}J_{8}}
		{J_{1}^{2}J_{2}^{2}J_{3}^{2}J_{4}^{2}J_{9}^{2}J_{10}J_{11}J_{12}^{2}J_{13}^{2}J_{14}}
	+\tfrac{A(6,-1,0)q^{-1}J_{0}J_{6}^{2}J_{8}}
		{J_{1}^{2}J_{2}^{2}J_{3}^{2}J_{4}^{2}J_{9}J_{10}^{2}J_{11}^{2}J_{12}^{2}J_{13}^{2}}
	+\tfrac{A(-3,-2,-2)q^{-1}J_{0}J_{6}^{2}J_{7}J_{10}}
		{J_{1}^{2}J_{2}^{2}J_{3}^{2}J_{4}^{2}J_{8}J_{9}^{2}J_{11}^{2}J_{12}^{2}J_{13}^{2}J_{14}}
	+\tfrac{A(14,1,7)q^{-1}J_{0}J_{6}^{2}J_{7}J_{12}}
		{J_{1}^{2}J_{2}^{2}J_{3}^{2}J_{4}^{2}J_{9}^{2}J_{10}^{2}J_{11}^{2}J_{13}^{2}J_{14}^{2}}
	\\&
	+\tfrac{A(17,-5,7)J_{0}J_{6}^{2}J_{7}J_{8}}
		{J_{1}^{2}J_{2}^{2}J_{3}^{2}J_{4}^{2}J_{9}^{2}J_{11}^{2}J_{12}^{2}J_{13}^{2}J_{14}^{2}}
	+\tfrac{A(-3,0,-3)q^{-1}J_{0}J_{5}J_{6}J_{8}^{2}}
		{J_{1}^{2}J_{2}^{2}J_{3}^{2}J_{4}^{2}J_{7}^{2}J_{10}J_{11}^{2}J_{12}J_{13}^{2}J_{14}^{2}}
	+\tfrac{A(-17,1,-10)q^{-1}J_{0}J_{5}J_{6}J_{8}J_{10}}
		{J_{1}^{2}J_{2}^{2}J_{3}^{2}J_{4}^{2}J_{7}J_{9}^{2}J_{11}J_{12}^{2}J_{13}^{2}J_{14}^{2}}
	+\tfrac{A(-8,0,-6)q^{-1}J_{0}J_{5}J_{6}J_{8}}
		{J_{1}^{2}J_{2}^{2}J_{3}^{2}J_{4}^{2}J_{7}J_{9}J_{11}^{2}J_{12}^{2}J_{13}^{2}J_{14}}
	\\&
	+\tfrac{A(2,0,2)q^{-1}J_{0}J_{5}J_{6}J_{8}^{2}}
		{J_{1}^{2}J_{2}^{2}J_{3}^{2}J_{4}^{2}J_{7}J_{9}^{2}J_{10}J_{11}^{2}J_{12}J_{13}J_{14}^{2}}
	+\tfrac{A(-46,0,-21)q^{-1}J_{0}J_{5}J_{6}J_{8}}
		{J_{1}^{2}J_{2}^{2}J_{3}^{2}J_{4}^{2}J_{9}^{2}J_{10}^{2}J_{11}^{2}J_{13}^{2}J_{14}}
	+\tfrac{A(-21,-2,-10)q^{-1}J_{0}J_{5}J_{6}^{2}J_{8}}
		{J_{1}^{2}J_{2}^{2}J_{3}^{2}J_{4}^{2}J_{7}^{2}J_{9}^{2}J_{11}J_{12}^{2}J_{13}J_{14}^{2}}
	+\tfrac{A(-18,-2,-7)q^{-1}J_{0}J_{5}J_{6}^{2}J_{8}}
		{J_{1}^{2}J_{2}^{2}J_{3}^{2}J_{4}^{2}J_{7}^{2}J_{9}J_{10}^{2}J_{11}J_{13}^{2}J_{14}^{2}}
	\\&
	+\tfrac{A(5,2,5)q^{-1}J_{0}J_{5}J_{6}^{2}J_{8}}
		{J_{1}^{2}J_{2}^{2}J_{3}^{2}J_{4}^{2}J_{7}^{2}J_{9}J_{10}J_{11}^{2}J_{12}^{2}J_{13}J_{14}}
	+\tfrac{A(21,2,10)q^{-1}J_{0}J_{5}J_{6}^{2}J_{8}^{2}}
		{J_{1}^{2}J_{2}^{2}J_{3}^{2}J_{4}^{2}J_{7}^{2}J_{9}^{2}J_{10}^{2}J_{11}^{2}J_{12}J_{14}^{2}}
	+\tfrac{A(1,-1,0)q^{-1}J_{0}J_{5}J_{6}^{2}J_{10}^{2}}
		{J_{1}^{2}J_{2}^{2}J_{3}^{2}J_{4}^{2}J_{7}J_{8}J_{9}^{2}J_{11}^{2}J_{12}^{2}J_{13}J_{14}^{2}}
	\\&
	+\tfrac{A(-12,2,-3)q^{-1}J_{0}J_{5}J_{6}^{2}}
		{J_{1}^{2}J_{2}^{2}J_{3}^{2}J_{4}^{2}J_{7}J_{9}^{2}J_{10}J_{11}J_{12}J_{13}^{2}J_{14}}
	+\tfrac{A(11,-2,3)q^{-1}J_{0}J_{5}J_{6}^{2}}
		{J_{1}^{2}J_{2}^{2}J_{3}^{2}J_{4}^{2}J_{7}J_{9}J_{10}^{2}J_{11}^{2}J_{12}J_{13}^{2}}
	+\tfrac{A(-20,3,-10)J_{0}J_{5}J_{6}^{2}J_{8}^{2}}
		{J_{1}^{2}J_{2}^{2}J_{3}^{2}J_{4}^{2}J_{7}J_{9}^{2}J_{10}^{2}J_{11}J_{12}J_{13}^{2}J_{14}^{2}}
	\\&
	+\tfrac{A(3,2,2)q^{-1}J_{0}J_{5}J_{6}^{2}J_{10}}
		{J_{1}^{2}J_{2}^{2}J_{3}^{2}J_{4}^{2}J_{8}^{2}J_{9}^{2}J_{11}^{2}J_{12}J_{13}^{2}J_{14}}
	+\tfrac{A(29,-1,12)q^{-1}J_{0}J_{5}J_{6}^{2}J_{12}^{2}}
		{J_{1}^{2}J_{2}^{2}J_{3}^{2}J_{4}^{2}J_{8}J_{9}^{2}J_{10}^{2}J_{11}^{2}J_{13}^{2}J_{14}^{2}}
	+\tfrac{A(0,-1,1)q^{-1}J_{0}J_{5}J_{6}^{2}J_{14}}
		{J_{1}^{2}J_{2}^{2}J_{3}^{2}J_{4}^{2}J_{8}J_{9}^{2}J_{10}J_{11}^{2}J_{12}^{2}J_{13}^{2}}
	+\tfrac{A(7,0,3)q^{-1}J_{0}J_{5}^{2}J_{8}J_{10}^{2}}
		{J_{1}^{2}J_{2}^{2}J_{3}^{2}J_{4}^{2}J_{7}^{2}J_{9}J_{11}^{2}J_{12}^{2}J_{13}^{2}J_{14}^{2}}
	\\&
	+\tfrac{A(29,-2,13)q^{-1}J_{0}J_{5}^{2}J_{8}}
		{J_{1}^{2}J_{2}^{2}J_{3}^{2}J_{4}^{2}J_{7}J_{9}^{2}J_{11}^{2}J_{13}^{2}J_{14}^{2}}
	+\tfrac{A(1,0,0)q^{-1}J_{0}J_{5}^{2}J_{6}J_{8}}
		{J_{1}^{2}J_{2}^{2}J_{3}^{2}J_{4}^{2}J_{7}^{2}J_{9}^{2}J_{10}J_{11}J_{12}^{2}J_{13}^{2}}
	+\tfrac{A(-7,0,-3)J_{0}J_{5}^{2}J_{6}J_{8}^{2}}
		{J_{1}^{2}J_{2}^{2}J_{3}^{2}J_{4}^{2}J_{7}^{2}J_{9}J_{10}J_{11}^{2}J_{12}J_{13}^{2}J_{14}^{2}}
	+\tfrac{A(-29,2,-13)q^{-1}J_{0}J_{5}^{2}J_{6}J_{12}}
		{J_{1}^{2}J_{2}^{2}J_{3}^{2}J_{4}^{2}J_{7}J_{9}^{2}J_{10}^{2}J_{11}^{2}J_{13}^{2}J_{14}}
	\\&
	+\tfrac{A(7,0,3)q^{-1}J_{0}J_{5}^{2}J_{6}^{2}J_{12}}
		{J_{1}^{2}J_{2}^{2}J_{3}^{2}J_{4}^{2}J_{7}^{2}J_{8}^{2}J_{9}J_{11}^{2}J_{13}^{2}J_{14}^{2}}
	+\tfrac{A(2,-1,1)q^{-1}J_{0}J_{6}J_{7}J_{8}^{2}}
		{J_{1}^{2}J_{2}^{2}J_{3}^{2}J_{4}J_{5}^{2}J_{10}J_{11}^{2}J_{12}^{2}J_{13}^{2}J_{14}^{2}}
	+\tfrac{A(-1,1,0)q^{-1}J_{0}J_{6}J_{7}^{2}J_{8}^{2}}
		{J_{1}^{2}J_{2}^{2}J_{3}^{2}J_{4}J_{5}^{2}J_{9}^{2}J_{10}J_{11}^{2}J_{12}^{2}J_{13}J_{14}^{2}}
	\\&
	+\tfrac{A(5,1,3)q^{-1}J_{0}J_{6}^{2}J_{8}^{2}}
		{J_{1}^{2}J_{2}^{2}J_{3}^{2}J_{4}J_{5}^{2}J_{10}^{2}J_{11}^{2}J_{12}^{2}J_{13}J_{14}^{2}}
	+\tfrac{A(5,-2,2)q^{-1}J_{0}J_{6}^{2}J_{7}J_{8}^{2}}
		{J_{1}^{2}J_{2}^{2}J_{3}^{2}J_{4}J_{5}^{2}J_{9}^{2}J_{10}^{2}J_{11}^{2}J_{12}^{2}J_{14}^{2}}
	+\tfrac{A(-5,-5,-4)q^{-1}J_{0}J_{6}^{2}J_{7}^{2}}
		{J_{1}^{2}J_{2}^{2}J_{3}^{2}J_{4}J_{5}^{2}J_{8}J_{9}J_{11}^{2}J_{12}J_{13}^{2}J_{14}^{2}}
	\\&
	+\tfrac{A(-1,1,0)q^{-1}J_{0}J_{6}^{2}J_{7}^{2}}
		{J_{1}^{2}J_{2}^{2}J_{3}^{2}J_{4}J_{5}^{2}J_{9}^{2}J_{10}J_{11}J_{12}^{2}J_{13}^{2}J_{14}}
	+\tfrac{A(10,2,5)J_{0}J_{6}^{2}J_{7}^{2}J_{8}^{2}}
		{J_{1}^{2}J_{2}^{2}J_{3}^{2}J_{4}J_{5}^{2}J_{9}^{2}J_{10}^{2}J_{11}J_{12}^{2}J_{13}^{2}J_{14}^{2}}
	+\tfrac{A(-2,-2,-3)q^{-1}J_{0}J_{7}^{2}J_{8}}
		{J_{1}^{2}J_{2}^{2}J_{3}^{2}J_{4}J_{5}J_{9}^{2}J_{11}^{2}J_{12}J_{13}^{2}J_{14}^{2}}
	\\&
	+\tfrac{A(1,0,0)q^{-1}J_{0}J_{7}^{2}J_{8}^{2}}
		{J_{1}^{2}J_{2}^{2}J_{3}^{2}J_{4}J_{5}J_{9}^{2}J_{10}^{2}J_{11}^{2}J_{12}^{2}J_{13}^{2}}
	+\tfrac{A(-11,4,-1)q^{-1}J_{0}J_{6}J_{8}}
		{J_{1}^{2}J_{2}^{2}J_{3}^{2}J_{4}J_{5}J_{10}J_{11}^{2}J_{12}J_{13}^{2}J_{14}^{2}}
	+\tfrac{A(-1,-3,-3)q^{-1}J_{0}J_{6}J_{8}^{2}}
		{J_{1}^{2}J_{2}^{2}J_{3}^{2}J_{4}J_{5}J_{9}J_{10}^{2}J_{11}J_{12}^{2}J_{13}^{2}J_{14}}
	+\tfrac{A(3,3,1)q^{-1}J_{0}J_{6}^{2}J_{10}^{2}}
		{J_{1}^{2}J_{2}^{2}J_{3}^{2}J_{4}J_{5}J_{8}^{2}J_{11}^{2}J_{12}^{2}J_{13}^{2}J_{14}^{2}}
	\\&
	+\tfrac{A(-2,-2,-1)qJ_{0}J_{6}J_{7}J_{8}^{2}}
		{J_{1}^{2}J_{2}^{2}J_{3}^{2}J_{4}J_{9}^{2}J_{10}J_{11}^{2}J_{12}^{2}J_{13}^{2}J_{14}^{2}}
	+\tfrac{A(3,2,3)q^{2}J_{0}J_{5}J_{6}^{2}J_{8}^{2}}
		{J_{1}^{2}J_{2}^{2}J_{3}^{2}J_{4}J_{7}J_{9}J_{10}^{2}J_{11}^{2}J_{12}^{2}J_{13}^{2}J_{14}^{2}}
.
\end{align*}
Due to the length of the identities, we perform the calculations in Maple and
only sketch the proof here. We note that while the 
calculations are quite long in terms of their statements, they are quite quick
to compute. A Maple file containing calculations relevant to proving Theorem
\ref{TheoremDissection} can be found in the publications section of the 
author's website,
\url{http://people.oregonstate.edu/~jennichr/#research}.

We multiply both sides of the identity
in Theorem \ref{TheoremDissection} by $q^{-\frac{1}{8}}$, in doing so we find
that the exponents simplify correctly to give
an identity between $\mathcal{R}2(1,7;\tau)$, various $\mathcal{S}(k,7;\tau)$,
and various quotients of $f_{196,7k}(\tau)$.
When then multiply both sides by $\frac{\eta(4\tau)\eta(\tau)}{\eta(2\tau)}$
and take the difference of both sides, we call this difference
$LHS-RHS$. We are to show $0=LHS-RHS$.
Upon inspection we find that the resulting generalized eta quotients
are all modular forms of weight $1$ on $\Gamma_0(196)\cap\Gamma_1(28)$
by Corollary \ref{CorollaryProductsModular}
and more importantly 
the roots of unity have been chosen correctly so
that part (3) of Theorem \ref{TheoremModularity} applies 
to give that the remaining terms also constitute a weight 1 modular form
on $\Gamma_0(196)\cap\Gamma_1(28)$.
We then choose one of the generalized eta quotients, call it $g_1$,
so that $\frac{LHS-RHS}{g_1}$ is a modular function on
$\Gamma_0(196)\cap\Gamma_1(28)$. We then verify that
$0=\frac{LHS-RHS}{g_1}$ according to the valence formula.

To state the valence formula, we need to recall the concept of the order
at a point with 
respect to a subgroup.
Suppose $f$ is a modular function on some congruence subgroup 
$\Gamma\subset\SLTwo$.
Suppose $B=\SmallMatrix{\alpha}{\beta}{\gamma}{\delta}\in\SLTwo$,
we then have a cusp $\zeta=B(\infty)=\frac{\alpha}{\gamma}$.
We let $ord(f;\zeta)$ denote the invariant order of $f$ at
$\zeta$. We define the width of $\zeta$ with respect to $\Gamma$
as $width_\Gamma(\zeta):=w$, where $w$ is the least positive integer such that
$B\SmallMatrix{1}{w}{0}{1}B^{-1}\in\Gamma$. We then define
the order of $f$ at $\zeta$ with respect to $\Gamma$ as
$ORD_\Gamma(f;\zeta)=ord(f;\zeta) width_\Gamma(\zeta)$.
For $z\in\mathcal{H}$ we let $ord(f;z)$ denote the order of $f$ at $z$ as a
meromorphic function. We then define the order of $f$ at $z$ with respect to
$\Gamma$ as $ORD_\Gamma(f;z)=ord(f;z)/m$ where $m$ is the order of $z$ as a
fixed point of $\Gamma$ (so $m=1$, $2$, or $3$).

The valence formula, for modular functions can be stated as follows.
Suppose $f$ is a modular function that is not the zero function 
and $\mathcal{D}\subset\mathcal{H}\cup\mathbb{Q}\cup\{\infty\}$ is a
fundamental domain for the action of $\Gamma$ on $\mathcal{H}$ along 
with a complete set of inequivalent cusps for the action,
then
\begin{align*}
	\sum_{\zeta \in \mathcal{D}} ORD_\Gamma(f;\zeta) = 0
.
\end{align*}

A complete set of inequivalent cusps, along with their widths, 
for $\Gamma_0(196)\cap\Gamma_1(28)$ is
given in Table \ref{Table1}. We let $\mathcal{D}$ denote these cusps along with a fundamental region of the 
action of $\Gamma$.

\begin{table}\caption{}\label{Table1}
$\renewcommand{\arraystretch}{1.1}	
\setlength\arraycolsep{3pt}
	\begin{array}{l|cccccccccccccccccccccc}
		\mbox{cusp}&
		0& \tfrac{1}{28}& \tfrac{3}{80}& \tfrac{2}{53}& \tfrac{1}{26}& 
		\tfrac{3}{77}& \tfrac{2}{51}& \tfrac{1}{24}& \tfrac{2}{47}& \tfrac{3}{70}& 
		\tfrac{2}{45}& \tfrac{5}{112}& \tfrac{1}{22}& \tfrac{1}{21}& \tfrac{1}{20}& 
		\tfrac{5}{98}& \tfrac{3}{56}& \tfrac{1}{18}& \tfrac{2}{35}& \tfrac{1}{16}&
	 	\tfrac{1}{15}& \tfrac{1}{14}
		\\		
		\mbox{width}&				
		196& 1& 49& 196& 98& 4& 196& 49& 196& 2& 
		196& 1& 98& 4& 49& 2& 1& 98& 4& 49& 
		196& 2
		\\			
		\hline
		\mbox{cusp}&
		\tfrac{5}{63}& \tfrac{1}{12}& \tfrac{3}{35}& \tfrac{8}{91}& \tfrac{5}{56}& 
	 	\tfrac{2}{21}& \tfrac{11}{112}& \tfrac{5}{49}& \tfrac{8}{77}& \tfrac{3}{28}& 
	 	\tfrac{5}{42}& \tfrac{11}{91}& \tfrac{6}{49}& \tfrac{1}{8}& \tfrac{8}{63}& 
		\tfrac{9}{70}& \tfrac{12}{91}& \tfrac{1}{7}& \tfrac{12}{77}& \tfrac{10}{63}& 
	 	\tfrac{9}{56}& \tfrac{6}{35}
		\\
		\mbox{width}&
		4& 49& 4& 4& 1& 4& 1& 4& 4& 1& 
		2& 4& 4& 49& 4& 2& 4& 4& 4& 4& 
		1& 4 
		\\
		\hline			
		\mbox{cusp}&
	 	\tfrac{11}{63}& \tfrac{16}{91}& \tfrac{5}{28}& \tfrac{13}{70}& \tfrac{4}{21}& 
		\tfrac{10}{49}& \tfrac{13}{63}& \tfrac{22}{105}& \tfrac{3}{14}& \tfrac{11}{49}& 
		\tfrac{19}{84}& \tfrac{13}{56}& \tfrac{18}{77}& \tfrac{5}{21}& \tfrac{20}{77}& 
		\tfrac{51}{196}& \tfrac{15}{56}& \tfrac{2}{7}& \tfrac{25}{84}& \tfrac{23}{77}& 
		\tfrac{17}{56}& \tfrac{13}{42}
		\\
		\mbox{width}&
		4& 4& 1& 2& 4& 4& 4& 4& 2& 4&
		1& 1& 4& 4& 4& 1& 1& 4& 1& 4& 
		1& 2 
		\\
		\hline			
		\mbox{cusp}&
		\tfrac{11}{35}& \tfrac{9}{28}& \tfrac{19}{56}& \tfrac{31}{91}& \tfrac{67}{196}&
		\tfrac{12}{35}& \tfrac{29}{84}& \tfrac{69}{196}& \tfrac{5}{14}& \tfrac{75}{196}& 
		\tfrac{11}{28}& \tfrac{17}{42}& \tfrac{23}{56}& \tfrac{81}{196}& \tfrac{29}{70}& 
		\tfrac{3}{7}& \tfrac{43}{98}& \tfrac{25}{56}& \tfrac{45}{98}& \tfrac{13}{28}& 
		\tfrac{10}{21}& \tfrac{27}{56}
		\\
		\mbox{width}&
		4& 1& 1& 4& 1& 4& 1& 1& 2& 1& 
		1& 2& 1& 1& 2& 4& 2& 1& 2& 1& 
		4& 1 
		\\
		\hline			
		\mbox{cusp}&
		\tfrac{41}{84}& \tfrac{43}{84}& \tfrac{29}{56}& \tfrac{15}{28}& \tfrac{23}{42}& 
		\tfrac{4}{7}& \tfrac{37}{63}& \tfrac{29}{49}& \tfrac{25}{42}& \tfrac{17}{28}& 
		\tfrac{13}{21}& \tfrac{22}{35}& \tfrac{9}{14}& \tfrac{55}{84}& \tfrac{19}{28}& 
		\tfrac{29}{42}& \tfrac{39}{56}& \tfrac{59}{84}& \tfrac{5}{7}& \tfrac{37}{49}& 
		\tfrac{53}{70}& \tfrac{65}{84} 
		\\
		\mbox{width}&
		1& 1& 1& 1& 2& 4& 4& 4& 2& 1& 
		4& 4& 2& 1& 1& 2& 1& 1& 4& 4& 
		2& 1 
		\\
		\hline			
		\mbox{cusp}&
		\tfrac{11}{14}& \tfrac{23}{28}& \tfrac{6}{7}& \tfrac{61}{70}& \tfrac{25}{28}& 
		\tfrac{101}{112}& \tfrac{13}{14}& \tfrac{107}{112}& \tfrac{27}{28}& \infty&
		&&&&&&&&&&
		&
		\\
		\mbox{width}&
		2& 1& 4& 2& 1& 1& 2& 1& 1& 1&
		&&&&&&&&&&
		&
	\end{array}
$\end{table}

We note $\frac{LHS-RHS}{g_1}$ has no poles in $\mathcal{H}$, but it may have zeros in $\mathcal{H}$. 
We take a lower bound on the
orders at the non-infinite cusps by taking the minimum order of each of the 
individual summands, which we compute with
Propositions 
\ref{PropositionOrderForR2Tilde}, \ref{PropositionOrderForSTilde}, and \ref{PropProductOrders}.
This lower bound yields
\begin{align*}
	\sum_{\zeta\in\mathcal{D}} ORD_{\Gamma}\left( \frac{LHS-RHS}{g_1}; \zeta \right)
	\ge	
	ord\left( \frac{LHS-RHS}{g_1}   ,\infty\right)
	-522.
\end{align*}
However, we can expand $\frac{LHS-RHS}{g_1}$ as a series in $q$
and find the coefficients of $\frac{LHS-RHS}{g_1}$ are zero to at least 
$q^{523}$. Thus
\begin{align*}
	\sum_{\zeta\in\mathcal{D}} ORD_{\Gamma}\left(\frac{LHS-RHS}{g_1};\zeta\right)
	\ge	
	1,
\end{align*}
and so $\frac{LHS-RHS}{g_1}$ must be identically zero by the valence formula. 
This establishes Theorem \ref{TheoremDissection}.

\bibliographystyle{abbrv}
\bibliography{M2Ref}

\end{document}